%% file: ncVEM_compressible_miscibledisplacement.tex
\documentclass[a4paper,12pt]{article}

\usepackage{amsmath,amsthm,amssymb,enumerate}
\usepackage{xcolor}
\usepackage{float}
\usepackage{multirow}
\usepackage{longtable}
\usepackage{diagbox}
\usepackage{soul}
\usepackage{cancel}
\usepackage{gensymb}
\usepackage{enumitem}
\usepackage[square,sort&compress,comma,numbers]{natbib}
\usepackage{bm}
\usepackage{hyperref}
\usepackage{subfig}
\usepackage{graphicx}
\usepackage[margin=1.7cm]{geometry}
\usepackage{tikz}
\usepackage{esint}
\usepackage{caption}
\usetikzlibrary{shapes,calc}
\usepackage{verbatim}
\usepackage{array}
\usepackage{bm}
\definecolor{maroon}{RGB}{144,0,32}
\newcolumntype{H}{>{\setbox0=\hbox\bgroup}c<{\egroup}@{}}

\usepackage{newtxtext,newtxmath}

\usepackage{multirow}
\usepackage{marginnote}
\chardef\bslash=`\\ 

\newtheorem{thm}{Theorem}[section]
\newtheorem{cor}[thm]{Corollary}
\newtheorem{lem}[thm]{Lemma}

\theoremstyle{definition}
\theoremstyle{remark}

\numberwithin{equation}{section}

\newcommand{\bI}{\boldsymbol{I}}

\newcommand{\cA}{\mathcal A}
\newcommand{\cK}{\mathcal K}

\newcommand{\cD}{\mathcal D}
\newcommand{\cE}{\mathcal E}

\newcommand{\cT}{\mathcal{T}}

\newcommand{\cN}{\mathcal N}

\newcommand{\cM}{{\mathcal M}}
\newcommand{\err}{{\rm err}}

\newcommand{\hc}{\widehat{c}}

\newcommand{\divc}{\mathrm{div}}
\newcommand{\rot}{\mathrm{rot\,}}
\newcommand{\dof}{\mathrm{dof}}

\newcommand{\pw}{\rm pw}

\newcommand{\fl}{\,\, \text{for all}\:}

\newcommand{\half}{\frac{1}{2}}

\newcommand{\dx}{{\rm\,dx}}
\newcommand{\ds}{{\rm\,ds}}

\newcommand{\Holder}{H\"{o}lder~}

{\algorithm}%
{\endalgorithm}

\input{command_vem}

\newcommand{\balpha}{{\boldsymbol{\alpha}}}
\newcommand{\bbeta}{{\boldsymbol{\beta}}}

\newcommand{\bx}{{\boldsymbol{x}}}
\newcommand{\bn}{{\boldsymbol{n}}}
\newcommand{\bu}{{\boldsymbol{u}}}
\newcommand{\bp}{{\boldsymbol{p}}}
\newcommand{\bsf}{{\boldsymbol{f}}}
\newcommand{\bPi}{{\boldsymbol{\Pi}}}
\newcommand{\bpsi}{{\boldsymbol{\psi}}}

\def \R{{{\Bbb R}}}
\def \P{{{\mathcal P}}}

\allowdisplaybreaks

\def\R{\mathbb{R}}

\def\cA{\mathcal{A}}

\def\O{\Omega}

\def\bv{{\boldsymbol v}}

\def\pw{\rm {pw}}

\def\cE{{\mathcal{E}}}

\def\jump#1{\left[\hskip -3.5pt\left[#1\right]\hskip -3.5pt\right]}
\def\bchi{\boldsymbol{\chi}}

\def\bg{\boldsymbol \gamma}

\newcommand{\be}{\begin{equation}}
\newcommand{\ee}{\end{equation}}

\usepackage{mathtools}  
\mathtoolsset{showonlyrefs} 
\usepackage[normalem]{ulem}
\definecolor{violet}{rgb}{0.580,0.,0.827}

\usepackage[normalem]{ulem}
\newcounter{corr}
\definecolor{violet}{rgb}{0.580,0.,0.827}
\newcommand{\corr}[3]{\typeout{Warning : a correction remains in page
		\thepage}
	\stepcounter{corr}        
	{\color{red}\ifmmode\text{\,{\ensuremath{#1}}\,}\else{#1}\fi}
	{\color{blue}#2}
	{\color{violet} #3}}

\newcounter{changeto}
\newcommand{\changeto}[2]{\typeout{Warning : a correction remains in page
		\thepage}
	\stepcounter{changeto}        
	{\color{blue}\ifmmode\text{\,\sout{\ensuremath{#1}}\,}\else\sout{#1}\fi}
	{\color{red}#2}}

\title{Convergence analysis of a nonconforming virtual element method for compressible miscible displacement problems in porous media}
\author{Sarvesh Kumar \footnote{Department of Mathematics, Indian Institute of Space Science and Technology, Thiruvanathapuram 695547, India. sarvesh@iist.ac.in} \quad Devika Shylaja \footnote{Department of Mathematics, Indian Institute of Space Science and Technology, Thiruvanathapuram 695547, India. devikas.pdf@iist.ac.in}}

\date{}
\begin{document}
	\maketitle
	
\abstract This article presents a priori error estimates of the miscible displacement of one compressible fluid by another in a porous medium. The study utilizes the $H(\divc)$ conforming virtual element method (VEM) for the approximation of the velocity, while a non-conforming virtual element approach is employed for the concentration. The pressure is discretised using the standard piecewise discontinuous polynomial functions. These spatial discretization techniques are combined with a backward Euler difference scheme for time discretization. Error estimates are established for velocity, pressure and concentration. The article also includes numerical results that validate the theoretical estimates.



	\section{Introduction}
This paper analyses the virtual element method for the miscible displacement of one compressible fluid by another in a porous medium which is described by nonlinear coupled system of two parabolic partial differential equations. Let $\Omega\subset \R^2$ be a polygonal bounded, convex domain, describing a reservoir of unit thickness. Given a time interval $J := [0,T],$ for $T > 0$, the problem is to find the Darcy velocity $\bu=\bu(\bx,t)$ of the fluid mixture, the pressure $p = p(\bx,t)$ in the fluid mixture, and the concentration $c=c(\bx,t)$ of one of the components in the mixture, with $(\bx,t)=\Omega_T:=\Omega \times J$ such that \cite{DJR_Compressible1983,ERW_Compressible1984}
\begin{subequations}\label{eqn.model}
\begin{align}
&d(c)\frac{\partial p}{\partial t}+\nabla \cdot\bu=q,\quad \bu=-a(c)\nabla p,\label{eqn.modela}\\
&\phi\frac{\partial c}{\partial t}+b(c)\frac{\partial p}{\partial t}+\bu\cdot \nabla c -\divc(D(\bu)\nabla c)=q(\hc-c),\label{eqn.modelb}
\end{align}
\end{subequations}\noeqref{eqn.modela,eqn.modelb}
where $\phi=\phi(\bx)$ is the porosity of the medium and $a(c)=a(c,\bx)$ is the scalar-valued function given by
\[a(c):=\frac{k}{\mu(c)}.\]
Here, $k=k(\bx)$ represents the permeability of the porous rock, and $\mu(c)$ is the viscosity of the fluid mixture. Further, $q(\bx,t)$ is a sum of sources (injection) and sinks (extraction), and $\hc=\hc(\bx,t)$ is the concentration of the external flow, which is specified at sources and is assumed to be equal to $c(\bx,t)$ at sinks. We assume that the flow rate $q$ is smoothly distributed so that the problem has smooth solutions. Moreover, the diffusion dispersion tensor $D(\bu) \in \R^{2 \times 2}$ is given by
\begin{equation}\label{defn.D}
D(\bu):=\phi[d_m\bI +|u|\left(d_\ell E(\bu)+d_t(\bI-E(\bu))\right)]
\end{equation}
where $d_m$ is the molecular diffusion coefficient, $d_\ell$ (resp. $d_t$) is the longitudinal (resp. transversal) dispersion coefficient, $\bI$ is the identity matrix of order 2, and $E(\bu)$ is the tensor that projects onto $\bu$ direction which is given by, for $\bu=(u_1,u_2)$,
\begin{align*}
E(\bu)&:=\frac{\bu\bu^T}{|\bu|^2},\quad |\bu|^2=u_1^2+u_2^2.
\end{align*}
The coefficients $d(c)$ and $b(c)$ are defined as
\begin{equation}\label{eqn.db}
d(c)=\phi\sum_{i=1}^2 z_ic_i, \quad b(c)=\phi c_1\displaystyle \left(z_1-\sum_{i=1}^2 z_ic_i\right),
\end{equation}
where $c=c_1=1-c_2$ , $c_i$ denotes the (volumetric) concentration of the $i$th component of the fluid mixture and $z_i$ is the constant compressibility factor for the $i$th component, $i=1,2$.

\medskip

\noindent Assume that no flow occurs across the boundary $\partial \Omega$, that is,
\begin{align}
&\bu\cdot \bn=0, \quad D(\bu)\nabla c\cdot \bn=0 \quad \mbox{ on } \partial \Omega \times J,\label{eqn.bc}
\end{align}
where $\bn$ denotes the outward unit normal to the boundary $\partial \Omega$ and the initial conditions
\begin{equation}\label{eqn.ic}
c(\bx,0)=c_0(\bx), \quad p(\bx,0)=p_0(\bx) \mbox{ in } \Omega.
\end{equation}

\smallskip

\noindent The study of miscible displacement is crucial for optimizing various industrial processes, such as enhanced oil recovery in the petroleum industry, seawater intrusion or contaminant transport in environmental remediation \cite{RussellWheeler,peaceman-77}. In \cite{DJR_Compressible1983}, Galerkin methods and mixed finite element methods were analyzed for the approximation of the compressible miscible displacement problem \eqref{eqn.model} with continuous time. Bound-preserving discontinuous Galerkin methods were explored in \cite{compressible_dG_2017,compressible_dG_2022}. A characteristic block-centered finite-difference method for this problem was derived in \cite{compressible_characteristics_2022,compressible_characteristics_2017}. Furthermore, a mixed-finite-element method combined with a mass-conservative characteristic finite method was proposed in \cite{compressible_characteristicsmixedfem_2019,compressiblemiscible_Zhang_2020} to address this challenge. Moreover, a mixed finite-element method, along with a two-grid method, was demonstrated in \cite{compressible_twogrid_2018,compressible_twogrid_2022} for solving the compressible miscible displacement problem.
\smallskip

\noindent Virtual element method (VEM) \cite{Veiga_basicVEM}, which is a generalization of the FEM,  has got more and more attention in recent years, because it can deal with the polygonal meshes and avoid an explicit construction of the discrete shape function, see \cite{Veiga_HdivHcurlVEM,Veiga_hitchhikersVEM,Brenner_errorVEM,VEM_general_2016,MixedVEM_general_2016,Cangiani_2017} and the references therein. The polytopal meshes can be very useful for a wide range of reasons, including meshing of the domain (such as cracks) and data features, automatic use of hanging nodes, adaptivity. Recently, virtual element methods for incompressible miscible displacement problem have been investigated in \cite{Veiga_miscibledisplacement_2021,SD_ncVEM_incompressible_2024}. In \cite{Veiga_miscibledisplacement_2021} (resp. \cite{SD_ncVEM_incompressible_2024}), conforming (resp. non-conforming) VEM is analysed for the concentration and mixed VEM for the velocity-pressure equations.

\smallskip

\noindent This paper employs the H(div) conforming VEM for approximation of the velocity, while the concentration is handled using a non-conforming virtual element approach. To discretize the pressure, standard piecewise discontinuous polynomial functions are used. These spatial discretizations are then combined with an uncomplicated time discretization using a backward Euler method, known for its computational efficiency. Optimal a priori error estimates are established for the concentration, pressure, and velocity in $L^2$ norm under regularity assumption on the exact solution. Numerical results are presented to justify the theoretical estimates.

\smallskip

\noindent The remaining parts are organised as follows. Section \ref{sec:wf} discusses the weak formulation of \eqref{eqn.model}-\eqref{eqn.ic}. The main result of this paper (Theorem~\ref{thm.main}) is stated at the end of this section. Section~\ref{sec:vem} deals with the virtual element method, semi-discrete and fully-discrete formulations. Error estimates for velocity, pressure and concentration are established in Section~\ref{sec:errors}. Section~\ref{sec:numericalresults} provides the results of computational experiments that validate the theoretical estimates on both an ideal test case and a more realistic test case.

\smallskip

\noindent {\textbf{Notation.}} The standard $L^2$ inner product and norm on $L^2(\O)$ are denoted by $(\cdot,\cdot)$ and $\|{\cdot}\|$. The semi-norm and norm in $W^{k,p}(D)$,  for $D \subseteq \Omega$ and $1 \le p \le \infty$, are denoted by $|\bullet|_{k,p,D}$ and $\|\bullet\|_{k,p,D}$. For $p=2$,  the semi-norm and norm are denoted by $|\bullet|_{k,D}$ and $\|\bullet\|_{k,D}$. Let $\P_k(D)$ denotes the space of polynomials of degree at most $k$ ($k \in \N_0$) with the usual convention that $\P_{-1}(D)=0$.

\smallskip

 \noindent Let $H(\divc,\Omega) $ denotes the Sobolev space
\[H(\divc,\Omega):=\{\bv \in (L^2(\Omega))^2: \divc \bv \in L^2(\O)\}.\] Define the velocity space $\bV$, the pressure space $Q$, and the concentration space $Z$, equipped with the following norms by
\begin{align}\label{defn.spaces}
	&\bV:=\{\bv \in H(\divc,\Omega):\bv \cdot \bn=0 \mbox{ on }\partial \Omega\},\quad \|\bu\|_\bV^2:=\|\bu\|^2+\|\divc \bu\|^2\\
	&Q:=L^2(\Omega), \quad \|q\|_Q^2:=\|q\|^2\\
	&Z:=H^1(\Omega), \quad \|z\|_Z^2:=\|z\|^2+\|\nabla z\|^2.
\end{align}
For $0 \le a \le b$,
\[\|\bv\|_{L^2(a,b;\bV)}^2:=\int_a^b\|\bv(t)\|_\bV^2 \dx, \quad \|\bv\|_{L^\infty(a,b;\bV)}:=\mbox{ess} \sup_{t \in [a,b]}\|\bv(t)\|_\bV.\]

\noindent For simplicity of notation, we denote $W^{m,p}(0,T;W^{r,q}(\Omega))$ by $W^{m,p}(W^{r,q})$. The regularity assumptions of the solution functions $c,p$ and $\bu$ are denoted collectively by
{\[(R)\begin{cases}
	&c \in L^\infty(H^{k+2})\cap H^1(H^{k+2})\cap L^\infty(W^{k+1,\infty})\cap H^2(L^2),\\
	&p \in L^\infty(H^{k+1})\cap H^2(H^{k+1}) \cap H^3(L^2),\\
	&\bu \in L^\infty(H^{k+1})\cap L^\infty(W^{1,\infty})\cap W^{1,\infty}(L^\infty)\cap H^2(L^2),
\end{cases}\]}
where $k\ge 0$ is the order of accuracy of $p$. 

\smallskip

\noindent For all $s>0$, define the broken Sobolev space as
\[H^s(\cT_h):=\{v \in L^2(\O);\,v_{|K} \in H^s(K) \fl K \in \cT_h\},\]
with the corresponding broken semi-norms and norms
\[|v|_{s,\cT_h}^2:=\sum_{K \in \cT_h}|v|_{s,K}^2, \quad \|v\|_{s,\cT_h}^2:=\sum_{K \in \cT_h}\|v\|_{s,K}^2.\]
\section{Weak formulation}\label{sec:wf}
This section deals with the weak formulation of the continuous problem \eqref{eqn.model}-\eqref{eqn.ic} and the properties of the associated bilinear forms.

\medskip

\noindent Assume that
\begin{align*}
	&a_* \le a(x,c) \le a^*, \, b_* \le b(x,c) \le b^*,\, d_* \le d(x,c) \le d^*,\\
	&\phi_* \le \phi(x) \le \phi^*,\, D_* \le D(\bu) \le D^*,\\
	&\left|\frac{\partial a}{\partial c}(x,c)\right|+\left|\frac{\partial^2 a}{\partial c^2}(x,c)\right|+\left|\frac{\partial b}{\partial c}(x,c)\right|+\left|\frac{\partial d}{\partial c}(x,c)\right| \le K_1,
\end{align*}
where $a_*,a^*,b_*,b^*,d_*,d^*,\phi_*,\phi^*, D_*,D^*$, and $K_1$ are positive constants. In order to simplify the presentation, we define
\[A(z)(\bx):=a^{-1}(z,\bx).\]
Additionally, assume the realistic relation of the diffusion and dispersion coefficients given by $0 <d_m \le d_t\le d_\ell.$

\medskip

\noindent The weak formulation of \eqref{eqn.model}-\eqref{eqn.ic} seeeks $c \in L^2(Z)\cap C^0(L^2)$, $\bu \in L^2(\bV)$, and $p \in L^2(Q)\cap C^0(L^2)$, such that
\begin{subequations}\label{eqn.weak}
\begin{align}
\cA(c(t);\bu(t),\bv)+B(\bv,p(t))&=0, \fl \bv \in \bV \label{eqn.weakv}\\
W\left(c(t);\frac{\partial p(t)}{\partial t},w\right)-B(\bu(t),w)&=(q(t),w), \fl w \in Q \label{eqn.weakq}\\
\cM\left(\frac{\partial c(t)}{\partial t},z\right)+K\left(c(t);\frac{\partial p(t)}{\partial t},z\right)+&\Theta(\bu(t),c(t);z)+\cD(\bu(t);c(t),z)+(qc(t),z)\nonumber \\
&=(q(\hc)(t),z), \fl z \in Z\label{eqn.weakz}
\end{align}
\end{subequations}
for almost all $t \in J$ and with initial conditions $c(0)=c_0$ and $p(0)=p_0$, where
\begin{subequations}\label{defn.bilinear}
\begin{align}
& \cA(c;\bu,\bv):=(A(c)\bu,\bv), \quad  B(\bv,w):=-(\divc \bv,w),\label{defn.bilinear.a}\\
&W(c;p,w)=(d(c)p,w),\quad \cM(c,z):=(\phi c,z), \quad  K(c;p,z)=(b(c)p,z),\label{defn.bilinear.b}\\
& \Theta(\bu,c;z)=(\bu\cdot \nabla c,z),\quad \cD(u;c,z):=(D(\bu)\nabla c, \nabla z).\label{defn.bilinear.c}
\end{align}
\end{subequations}\noeqref{defn.bilinear.a,defn.bilinear.b,defn.bilinear.c}
For the sake of readability, here and throughout this paper, we write $\bu$ for $\bu(t)$ and for the other functions depending on space and time. The interpretation of whether $\bu$ represents a function of space only or a function of both space and time should be inferred from the surrounding context. 

\medskip

\noindent The kernel is defined as
\begin{equation}\label{defn.kernel}
\cK:=\{\bv \in \bV:B(\bv,w)=0 \fl w \in Q\}.
\end{equation}

\begin{lem}[Properties of the bilinear forms]\label{lem.propertiescontinuous}The following properties hold for the bilinear forms in \eqref{defn.bilinear}\cite[Section 2.2]{Veiga_miscibledisplacement_2021}:
	\begin{itemize}
		\item[$(a)$] $\cM(c,z) \le \phi^*\|c\|\|z\|$ for all $c,z \in Z$,
		\item[(b)] $\cM(z,z) \ge \phi_*\|z\|^2$ for all $z \in Z$,
		\item[(c)] $\cA(c;\bu,\bv) \le \frac{1}{a_*}\|\bu\|\|\bv\|$ for all $c \in L^\infty(\Omega)$ and $\bu,\bv \in (L^2(\O))^2$,
		\item[(d)] $\cA(c;\bv,\bv) \le \|A(c)\|\|\bu\|_{0,\infty,\O}\|\bv\|$ for all $ c\in L^2(\O)$, $\bu \in (L^\infty(\O))^2$, and $\bv \in (L^2(\O))^2$,
		\item[(e)] $\cA(c;\bv,\bv) \ge \frac{1}{a^*}\|\bv\|^2$ for all $ c \in L^\infty(\O)$ and $\bv \in (L^2(\O))^2$,
		\item[(f)] $\cA(c,\bv,\bv) \ge \frac{1}{a^*}\|\bv\|_\bV^2$ for all $ c \in L^\infty(\O)$ and $\bv \in \cK$,
		\item[(g)] $\cD(\bu;c,z) \le \phi^*[d_m+\|u\|_{0,\infty,\O}(d_\ell+d_t)]\|\nabla c\|\|\nabla z\|$ for all $\bu \in (L^\infty(\O))^2$ and $c,z \in H^1(\O)$,
		\item[(h)] $\cD(\bu;c,z) \le\eta_\cD(1+\|u\|)\|\nabla c\|_{0,\infty,\O}\|\nabla z\|$ for all $\bu \in (L^2(\O))^2$ and $c,z \in H^1(\O)$ with $\nabla c \in L^\infty(\O)$,
		\item[(i)] $(D\bu,\boldsymbol{\mu},\boldsymbol{\mu}) \ge \phi_*(d_m\|\boldsymbol{\mu}\|^2+d_t\||u|^\half\boldsymbol{\mu}\|^2)$ for all $\boldsymbol{\mu} \in (L^2(\O))^2$,
	\end{itemize}
where $\eta_\cD$ is a positive constant depending only on $d_m$, $d_\ell$, and $d_t$.
\end{lem}
\noindent The definition of $W(c;\bullet,\bullet)$, the boundedness of $d(c)$, and the Cauchy-Schwarz inequality shows, for all $p,w \in Q$,
\[W(c;p,w) \le d^*\|p\|\|w\| \,\mbox{ and }\, W(c;w,w) \ge d_*\|w\|^2.\]
Analogously,
\[K(c;p,z) \le b^*\|p\|\|z\|\, \mbox{ and }\, K(c;z,z) \ge b_*\|z\|^2.\] 
\subsection{Main result}
The main result of this paper is briefly presented below. 

\medskip

\noindent  Let $0=t_0<t_1<\cdots<t_N=T$ be a given partition of $J=[0,T]$ with time step size $\tau$ and let $h$ be the mesh-size. Let $(\bu,p,c) $ solves \eqref{eqn.weak} at time $t=t_n$. For $k\ge 0$, let $\bu_h$ be the $H(\divc)$ conforming VEM approximation to $\bu$ of order $k$, $p_h$ be the polynomial approximation to $p$ of order $k$, and $c_h$ be the non-conforming VEM approximation to $c$ of order $k+1$. Suppose the space and time discretisation satisfy $\tau=\mathcal{O}(h^{k+1})$. 

\begin{thm}\label{thm.main}
		Under the regularity assumptions $(R)$, there exists a constant $C>0$, independent of $h$ and $\tau$, such that the following hold: for $n=1,\cdots,N,$
	\begin{align}
		&	\max_{n}\|\bu(t_n)-\bu_h(t_n)\|+\max_n \|p(t_n)-p_h(t_n)\|+\max_n \|c(t_n)-c_h(t_n)\| \le C(h^{k+1}+\tau).
	\end{align}
\end{thm}
\noindent The proof of this theorem is presented in Section~\ref{sec:errors}.
\section{The virtual element method}\label{sec:vem}
This section presents the virtual element method for the weak formulation \eqref{eqn.weak}.

\smallskip

\noindent Let $\cT_h$ be a discretisation of $\O$ into polygons $K$. Let $\cE_h$ denote the set of all edges of $\cT_h$, and let $\cE_h^K$ be the set of all edges of $K \in \cT_h$. Let $h_K$ be the diameter of $K$ and mesh-size $h:=\max_{K \in \cT_h} h_K$. Let $ h_e$ be the length of the edge $e$, and $n_K$ be the number of edges of $K$. Assume that there exists a $\rho_0>0$ such that for all $h>0$ and for all $K \in \cT_h$:\\

\noindent (\textbf{D1}) $K$ is star-shaped with respect to a ball of radius $\rho \ge \rho_0h_K$,\\
(\textbf{D2}) $h_e \ge \rho_0h_K$ for all $e \in \cE_h^K$.

\medskip

\noindent Note that these two assumptions imply that the number of edges of each element is uniformly bounded. Additionally,
we will require following quasi-uniformity:\\
(\textbf{D3}) for all $h>0$ and for all $K \in \cT_h$, it holds $h_K \ge \rho_1 h$, for some positive uniform constant $\rho_1$.

\subsection{Discrete spaces}
Let $K \in \cT_h$ and let $k \in \N_0$ be a given degree of accuracy. Then the local velocity virtual element space \cite{Veiga_HdivHcurlVEM,MixedVEM_general_2016} is defined by
\begin{align}\label{defn.Vhk}
\bV_h(K):&=\{\bv \in H(\divc;K)\cap H(\rot;K):\, \bv\cdot\bn_{|e} \in \P_k(e) \fl e \in \cE_h^K,\\
&\qquad \quad \divc \bv \in \P_k(K),\, \rot \bv \in \P_{k-1}(K)\}.
\end{align}
Obviously, $(\P_k(K))^2 \subseteq \bV_h(K)$. The degrees of freedom $\{\dof_j^{\bV_h(K)}\}_{j=1}^{\dim \bV_h(K)}$ on $\bV_h(K)$ are
\begin{enumerate}
	\item $\displaystyle \frac{1}{|e|}\int_e \bv \cdot \bn p_k \ds \fl p_k \in \P_k(e) \fl e \in \cE_h^K$\\
	\item $\displaystyle \frac{1}{\sqrt{|K|}}\int_K \divc \bv p_k \dx \fl p_k \in \P_k(K)\setminus \R$\\
	\item $\displaystyle \frac{1}{|K|}\int_K \bv \cdot \bx^\perp p_{k-1} \dx \fl p_{k-1} \in \P_{k-1}(K)$,	
\end{enumerate}
with $\bx^\perp:=(x_2,-x_1)^T$, where we assume the coordinates to be centered at the barycenter of the element.

\medskip

\noindent The local pressure virtual element space \cite{Veiga_HdivHcurlVEM,MixedVEM_general_2016} is
\begin{align}\label{defn.Vhk1}
	Q_h(K):&=\{q \in L^2(K):\, q \in \P_k(K)\}.
\end{align}
Observe that $\P_k(K) \subseteq Q_h(K)$. The degrees of freedom $\{\dof_j^{Q_h(K)}\}_{j=1}^{\dim Q_h(K)}$ on $Q_h(K)$ are
\begin{enumerate}
	\item $\displaystyle \frac{1}{{|K|}}\int_K q p_k \dx \fl p_k \in \P_k(K).$
\end{enumerate}

\medskip

\noindent These two spaces are coupled with the preliminary local concentration spaces \cite{ncVEM_2016}
\begin{align}\label{defn.ZhKtilde}
	\widetilde{Z}_h(K):&=\displaystyle \{z \in H^1(K):\,\frac{\partial z}{\partial n} \in \P_k(e) \fl e \in \cE_h^K, \Delta z \in \P_{k-1}(K)\}
\end{align}
It is clear from \eqref{defn.ZhKtilde} that $\P_{k+1}(K) \subseteq \widetilde{Z}_h(K)$. The degrees of freedom $\{\dof_j^{\widetilde{Z}_h(K)}\}_{j=1}^{\dim \widetilde{Z}_h(K)}$ on $\widetilde{Z}_h(K)$ is defined by
\begin{enumerate}
	\item $\displaystyle \frac{1}{|e|}\int_e zp_k \ds \fl p_k \in \P_k(e) \fl e \in \cE_h^K$\\
	\item $\displaystyle \frac{1}{{|K|}}\int_K zp_{k-1} \dx \fl p_{k-1} \in \P_{k-1}(K).$
\end{enumerate}

\noindent Note that $k=0$ provides the lowest order local VE spaces. Let $\bPi_k^{0,K}:(L^2(K))^2 \to (\P_k(K))^2$ be the $L^2$ projector onto the vector-valued polynomials of degree atmost $k$ in each component. That is, for a given $\bsf \in (L^2(\O))^2$,
\begin{equation}\label{defn.Pi}
(\bPi_k^{0,K}\bsf,\bp_k)=(\bsf,\bp_k) \fl \bp_k \in (\P_k(K))^2.
\end{equation}
This operator is computable for functions in $\bV_h(K)$ only by knowing their values at the degrees of freedom. Also, an integration by parts leads to, for $z_h \in \widetilde{Z}_h(K)$,
\[\int_K \bPi_k^{0,K}\nabla z_h \cdot \bp_k \dx=\int_K \nabla z_h \cdot \bp_k \dx=-\int_K z_h \divc \bp_k \dx +\int_{\partial K}z_h\bp_k\cdot \bn \ds,\]
for all $\bp_k \in (\P_k(K))^2$. The right-hand side is computable using the degrees of freedom of $\widetilde{Z}_h(K)$ and so is the left-hand side.

\smallskip

\noindent In addition to the $L^2$ projector described in \eqref{defn.Pi}, one needs the elliptic projector $\Pi_{k+1}^{\nabla,K}:H^1(K) \to \P_{k+1}(K)$, which is defined as follows:
\begin{subequations}\label{defn.Pielliptic}
\begin{align}
(\nabla \Pi_{k+1}^{\nabla,K}z, \nabla p_{k+1})&=(\nabla z,\nabla p_{k+1}) \fl p_{k+1} \in \P_{k+1}(K) \label{eqn.Pielliptic.a}\\
\frac{1}{|\partial K|}\int_{\partial K} \Pi_{k+1}^{\nabla,K}z&=\frac{1}{|\partial K|}\int_{\partial K} z \mbox{ for } k=0\label{eqn.Pielliptic.k0}\\
\frac{1}{|K|}\int_{ K} \Pi_{k+1}^{\nabla,K}z&=\frac{1}{| K|}\int_{ K} z \mbox{ for } k\ge 1\label{eqn.Pielliptic.k1}
\end{align}
\end{subequations}
for all $z \in H^1(K)$. Note that $\Pi_{k+1}^{\nabla,K}p_{k+1}=p_{k+1}$ for all $p_{k+1} \in \P_{k+1}(K)$. For any $z \in \widetilde{Z}_h(K)$, $\Pi_{k+1}^{\nabla,K}z$ can be computed using integration by parts and the degrees of freedom of $\widetilde{Z}_h(K)$.

\smallskip

\noindent Since the projections in the $L^2$ norm are available only on polynomials of degree $\le k-1$ directly from the degrees of freedom of $\widetilde{Z}_h(K)$, in order to compute the $L^2$ projections on $\P_{k+1}(K)$, we consider a modified virtual element space \cite{Equivalentprojectors_2013} for the concentration. Define
\begin{align}\label{defn.ZhK}
{Z}_h(K):&=\displaystyle \{z \in H^1(K):\,\frac{\partial z}{\partial n} \in \P_k(e) \fl e \in \cE_h^K, \Delta z \in P_{k+1}(K),\\
&\qquad  \int_K \Pi_{k+1}^{\nabla,K}z p_{k+1}\dx= \int_K z p_{k+1}\dx \fl p_{k+1} \in \P_{k+1}(K)\setminus\P_{k-1}(K)\},
\end{align}
where $\P_{k+1}(K)\setminus\P_{k-1}(K)$ is the space of polynomials in $\P_{k+1}(K)$ which are $L^2(K)$ orthogonal to $\P_{k-1}(K)$. It can be shown that the space $Z_h(K)$ has the same degrees of freedom and the same dimension as $\widetilde{Z}_h(K)$ \cite{Equivalentprojectors_2013,ncVEM_Stokes}.

\smallskip
\noindent In the sequel, the notation ``$a\lesssim b$ (resp. $ a \gtrsim b$)'' means that there exists a generic constant $C$ independent of the mesh parameter $h$ and time step size $\tau$ such that $a \le Cb$ (resp. $a \ge C b$). The approximation properties for the projectors are stated next \cite[Lemma 5.1]{VEM_general_2016}, \cite[Lemma 3.1]{Veiga_miscibledisplacement_2021}.

\begin{lem}[Approximation properties]\label{lem.approx}Given $K \in \cT_h$, let $\psi$ and $\bpsi$ be sufficiently smooth scalar and vector-valued functions, respectively. Then, it holds, for all $k \in \N_0$,
	\begin{align}
&(a)\,	\|\psi-\Pi_k^{0,K}\psi\|_{\ell,K} \lesssim h_K^{s-\ell} |\psi|_{s,K}, \quad 0 \le \ell \le s\le k+1\\
&(b)\,	\|\bpsi-\bPi_k^{0,K}\bpsi\|_{\ell,K} \lesssim h_K^{s-\ell} |\bpsi|_{s,K}, \quad 0 \le \ell \le s\le k+1\\
&(c)\,	\|\psi-\Pi_k^{\nabla,K}\psi\|_{\ell,K} \lesssim h_K^{s-\ell} |\psi|_{s,K}, \quad 0 \le \ell \le s\le k+1, \, s\ge 1.
	\end{align}
\end{lem}

\noindent For every decomposition $\cT_h$ of $\Omega$ into simple polygons $K$, define the global spaces by
\begin{align*}
\bV_h:&=\{\bv \in \bV:\bv_{|K} \in \bV_h(K) \fl K \in \cT_h\}\\
Q_h:&=\{q \in Q:q_{|K} \in Q_h(K) \fl K \in \cT_h\}\\
Z_h:&=\{z \in H^{1,\rm{nc}}(\cT_h;k):z_{|K} \in Z_h(K) \fl K \in \cT_h\},
\end{align*}
where
\[H^{1,\rm{nc}}(\cT_h;k):=\{z \in H^1(\cT_h):\int_e \jump{z}\cdot\bn q\ds=0 \fl q \in P_k(e) \fl e \in \cE_h\}\]
Define, for all $\bu_h \in \bV_h$,
\[\|\bu_h\|_{\bV_h}^2:=\sum_{K \in \cT_h}\|\bu_h\|_{V,K}^2:=\sum_{K \in \cT_h}[\|\bu_h\|^2+\|\divc \bu_h\|^2].\]
Let the $\Pi_k^0$, $\bPi_k^0$, and $\Pi_{k+1}^\nabla$ be the global projectors such that for all $K \in \cT_h$,
\[\Pi^0_{{k}_{|K}}=\Pi_k^{0,K},\,\, \bPi^0_{{k}_{|K}}=\bPi_k^{0,K},\,\, \Pi^\nabla_{{k+1}_{|K}}=\Pi_{k+1}^{\nabla,K}.\]
The sets of global degrees of freedom $\{\dof_j^{\bV_h}\}_{j=1}^{\dim \bV_h}$,$\{\dof_j^{Q_h}\}_{j=1}^{\dim Q_h}$ and $\{\dof_j^{Z_h}\}_{j=1}^{\dim Z_h}$ are achieved by linking together their corresponding local counterparts.

\subsection{Semi-discrete formulation}\label{sec.semidiscrete}
This section deals with the semi-discrete weak formulation of \eqref{eqn.weak} which is continuous in time and discrete in space. Let $\Pi_0^{0,K}:L^2(K)\to \P_0(K)$ and $\bPi_0^{0,K}:(L^2(K))^2\to (\P_0(K))^2$ be the $L^2$ projectors onto the scalar and vector valued functions.

\noindent The semi-discrete variational formulation of \eqref{eqn.weak} seeks $\bu_h \in \bV_h$, $p_h \in Q_h$, and $c_h \in Z_h$ such that, for almost every $t \in J$,
\begin{subequations}\label{eqn.semidiscrete}
\begin{align}
\cA_h(c_h;\bu_h,\bv_h)+B(\bv_h,p_h)&=0, \fl \bv_h \in \bV_h \label{eqn.semidiscreteqv}\\
W_h\left(c_h;\frac{\partial p_h}{\partial t},w_h\right)-B(\bu_h,w_h)&=(q,w_h), \fl w_h \in Q_h\label{eqn.semidiscreteqq}\\
\cM_h\left(\frac{\partial c_h}{\partial t},z_h\right)+K_h\left(c_h;\frac{\partial p_h}{\partial t},z_h\right)+&\Theta_h(\bu_h,c_h;z_h)+\cD_h(\bu_h;c_h,z_h)+(qc_h,z_h)_h\nonumber\\
&=(q\hc,z_h)_h, \fl z_h \in Z_h\label{eqn.semidiscreteqz}\\
\end{align}
\end{subequations}
with the initial condition 
\[c_h(0)=c_{0,h} \mbox{ and } p_h(0)=p_{0,h},\]
where $c_{0,h}$ (resp. $p_{0,h}$) denotes the interpolant of $c$ (resp. $p$) , which will be defined subsequently. Here,
\begin{enumerate}
	\item The term $\cA_h(\bullet;\bullet,\bullet)$ in \eqref{eqn.semidiscreteqv} is defined by
	\begin{align}
		\cA_h\left(c_h;\bu_h,\bv_h\right)&:=\sum_{K\in \cT_h} \cA_h^K\left(c_h;\bu_h,\bv_h\right),
	\end{align}
	where 
	\begin{align}
		\cA_h^K\left(c_h;\bu_h,\bv_h\right):=&\int_K A(\Pi_{k+1}^{0,K}c_h) \bPi_{k}^{0,K}(\bu_h)\cdot\bPi_{k}^{0,K}(\bv_h)\dx\\
		&\qquad + \nu_\cA^K(c_h) S_\cA^K((I-\bPi_{k}^{0,K})\bu_h,(I-\bPi_{k}^{0,K})\bv_h)
	\end{align}
	with $S_\cA^K(\bullet,\bullet)$ denotes the stabilization term with property given below in \eqref{eqn.SA} and 
	\[\nu_A^K(c_h)=|A(\Pi_{0}^{0,K}c_h)|.\]
	\item The term $W_h(\bullet;\bullet,\bullet)$ in \eqref{eqn.semidiscreteqq} is defined by
	\begin{align}
		W_h\left(c_h;\frac{\partial p_h}{\partial t},w_h\right)=\sum_{K\in \cT_h}\int_K d(\Pi_{k+1}^{0,K}c_h)\frac{\partial p_h}{\partial t}w_h \dx.
	\end{align}
	\item The term $\cM_h(\bullet,\bullet)$ in \eqref{eqn.semidiscreteqz} is defined by
\begin{align}
\cM_h\left(\frac{\partial c_h}{\partial t},z_h\right)&:=\sum_{K\in \cT_h} \cM_h^K\left(\frac{\partial c_h}{\partial t},z_h\right), \label{defn.Mh}
\end{align}
where 
\begin{align}
\cM_h^K\left(c_h,z_h\right):=&\int_K \phi (\Pi_{k+1}^{0,K}c_h)(\Pi_{k+1}^{0,K}z_h)\dx+ \nu_\cM^K(\phi) S_\cM^K((I-\Pi_{k+1}^{0,K})c_h,(I-\Pi_{k+1}^{0,K})z_h)
\end{align}
with $S_\cM^K(\bullet,\bullet)$ denotes the stabilization term with property given below in \eqref{eqn.SM} and 
\[\nu_M^K(\phi)=|\Pi_{0}^{0,K}\phi|.\]
\item The term $K_h(\bullet;\bullet,\bullet)$ in \eqref{eqn.semidiscreteqz} is defined by
\begin{align}
	K_h\left(c_h;\frac{\partial p_h}{\partial t},z_h\right)=\sum_{K\in \cT_h}\int_K b(\Pi_{k+1}^{0,K}c_h)\frac{\partial p_h}{\partial t}\Pi_{k+1}^{0,K}z_h \dx.
\end{align}
	\item The term $\Theta_h(\bullet,\bullet;\bullet)$ in \eqref{eqn.semidiscreteqz} is defined by
\begin{align}
\Theta_h(\bu_h,c_h;z_h)&:=(\bu_h\cdot\nabla c_h,z_h)_h=\sum_{K \in \cT_h}\int_K \bPi_{k}^{0,K}\bu_h\cdot \bPi_{k}^{0,K}(\nabla c_h)\Pi_{k+1}^{0,K}z_h\dx.
\end{align}
\item The term $\cD_h(\bullet;\bullet,\bullet)$ in \eqref{eqn.semidiscreteqz} is defined by
\begin{align}
\cD_h\left(\bu_h,c_h;z_h\right)&:=\sum_{K\in \cT_h} \cD_h^K\left(\bu_h,c_h;z_h\right),
\end{align}
where 
\begin{align}
\cD_h^K\left(\bu_h,c_h;z_h\right):=&\int_K D(\bPi_{k}^{0,K}\bu_h) \bPi_{k}^{0,K}(\nabla c_h)\cdot\bPi_{k}^{0,K}(\nabla z_h)\dx\\
&\qquad + \nu_\cD^K(\bu_h)S_\cD^K((I-\Pi_{k+1}^{\nabla,K})c_h,(I-\Pi_{k+1}^{\nabla,K})z_h)
\end{align}
with $S_\cD^K(\bullet,\bullet)$ denotes the stabilization term with property given below in \eqref{eqn.SD} and 
\[\nu_D^K(\bu_h)=\nu_\cM^K(\phi)(d_m+d_t|\bPi_{0}^{0,K}\bu_h|).\]
\item The term $(q\bullet,\bullet)_h$  in \eqref{eqn.semidiscreteqz} is defined by
\begin{equation}
	(qc_h,z_h)_h:=\sum_{K\in \T_h}\int_K q \Pi_{k+1}^{0,K}c_h\Pi_{k+1}^{0,K}z_h \dx.
\end{equation}
\item The term $(q\hc,\bullet)_h$  in \eqref{eqn.semidiscreteqz} is defined by
\begin{equation}
(q\hc,z_h)_h:=\sum_{K\in \T_h}\int_K q\hc\Pi_{k+1}^{0,K}z_h \dx.
\end{equation}
\end{enumerate}
The stabilisation terms $S_\cM:Z_h\times Z_h\to \R$, $S_\cD:Z_h\times Z_h \to \R$, and $S_\cA:\bV_h \times \bV_h \to \R$ are symmetric and positive definite bilinear forms with the property that for all $K \in \cT_h$, there exists positive constants $M_0^\cM,M_1^\cM,M_0^\cD,M_1^\cD,M_0^\cA,$ and $M_1^\cA$ independent of $h$ and $K$ such that
\begin{subequations}
\begin{align}
&M_0^\cM \|z_h\|_{0,K}^2 \le S_\cM^K(z_h,z_h) \le M_1^\cM \|z_h\|_{0,K}^2 \quad \fl z_h \in Z_h \cap \ker(\Pi_{k+1}^{0,K})\label{eqn.SM}\\
&M_0^\cD \|\nabla z_h\|_{0,K}^2 \le S_\cD^K(z_h,z_h) \le M_1^\cD \|\nabla z_h\|_{0,K}^2 \quad \fl z_h \in Z_h \cap \ker(\Pi_{k+1}^{\nabla,K})\label{eqn.SD}\\
&M_0^\cA \|\bv_h\|_{0,K}^2 \le S_\cA^K(\bv_h,\bv_h) \le M_1^\cA \|\bv_h\|_{0,K}^2 \quad \fl \bv_h \in \bV_h \cap \ker(\bPi_{k}^{0,K}).\label{eqn.SA}
\end{align}
\end{subequations}
The above properties prove the continuity of $\cM_h^K(\bullet,\bullet)$, $\cD_h^K(\bullet,\bullet;\bullet)$, and $\cA_h^K(\bullet,\bullet;\bullet)$ respectively. Under the mesh assumption (\textbf{D1})-(\textbf{D2}), a simplier choice of these stabilisation terms are given by
\begin{align*}
S_\cM^K(c_h,z_h)&=|K|\sum_{j=1}^{\dim Z_h(K)}\dof_j^{Z_h(K)}(c_h) \dof_j^{Z_h(K)} (z_h)\\
S_\cD^K(c_h,z_h)&=\sum_{j=1}^{\dim Z_h(K)}\dof_j^{Z_h(K)}(c_h) \dof_j^{Z_h(K)} (z_h)\\
S_\cA^K(\bu_h,\bv_h)&=|K|\sum_{j=1}^{\dim \bV_h(K)}\dof_j^{\bV_h(K)}(\bu_h) \dof_j^{\bV_h(K)} (\bv_h).
\end{align*}

\noindent The lemma stated below shows the continuity and coercivity properties of the discrete bilinear form in \eqref{eqn.semidiscrete}.The proof follows analogous to \cite[Lemma 3.2]{Veiga_miscibledisplacement_2021} with \cite[Lemma 3.1.c]{Veiga_miscibledisplacement_2021} replaced with Lemma~\ref{lem.approx}.c and hence is skipped.
\begin{lem}[Properties of the discrete bilinear forms]\label{lem.propertiesdiscrete} The following properties hold for the discrete bilinear forms in \eqref{eqn.semidiscrete}
	\begin{itemize}
		\item[$(a)$] $\cM_h(c_h,z_h) \lesssim \|c_h\|\|z_h\|$ for all $c_h,z_h\in Z_h$,
		\item[(b)] $\cM_h(z_h,z_h) \gtrsim\|z_h\|^2$ for all $z_h \in Z_h$,
			\item[(c)] $\cD_h(\bu_h;c_h,z_h) \lesssim | c_h|_{1,\cT_h}|z_h|_{1,\cT_h}$ for all $\bu_h \in \bV_h$ and $c_h,z_h \in Z_h$,
		\item[(d)] $\cD_h(\bu_h;z_h,z_h) \gtrsim |z_h|_{1,\cT_h}^2$ for all $\bu_h \in \bV_h$ and $z_h \in Z_h$,
	\item[(e)] $\cA_h(c_h;\bu_h,\bv_h) \lesssim \|\bu_h\|\|\bv_h\|$ for all $c_h \in Z_h$ and $\bu_h,\bv_h \in \bV_h$,
		\item[(f)] $\cA_h(c_h;\bv_h,\bv_h) \gtrsim \|\bv_h\|^2$ for all $c_h \in Z_h$ and $\bv_h \in \bV_h$.
	\end{itemize}
Thus, $\cA_h(c_h,\bullet,\bullet)$ is coercive on the kernel
\begin{equation}\label{defn.kerneldiscrete}
\cK_h:=\{\bv_h \in \bV_h:B(\bv_h,q_h)=0 \fl q_h \in Q_h\} \subset \cK
\end{equation}
with respect to $\|\bullet\|_{\bV_h}$, where $\cK$ is given in \eqref{defn.kernel}.
\end{lem}

\subsection{Energy projection for velocity and pressure}
For a fixed $c$, define the projection operators $\P_\bu:\bV \to \bV_h$ and $\P_p:Q \to Q_h$ by
\begin{subequations}\label{defn.projection}
\begin{align}
	\cA_h^c(\P_\bu \bu,\bv_h)+B(\bv_h,\P_p p)&=\cA(c;\bu,\bv_h)+B(\bv_h,p) \fl \bv_h \in \bV_h \label{defn.projectiona}\\
	B(\P_\bu \bu,w_h)&=B(\bu,w_h) \fl w_h \in Q_h.\label{defn.projectionb}
\end{align}
\end{subequations}
where
\begin{align*}
	\cA_h^c( \bu_h,\bv_h)&=\sum_{K \in \cT_h}\int_K A(c) \bPi_{k}^{0,K}(\bu_h)\cdot\bPi_{k}^{0,K}(\bv_h)\dx\\
	&\qquad +\sum_{K \in \cT_h} \nu_\cA^K(c) S_\cA^K((I-\bPi_{k}^{0,K})\bu_h,(I-\bPi_{k}^{0,K})\bv_h).
\end{align*}
The wellposedness of the projection operators $\P_\bu$ and $\P_p$ follows from the coercivity of $\cA_h^c(\bullet,\bullet)$ and the inf-sup condition of $B(\bullet,\bullet)$.

\medskip

\noindent Given $\bu \in (H^{k+1}(\cT_h))^2$ and $p \in H^{k+1}(\cT_h)$, there exist interpolants $\bu_I \in \bV_h$ and $p_I \in Q_h$ such that
\begin{align}
&	\|\bu-\bu_I\| \lesssim h^{k+1}\|\bu\|_{k+1,\cT_h} \label{eqn.interpolantu}\\
&	\|p-p_I\| \lesssim h^{k+1}\|p\|_{k+1,\cT_h}.\label{eqn.interpolantp}
\end{align}
	 For $\bv_h \in \bV_h$, the following inverse estimate \cite[(41)]{Veiga_miscibledisplacement_2021} holds:
	 \begin{equation}\label{eqn.vhinfty}
\|\bPi_k^{0,K}\bv_h\|_{0,\infty,K} \lesssim \|\bv_h\|_{0,\infty,K}.
	 \end{equation}
	Analogously, for $\bv \in L^\infty(K)$,
	 \begin{equation}\label{eqn.vinfty}
		\|\bPi_k^{0,K}\bv\|_{0,\infty,K} \lesssim \|\bv\|_{0,\infty,K}.
	\end{equation}
	\begin{lem}[Auxiliary result]\cite[Lemma 4.1]{Veiga_miscibledisplacement_2021}\label{lem.aux}
		Let $r,s,t \in \N_0$. Let $\Pi_r^0$ and $\bPi_s^0$ denote the elementwise $L^2$ projectors onto scalar and vector valued polynomials of degree atmost $r$ and $s$ respectively. Given a scalar function $\sigma \in H^{m_r}(\cT_h)$, $0 \le m_r \le r+1$, let $\kappa(\sigma)$ be a tensor valued piecewise Lipschitz continuous with respect to $\sigma$. Further, let $\widehat{\sigma} \in L^2(\O)$ and let $\bchi$ and $\bpsi$ be vector valued functions. Assume that $\kappa(\sigma) \in (L^\infty(\O))^{2 \times 2}$, $\bchi \in (H^{m_s}(\cT_h)\cap L^\infty(\O))^2$, $\bpsi \in (L^2(\O))^2$ and $\kappa(\sigma)\bchi \in (H^{m_t}(\cT_h))^2$, for some $0 \le m_s\le s+1$ and $0 \le m_t \le t+1.$ Then, for any $K \in \cT_h$,
		\begin{align*}
			&(\kappa(\sigma)\bchi,\bpsi)_{0,K}-(\kappa(\Pi_r^{0,K}\widehat{\sigma})\bPi_s^{0,K}\bchi,\bPi_t^{0,K}\bpsi)_{0,K}
			\\
			&\le \eta \left[h^{m_t}|\kappa(\sigma)\bchi|_{{m_t},K}+h^{m_s}|\bchi|_{{m_s},K}\|\kappa(\sigma)\|_{0,\infty,K}+(h^{m_r}|\sigma|_{{m_r},K}+\|\sigma-\widehat{\sigma}\|_{0,K})\|\bchi\|_{0,\infty,K}\right]\|\bpsi\|_{0,K}.
		\end{align*}
		Consequently,
		\begin{align*}
			(\kappa(\sigma)&\bchi,\bpsi)-(\kappa(\Pi_r^0\widehat{\sigma})\bPi_s^0\bchi,\bPi_t^0\bpsi)
			\\
			&\le \eta \left[h^{m_t}|\kappa(\sigma)\bchi|_{{m_t},\T_h}+h^{m_s}|\bchi|_{{m_s},\T_h}\|\kappa(\sigma)\|_{0,\infty,\O}+(h^{m_r}|\sigma|_{{m_r},\T_h}+\|\sigma-\widehat{\sigma}\|)\|\bchi\|_{0,\infty,\O}\right]\|\bpsi\|.
		\end{align*}
	\end{lem}
\begin{lem}\label{lem.projectionup}
 Let $(\bu,p,c) \in  \bu \in (H^{k+1}(\cT_h)\cap L^\infty(\Omega))^2 \times H^{k+1}(\cT_h) \times H^{k+1}(\cT_h)$ solves \eqref{eqn.weak} at time $t=t_n$. Given $c_h^n \in Z_h$, let $(u_h^n,p_h^n) \in \bV_h \times Q_h$ be the solution to \eqref{eqn.fullydiscrete_uhph}. Assume that $A(c)$ is piecewise Lipschitz continuous functions with respect to $c \in L^2(\Omega)$, and $A(c^n)\bu^n \in (H^{k+1}(\cT_h))^2$. Then,
 \begin{equation}
 	(a) \, \|\bu-\P_\bu \bu\|\lesssim h^{k+1} \quad (b)\,\|p-\P_p p\| \lesssim h^{k+1}.
 \end{equation}
\end{lem}
	{\it Proof of $(a)$.} The triangle inequality and \eqref{eqn.interpolantu} lead to
	\begin{equation}\label{eqn.tri.}
		\|\bu-\P_\bu \bu\| \le \|\bu-\bu_I\|+\|\bu_I-\P_\bu \bu\| \lesssim h^{k+1}+\|\bu_I-\P_\bu \bu\|.
	\end{equation}
	Let $\delta_h=\P_\bu \bu-\bu_I$. Since $\divc \delta_h =0$, the coercivity of $\cA_h^c(\bullet,\bullet)$ and \eqref{defn.projectiona} show
	\begin{align}
	\|\delta_h\|^2=	\|\delta_h\|_\bV^2& \lesssim \cA_h^c(\delta_h,\delta_h)=\cA_h^c(\P_\bu \bu-\bu_I,\delta_h)\\
		&=\cA(c;\bu,\delta_h)+B(\delta_h,p)-B(\delta_h,\P_p p)-\cA_h^c(\bu_I,\delta_h)\\
		&=\cA(c;\bu,\delta_h)-\cA_h^c(\bu_I,\delta_h)\\
		&=(\cA(c;\bu,\delta_h)-\cA_h^c(\bu,\delta_h))+\cA_h^c(\bu-\bu_I,\delta_h):=T_1+T_2.\label{eqn.t1t2}
	\end{align}
	Simple manipulation leads to
	\begin{align}
		T_1&=\sum_{K\in \cT_h}\int_K(A(c)\bu \cdot \delta_h-A(\Pi_{k+1}^{0,K}c)\bPi_k^{0,K}\bu \cdot \bPi_{0,K}\delta_h)\dx\nonumber\\
	&\quad +\sum_{K\in \cT_h}\int_K(A(\Pi_{k+1}^{0,K}c)-A(c))\bPi_k^{0,K}\bu \cdot \bPi_{0,K}\delta_h\dx\nonumber\\
	&\quad +\sum_{K \in \cT_h} \nu_\cA^K(c) S_\cA^K((I-\bPi_{k}^{0,K})\bu,(I-\bPi_{k}^{0,K})\delta_h)=:S_{1,1}+S_{1,2}+S_{1,3}.\label{eqn.t1}
	\end{align}
	Lemma~\ref{lem.aux} provides
	$S_{1,1} \lesssim h^{k+1}\|\delta_h\|$. The \Holders inequality, the Lipschitz continuity of $A(c)$, and \eqref{eqn.vinfty} read $S_{1,2}\lesssim h^{k+1}\|\delta_h\|$. The \Holders inequality, boundedness of $A(\bullet)$, \eqref{eqn.SA}, and Lemma~\ref{lem.approx} show $S_{1,3}\lesssim h^{k+1}\|\delta_h\|$. A substitution of these estimates in \eqref{eqn.t1} leads to
	\begin{equation}\label{eqn.t1.1}
		T_1 \lesssim h^{k+1}\|\delta_h\|.
	\end{equation}
	The boundedness of $\cA_h^c(\bullet,\bullet)$ and \eqref{eqn.interpolantu} provide $T_2 \lesssim h^{k+1}\|\delta_h\|$. A combination of $T_1$ and $T_2$ in \eqref{eqn.t1t2} together with \eqref{eqn.tri.} concludes the proof of $(a)$.
	
	\medskip
	
\noindent {\it Proof of $(b)$.} The triangle inequality and \eqref{eqn.interpolantp} lead to
\begin{equation}\label{eqn.tri.1}
	\|p-\P_p p\| \le \|p-p_I\|+\|p_I-\P_p p\| \lesssim h^{k+1}+\|p_I-\P_p p\|.
\end{equation}
Let $w_h^*=p_I-\P_p p \in Q_h \subset Q$. Then there exists $\bv_h^* \in \bV_h$ \cite{mixedVEM_Brezzi} such that $\divc \bv_h^*=w_h^*$ and $\beta^*\|\bv_h^*\| \le \|w_h^*\|.$ This, the interpolation property of $p_I$, and \eqref{defn.projectiona} show
\begin{align}
	\|w_h^*\|^2&=(p_I-\P_p p, \divc \bv_h^*)=(p-\P_p p, \divc \bv_h^*)\\
	&=\cA(c;\bu,\bv_h^*)-\cA_h^c(\P_\bu \bu,\bv_h^*)\\
		&=(\cA(c;\bu,\bv_h^*)-\cA_h^c(\bu,\bv_h^*))+\cA_h^c(\bu-\P_\bu \bu,\bv_h^*).\label{eqn.wh*}
\end{align}
The first term on the right hand side of \eqref{eqn.wh*} can be estimated as in $T_1$ of \eqref{eqn.t1t2} and the second term can be bounded using the continuity of $\cA_h^c(\bullet,\bullet)$ and $(a)$. Hence, $\|w_h^*\|^2 \lesssim h^{k+1}\|\bv_h^*\|$. Since $\beta^*\|\bv_h^*\| \le \|w_h^*\|$, $\|w_h^*\| \lesssim h^{k+1}$. A combination of this with \eqref{eqn.tri.1} concludes the proof of $(b)$.

\medskip

\noindent By differentiation of \eqref{defn.projectiona}-\eqref{defn.projectionb}, analogous result can be obtained for $\frac{\partial}{\partial t}(\bu-\P_\bu \bu)$ and $\frac{\partial}{\partial t}(p-\P_p p)$. These results are stated next.

\begin{cor}\label{cor.up}
	Provided that the continuous data and solution are sufficiently regular in space and time, it holds that
	\begin{equation}
	\displaystyle	(a) \, \|\frac{\partial}{\partial t}(\bu-\P_\bu \bu)\|\lesssim h^{k+1} \quad (b)\,\|\frac{\partial}{\partial t}(p-\P_p p)\| \lesssim h^{k+1}.
	\end{equation}
\end{cor}
\subsection{Energy projection for concentration}
 

\noindent For fixed $\bu(t) \in \bV$ and $t \in J$, define the projector $\P_c:Z \cap H^2(\O) \to Z_h$ \cite{ncVEM_2016,ncVEM_parabolic} by
\begin{equation}\label{eqn.projectionP}
	\Gamma_{c,h}(\bu(t);\P_cc,z_h)=\Gamma_{c,\pw}(\bu(t);c,z_h)-\cN_h(\bu(t);c,z_h)
\end{equation}
for all $z_h \in Z_h$, where
\begin{align}
	\Gamma_{c,h}(\bu;c_h,z_h)&:=\cD_h^\bu(c_h,z_h)+\Theta_h^\bu(c_h,z_h)+\lambda(c_h,z_h)_h\\
	\Gamma_{c,\pw}(\bu;c,z_h)&:=\cD_{\pw}^\bu(c,z_h)+\Theta_{\pw}^\bu(c,z_h)+\lambda(c,z_h)\\
	\cN_h(\bu;c,z_h)&:=\sum_{e \in \cE_h}\int_e\left(D(\bu)\nabla c \cdot \bn_e \right)\jump{z_h}\ds
\end{align}
with
\begin{align}
	\cD_h^\bu\left(c_h,z_h\right)&:=\sum_{K \in \cT_h}\int_K D(\bu) \bPi_{k}^{0,K}(\nabla c_h)\cdot\bPi_{k}^{0,K}(\nabla z_h)\dx\\
	&\qquad +\sum_{K \in \cT_h} \nu_\cD^K(\bu)S_\cD^K((I-\Pi_{k+1}^{\nabla,K})c_h,(I-\Pi_{k+1}^{\nabla,K})z_h)\\
	\Theta_h^\bu(c_h,z_h)&:=\sum_{K\in \cT_h}\int_K \bu\cdot \bPi_{k}^{0,K}(\nabla c_h)\Pi_{k+1}^{0,K}z_h\dx\label{eqn.Thetahu}\\
\cD_{\pw}^\bu(\bullet,\bullet)&:=\sum_{K \in \cT_h}\cD^K(\bu;\bullet,\bullet),\qquad \Theta_{\pw}^\bu(\bullet,\bullet):=\sum_{K \in \cT_h}\Theta^K(\bu,\bullet;\bullet)
\end{align}
and $\jump{z_h}$ denotes the jump of $z_h$ across the edge $e$. Here, $\nu_D^K(\bu)=\nu_\cM^K(\phi)(d_m+d_t|\bPi_{0}^{0,K}\bu|)$ and and $\cD^K(\bu,\bullet,\bullet)$ and $\Theta^K(\bu,\bullet,\bullet)$ denotes the piecewise (elementwise) contribution of $\cD(\bullet,\bullet,\bullet)$ and $\Theta(\bullet,\bullet,\bullet)$, respectively. The constant $\lambda>0$ is chosen such that $\Gamma_{c,h}(\bu;\bullet,\bullet)$ is coercive.

%

\begin{lem}\cite[Lemma 4.1]{ncVEM_2016}\label{lem.cN}
	Let $k\ge 0$, $c \in H^{k+2}(\cT_h)$ and $D(\bu)\nabla c \in (H^{k+1}(\cT_h))^2$. Then,  for all $z_h \in Z_h$,
	$$\cN_h(\bu;c,z_h) \lesssim h^{k+1}\|D(\bu)\nabla c\|_{k+1,\cT_h}|z_h|_{1,\cT_h}.$$
	Moreover, if  $c\bu \in (H^{k+1}(\cT_h))^2$,
	\[\sum_{e \in \cE_h}\int_e\left({c\bu \cdot n_e} \right)\jump{z_h}\ds\lesssim h^{k+1}\|c\bu\|_{k+1}|z_h|_{1,\cT_h}.\]
\end{lem}
\begin{lem}\label{lem.projectionPc}
	The projector operator $\P_c$ in \eqref{eqn.projectionP} is well-defined under the assumption that $\bu \in (L^{\infty}(\O))^2$ for all $t \in J$.
\end{lem}
\begin{proof}
	To apply Lax-Milgram Lemma, the continuity of the right-hand side of \eqref{eqn.projectionP} with respect to $\|\bullet\|_{1,\cT_h}$ is proved first. For $\bu \in L^\infty(\O)$, $c \in Z$, and $z_h \in Z_h$,
	the generalised \Holder inequality and the definition of $D(\bu)$ in \eqref{defn.D} (also see Lemma \ref{lem.propertiescontinuous}.g) imply
	\begin{align}
		\cD_{\pw}^\bu(c,z_h)&=\sum_{K \in \cT_h}(D(\bu)\nabla c,\nabla z_h)_{0,K} \\
		&\le \sum_{K \in \cT_h}\|D(\bu)\|_{0,\infty,K}\|\nabla c\|_{0,K}\|\nabla z_h\|_{0,K}\lesssim \|\bu\|_{0,\infty,\O}\|c\|_{1,\cT_h}\|z_h\|_{1,\cT_h}.\label{eqn.Dpw}
	\end{align}
	The remaining two terms on the definition of $\Gamma_{c,\pw}(\bu;\bullet,\bullet)$ is estimated using a generalised \Holder inequality as
	\begin{align}
	\Theta_{\pw}^\bu(c,z_h)+\lambda (c,z_h)&
		\lesssim \|\bu\|_{0,\infty,\O}|c|_{1,\cT_h}\|z_h\|+\lambda\|c\|\|z_h\|.\label{eqn.thethapw}
	\end{align}
The definition of $H^1(\cT_h;k)$ and $L^2(e)$ projection, Cauhy-Schwarz inequality, and approximation properties $\|D(\bu)\nabla c  -\P_k^e(D(\bu)\nabla c)\|_{0,e}\lesssim h^{-1/2}\|D(\bu)\nabla c\|$ and $\|\jump{z_h}-\P_0^e(\jump{z_h})\|_{0,e}\lesssim h^{1/2}|z_h|_{1,\cT_h}$ prove
\begin{align*}
	\cN_h(\bu;c,z_h)&=\sum_{e \in \cE_h}\int_e\left(D(\bu)\nabla c  -\P_k^e(D(\bu)\nabla c)\right)\cdot \bn_e\jump{z_h}\ds\\
	&=\sum_{e \in \cE_h}\int_e\left(D(\bu)\nabla c  -\P_k^e(D(\bu)\nabla c)\right)\cdot \bn_e(\jump{z_h}-\P_0^e(\jump{z_h}))\ds\\
	&\lesssim \|\bu\|_{0,\infty,\O}|c|_{1,\cT_h}\|z_h\|_{1,\cT_h}.
\end{align*}
This, \eqref{eqn.thethapw}, and \eqref{eqn.Dpw} shows that $\Gamma_{c,\pw}(\bu;c,\bullet)$ is a continuous functional with respect to $\|z_h\|_{1,\cT_h}$.

	\smallskip
	
	\noindent  The continuity of $\cD_h^\bu(\bullet,\bullet)$ is analogue to the one in Lemma~\ref{lem.propertiesdiscrete}.c. A generalised \Holder inequality and the stability of the $L^2$ projector shows
	\[	\Theta_h^\bu(c_h,z_h)_h+\lambda(c_h,z_h)_h
	\lesssim \|\bu\|_{0,\infty,\O}|c_h|_{1,\cT_h}\|z_h\|+\lambda\|c\|\|z_h\|.\]
A combination of these estimates shows the continuity of $\Gamma_{c,h}(\bu;\bullet,\bullet)$. An appropriate choice of $\lambda$ implies the coercivity of $\Gamma_{c,h}(\bu;z_h,z_h)$. The result then follows from the Lax-Milgram lemma.
\end{proof}
\noindent Given $c \in H^{k+2}(\cT_h)$, there exists an interpolant $c_I \in Z_h$ and a piecewise $\P_{k+1}$ polynomial $c_{\pi}$ such that \cite{ncVEM_2016,ncVEM_parabolic}
\begin{align}
	&\|c-c_I\| \lesssim h^{k+2}\|c\|_{k+2,\cT_h}, \quad \|c-c_I\|_{1,\cT_h} \lesssim h^{k+1}\|c\|_{k+2,\cT_h}\label{eqn.cI}\\
	&\|c-c_{\pi}\| \lesssim h^{k+2}\|c\|_{k+2,\cT_h}, \quad \|c-c_{\pi}\|_{1,\cT_h} \lesssim h^{k+1}\|c\|_{k+2,\cT_h}.\label{eqn.cpi}
\end{align}
Also, note that
\begin{equation}\label{eqn.cpiproperties}
	\Pi_{k+1}^{0}c_\pi=c_\pi, \quad \Pi_{k+1}^{\nabla}c_\pi=c_\pi.
\end{equation}
\begin{lem}\label{lem.cPcerror}
	Assume that $\bu \in (H^{k+1}(\cT_h)\cap W^{1,\infty}(\O))^2,$ $c \in H^{k+2}(\cT_h) \cap W^{1,\infty}(\cT_h)$, $\bu \cdot \nabla c \in H^{k+1}(\cT_h)$, $D(\bu) \in (W^{1,\infty}(\Omega))^2$, and $D(\bu)\nabla c \in (H^{k+1}(\cT_h))^2$ for all $t \in J$. Then 
	\begin{align*}
		(a)\,	\|c-\P_c c\|_{1,\cT_h} \lesssim h^{k+1}, \qquad & (b)\,\|c-\P_c c\| \lesssim h^{k+2},
	\end{align*}
	where $\P_c$ is defined in \eqref{eqn.projectionP}.
\end{lem}
\begin{proof}
	The triangle inequality and \eqref{eqn.cI} lead to
	\begin{equation}\label{eqn.tri}
		\|c-\P_c c\|_{1,\cT_h} \le \|c-c_I\|_{1,\cT_h}+\|c_I-\P_c c\|_{1,\cT_h} \lesssim h^{k+1}+\|c_I-\P_c c\|_{1,\cT_h}.
	\end{equation}
	For a fixed time $t$, the coercivity of $\Gamma_{c,h}(\bu;\bullet,\bullet)$ with respect to $\|\bullet\|_{1,\cT_h}$ and \eqref{eqn.projectionP} show, for $z_h=\P_c c-c_I \in Z_h$,
	\begin{align}
		\|z_h\|_{1,\cT_h}^2& \lesssim \Gamma_{c,h}(\bu;z_h,z_h)=\Gamma_{c,h}(\bu;\P_c c,z_h)-\Gamma_{c,h}(\bu;c_I,z_h)\nonumber\\
		&=\Gamma_{c,\pw}(\bu;c,z_h)-\Gamma_{c,h}(\bu;c_I,z_h)-\cN_h(\bu;c,z_h) \nonumber\\
		&=(\Gamma_{c,\pw}(\bu;c,z_h)-\Gamma_{c,h}(\bu;c_{\pi},z_h))+\Gamma_{c,h}(\bu;c_{\pi}-c_I,z_h)-\cN_h(\bu;c,z_h)\nonumber\\
		&=(\cD_{\pw}^\bu(c,z_h)-\cD_h^\bu(c_{\pi},z_h))+(\Theta_{\pw} ^\bu(c,z_h)-\Theta_h^\bu(c_{\pi},z_h))\nonumber\\
		&\quad + \lambda((c,z_h)-(c_{\pi},z_h)_h)+\Gamma_{c,h}(\bu;c_{\pi}-c_I,z_h) -\cN_h(\bu;c,z_h)=: \sum_{i=1}^5A_i \label{eqn.a1a2a3a4a5}
	\end{align}
	with the definition of $\Gamma_{c,\pw}(\bu;\bullet,\bullet)$ and $\Gamma_{c,h}(\bu;\bullet,\bullet)$ in the last step. 
	An introduction of $\bPi_k^{0,K}(\nabla c)$, Lemma \ref{lem.aux} with the substitutions $\sigma=\bu$, $\widehat{\sigma}=0$, $\bchi=\nabla c$, and $\bpsi=\nabla z_h$, generalised \Holder inequality, \eqref{eqn.cpi}, and the approximation property $\|\bu-\Pi_k^{0,K}\bu\| \lesssim h^{k+1}$ leads to 
	\begin{align}
		A_{1}&=\sum_{K \in \cT_h} \big((D(\bu)\nabla c,\nabla z_h)_{0,K}-(D(\bu)\bPi_k^{0,K}(\nabla c_{\pi}),\bPi_k^{0,K}(\nabla z_h))_{0,K}\big)\\
		&=\sum_{K \in \cT_h} \big((D(\bu)\nabla c,\nabla z_h)_{0,K}-(D(\bPi_k^{0,K}\bu)\bPi_k^{0,K}(\nabla c),\bPi_k^{0,K}(\nabla z_h))_{0,K}\big)\\
		&\qquad +\sum_{K \in \cT_h} (D(\bPi_k^{0,K}\bu)\bPi_k^{0,K}(\nabla (c-c_{\pi})),\bPi_k^{0,K}(\nabla z_h))_{0,K}\\
		&\qquad +\sum_{K \in \cT_h} (D(\bPi_k^{0,K}\bu-\bu)\bPi_k^{0,K}(\nabla c_{\pi}),\bPi_k^{0,K}(\nabla z_h))_{0,K}\lesssim h^{k+1}|z_h|_{1,\cT_h}.\label{eqn.a1}
	\end{align}
A simple manipulation yields
\begin{align}
	A_{2}
	&= \sum_{K \in \cT_h}[((\bu\cdot\nabla c,z_h)_{0,K}-(\bu\cdot\nabla c,z_h)_h)+(\bu\cdot\nabla (c-c_\pi),z_h)_h]:=A_{2,1}+A_{2,2}.\label{eqn.a2}
\end{align}
Lemma~\ref{lem.aux}, the generalised \Holder inequality, and the continuity property of $L^2$ projectors $\Pi_{k+1}^{0,K}$ and $\bPi_k^{0,K}$ provide
\begin{equation}
	A_{2,1} \lesssim h^{k+1}\|z_h\|_{1,\cT_h}. \label{eqn.a21a22a23}
\end{equation}
The generalised \Holder inequality, the stability of $L^2$ projectors, and \eqref{eqn.cpi} read
\begin{equation}
	A_{2,2}\lesssim h^{k+1}\|z_h\|_{1,\cT_h}. \label{eqn.a24a25a26}
\end{equation}
A combination of \eqref{eqn.a21a22a23}-\eqref{eqn.a24a25a26} in \eqref{eqn.a2} shows 
\begin{equation}
	A_2 \lesssim h^{k+1}\|z_h\|_{1,\cT_h}.
\end{equation}
	Since $\Pi_{k+1}^{0}c_\pi=c_\pi$, $(c_{\pi},z_h)_h=(c_{\pi},z_h)$. This, \Holder inequality, and \eqref{eqn.cpi} result in
	\begin{align}
		A_{3}&=\lambda(c-c_{\pi},z_h) \lesssim h^{k+1}\|z_h\|_{1,\cT_h}.\label{eqn.a3}
	\end{align}
	The continuity of $\Gamma_{c,h}(\bu;\bullet,\bullet)$ from Lemma \ref{lem.projectionPc}, a triangle inequality with $c$ and the approximation properties in \eqref{eqn.cI} and \eqref{eqn.cpi} provide
	\begin{align}
		A_4&\lesssim \|\bu\|_{\infty,\O}\|c_{\pi}-c_I\|_{1,\cT_h}\|z_h\|_{1,\cT_h}\lesssim h^{k+1}\|z_h\|_{1,\cT_h}.\label{eqn.a4}
	\end{align}
	A substitution of the estimates for $A_1-A_4$ in \eqref{eqn.a1a2a3a4a5} and Lemma~\ref{lem.cN} for $A_5$ leads to $\|\P_c c-c_I\|_{1,\cT} \lesssim h^{k+1}$. This and \eqref{eqn.tri} concludes the proof of first estimate $(a)$.
	
	\medskip
	
	
	\noindent For the $L^2$ estimate, we follow the Aubin-Nitsche duality arguments. Let $\psi \in H^1(\O)$ solves the following adjoint problem.
	\begin{align}\label{eqn.dual}
		-\divc(D(\bu)\nabla \psi+\bu\psi)+\lambda\psi&=c-\P_c c\mbox{ in } \O, \quad (D(\bu)\nabla \psi+\bu \psi)\cdot \bn=0 \mbox{ in } \partial \O.
	\end{align}
	Since $\O$ is convex, $\psi \in H^2(\O)$ and 
	\begin{equation}\label{eqn.aprioridual}
		\|\psi\|_2 \lesssim \|c-\P_c c\|.
	\end{equation}
	The dual problem \eqref{eqn.dual}
	and introduction of the interpolant $\psi_I$ of $\psi$ show
	\begin{align}
		\|c-\P_c c\|^2
		&=(c-\P_c c,-\divc(D(\bu)\nabla \psi+\bu \psi)+\lambda\psi)\\
	&= (\cD_{\pw}^\bu(c-\P_c c,\psi-\psi_I)+\Theta_{\pw}^\bu(c-\P_c c,\psi-\psi_I)+\lambda(c-\P_c c,\psi-\psi_I))\\
	&\quad+(\cD_{\pw}^\bu(c-\P_c c,\psi_I)+\Theta_{\pw}^\bu(c-\P_c c,\psi_I)+\lambda(c-\P_c c,\psi_I) -\cN_h(\bu;\psi,c-\P_c c))\\
	&\quad +\widehat{\cN}_h(\bu;\psi,\P_c c-c)=:B_1+B_2+B_3\label{eqn.b1b2}
	\end{align}
where $$ \widehat{\cN}_h(\bu;\psi,z_h)=\sum_{e \in \cE_h}\int_e\left({\psi\bu \cdot n_e} \right)\jump{z_h}\ds.$$
A generalised \Holder inequality, the estimate $(a)$, \eqref{eqn.cI} for $\psi \in H^2(\O)$ and \eqref{eqn.aprioridual} provide
	\begin{equation}\label{eqn.b1}
		B_1 \lesssim h^{k+2}\|\psi\|_2 \lesssim h^{k+2}\|c-\P_c c\|.
	\end{equation}
	The definition of projection in \eqref{eqn.projectionP} implies
	\begin{align}
		B_2
		&=\Gamma_{c,\pw}(\bu;c,\psi_I)-\cD_{\pw}^\bu(\P_c c,\psi_I)-\Theta_{\pw}^\bu(\P_c c,\psi_I)-\lambda(\P_c c,\psi_I) -\cN_h(\bu;\psi,c-\P_c c)\\
		&=\Gamma_{c,h}(\bu;\P_c c,\psi_I)+\cN_h(\bu;c,\psi_I)-\cD_{\pw}^\bu(\P_c c,\psi_I)-\Theta_{\pw}^\bu(\P_c c,\psi_I)-\lambda(\P_c c,\psi_I)\\
		&\qquad  -\cN_h(\bu;\psi,c-\P_c c)\\
		&=(\cD_h^\bu(\P_c c,\psi_I)-\cD_{\pw}^\bu(\P_c c,\psi_I))+(\Theta_h^\bu(\P_c c,\psi_I)-\Theta_{\pw}^\bu(\P_c c,\psi_I))\\
		&\qquad+\lambda((\P_c c,\psi_I)_h-(\P_c c,\psi_I)) +(\cN_h(\bu;c,\psi_I) -\cN_h(\bu;\psi,c-\P_c c))\\
		&=:B_{2,1}+B_{2,2}+B_{2,3}+B_{2,4}.\label{eqn.b2}
	\end{align}
	Arguments analogous to the proof of \cite[(5.48),(5.49)]{VEM_general_2016} with $p_h=\P_c c$ and $p=c$ leads to
	\begin{align}
		B_{2,1}&+B_{2,2}+B_{3,2} \lesssim h^{k+2}\|c-\P_c c\|.
	\end{align}
	Since $\psi \in H^2(\O)$, $\cN_h(\bu;c,\psi)=0$. This, Lemma~\ref{lem.cN}, \eqref{eqn.cI}, $(a)$, and \eqref{eqn.aprioridual} show
	\begin{align}
		B_{2,4}&=\cN_h(\bu;c,\psi-\psi_I) -\cN_h(\bu;\psi,c-\P_c c) \\
		&\lesssim h^{k+1}\|D(\bu)\nabla c\|_{k+1}|\psi-\psi_I|_{1,h}+h\|D(\bu)\psi\|_1|c-\P_c c|_{1,h}\\
		&\lesssim h^{k+2}\|\psi\|_2 \lesssim h^{k+2}\|c-\P_c c\|.
	\end{align}
	A combination of the above estimates in \eqref{eqn.b2} leads to $B_2 \lesssim h^{k+2}\|c-\P_c c\|$. Analogous arguments show $B_3 \lesssim h^{k+2}\|c-\P_c c\|$. These estimates and \eqref{eqn.b1} in \eqref{eqn.b1b2} concludes the proof.
\end{proof}
Differentiation of \eqref{eqn.projectionP} with respect to time and analogous arguments as in Lemma~\ref{lem.cPcerror} yield the following result.
\begin{cor}\label{cor.ctPcterror}
	Provided that the continuous data and solution are sufficiently regular in space and time, it holds
	\begin{align*}
		\displaystyle	(a)\,	\left\|\frac{\partial}{\partial t}(c-\P_c c)\right\|_{1,\cT_h} \le h^{k+1}\xi_{1,t}, \qquad & (b)\,\left\|\frac{\partial}{\partial t}(c-\P_c c)\right\| \le h^{k+2}\xi_{0,t},
	\end{align*}
	where the constants $\xi_{0,t},\xi_{1,t}>0$ are independent of $h$.
\end{cor}
\begin{lem}\cite[Lemma 4.4, 4.5]{Veiga_miscibledisplacement_2021}\label{lem.cPcnuuhn}
	Under sufficient smoothness of the continuous data and solution, it holds
	\begin{align}
		&\,\left\|\frac{\partial c^n}{\partial t}-\frac{\P_c c^n-\P_c c^{n-1}}{\tau}\right\|\le \tau^{\half}\left\| \frac{\partial^2 c}{\partial s^2}\right\|_{L^2(t_{n-1},t_n;L^2(\O))}+\tau^{-\half}h^{k+2}\left(\int_{t_{n-1}}^{t_n}\xi_{0,t}^2\ds\right)^{\half},\\
	\end{align}
	where $\xi_{0,t}$ is defined in Corollary \ref{cor.ctPcterror}.
\end{lem}
\subsection{Fully discrete formulation}
\noindent A semi-discrete formulation of \eqref{eqn.weak} is presented in Section~\ref{sec.semidiscrete}. This section deals with the fully discrete formulation which is discrete in both space and time. The temporal discretization is achieved through the utilization of a backward Euler method.

\smallskip

\noindent Let $0=t_0<t_1<\cdots<t_N=T$ be a given partition of $J=[0,T]$ with time step size $\tau$. That is, $t_n=n\tau$, $n=0,1,\cdots,N$. For a generic function $f(t)$, define $f^n:=f(t_n)$, $n=0,1,\cdots,N$. Also, define
\[\bu^n:=\bu(t_n),\, p^n:=p(t_n),\, c^n=c(t_n)\]
and
\[\bu_h^n:=\bu_h(t_n),\, p_h^n:=p_h(t_n),\, c_h^n=c_h(t_n).\]
At $t_n$, $n=1,\cdots,N$ with $c_h^0=c_{0,h}=\P_c c^0$ and $(\bu_h^0,p_h^0)=(\bu_{0,h},p_{0,h})=(\P_\bu \bu^0,\P_p p^0)$, the fully discrete formulation corresponding to the velocity-pressure equation seeks $(\bu_h^n,p_h^n) \in \bV_h \times Q_h$ such that
\begin{subequations}\label{eqn.fullydiscrete_uhph}
	\begin{align}
		\cA_h(c_{h}^{n-1};\bu_h^n,\bv_h)+B(\bv_h,p_h^n)&=0, \fl \bv_h \in \bV_h \label{eqn.fullydiscreteqv}\\
		W_h\left(c_h^{n-1};\frac{p_h^{n}-p_h^{n-1}}{\tau},w_h\right)-B(\bu_h^n,w_h)&=(q^n,w_h), \fl w_h \in Q_h.\label{eqn.fullydiscreteqq}
	\end{align}
\end{subequations}\noeqref{eqn.fullydiscreteqv,eqn.fullydiscreteqq}
Once $(\bu_h^n,p_h^n)$ is solved, the approximation to concentration at time $t=t_{n}$ can be obtained. The fully discrete formulation corresponding to the concentration equation seeks $c_h^{n}\in Z_h$ such that
\begin{align}
	\cM_h(\frac{ c_h^{n}-c_h^{n-1}}{\tau},z_h)+	K_h\left(c_h^{n};\frac{p_h^{n}-p_h^{n-1}}{\tau},z_h\right)&+\Theta_h(\bu_h^n,c_h^{n};z_h)+(q^nc_h^{n},z_h)_h+\cD_h(\bu_h^n;c_h^{n},z_h)\\
	=&(q^{n}\hc^{n},z_h)_h, \fl z_h \in Z_h.\label{eqn.fullydiscreteq_zh}
\end{align}
Note that \eqref{eqn.fullydiscrete_uhph} and \eqref{eqn.fullydiscreteq_zh} are decoupled from each other and hence represent system of linear equations eventhough the original problem is a nonlinear coupled system problem for concentration, pressure, and velocity.

\section{Error estimates}\label{sec:errors}
This section establishes the error estimates for velocity, pressure and concentration.

\medskip

\noindent For a generic function $f$, denote $\displaystyle d_tf^n=\frac{f^n-f^{n-1}}{\tau}$. Set
\begin{align*}
&\theta=\P_p p-p_h, \quad \bbeta=\P_\bu \bu -\bu_h, \quad \xi=\P_c c-c_h\\
&\eta=p-\P_p p, \quad \balpha=\bu-\P_\bu \bu, \quad \zeta=c-\P_c c.
\end{align*}
Note that $\theta^0=\bbeta^0=\xi^0=0$. 
\subsection{Error estimates for velocity and pressure}
This section deals with the error estimates for velocity and pressure.

\medskip

\noindent A use of \eqref{eqn.weakv}-\eqref{eqn.weakq} and \eqref{eqn.fullydiscrete_uhph} at $t=t^n$ and simple manipulation leads to
\begin{align}
	\cA(c^n;\bu^n,\bv_h)-\cA_h(c_h^{n-1};\bu_h^n,\bv_h)+B(\bv_h,p^n-p_h^n)&=0 \quad \fl \bv_h \in \bV_h,\\
	W(c^n;\frac{\partial p^n}{\partial t},w_h)-W_h(c_h^{n-1};\frac{p_h^n-p_h^{n-1}}{\tau},w_h)+B(\bu_h^n-\bu^n,w_h)&=0 \quad  \fl w_h \in Q_h.
\end{align}
This with \eqref{defn.projection} results in
\begin{align}
		\cA_h^{c^n}(\P_\bu \bu^n,\bv_h)-\cA_h(c_h^{n-1};\bu_h^n,\bv_h)+B(\bv_h,\theta^n)&=0 \quad \fl \bv_h \in \bV_h,\\
	W(c^n;\frac{\partial p^n}{\partial t},w_h)-W_h(c_h^{n-1};\frac{p_h^n-p_h^{n-1}}{\tau},w_h)-B(\bbeta^n,w_h)&=0 \quad  \fl w_h \in Q_h.
\end{align}
This can be rewritten as
\begin{align}
	&\cA_h(c_h^{n-1};\bbeta^n,\bv_h)+B(\bv_h,\theta^n)=	\cA_h(c_h^{n-1};\P_\bu \bu^n,\bv_h)-\cA_h^{c^n}(\P_\bu \bu^n,\bv_h),\label{eqn.upn}\\
	&W_h(c_h^{n-1};d_t \theta^n,w_h)-B(\bbeta^n,w_h)=-\left((d(c^n)-d(\Pi_{k+1}^{0}c_h^{n-1}))\frac{\partial p^n}{\partial t},w_h\right)\\
&-\left(d(\Pi_{k+1}^{0}c_h^{n-1})(\frac{\partial p^n}{\partial t}-\frac{p^n-p^{n-1}}{\tau}),w_h\right)
	-\left(d(\Pi_{k+1}^{0}c_h^{n-1})d_t\eta^n,w_h\right).\label{eqn.upn-1}
\end{align}
\begin{thm}\label{thm.up}
	 Let $(\bu,p,c)$ solves \eqref{eqn.weak} at time $t=t_n$. Given $c_h^n \in Z_h$, let $(u_h^n.p_h^n) \in \bV_h \times Q_h$ be the solution to \eqref{eqn.fullydiscrete_uhph}.  Suppose the space and time discretisation satisfy $\tau=\mathcal{O}(h^{k+1})$. Then, under the regularity assumption $(R)$,
	 \begin{align*}
	 	\|d_t\theta^N\|^2+\tau \sum_{n=1}^{N}\|d_t\bbeta^n\|^2 \le C\bigg(\tau^2+h^{2(k+1)}+ \tau \sum_{n=1}^N(\|d_t \theta^n\|^2+\|d_t \xi^{n-1}\|^2+\|\xi^{n-1}\|^2+\|\bbeta^{n-1}\|^2)\bigg).
	 \end{align*}
\end{thm}
\begin{proof}
	Subtracting \eqref{eqn.upn}-\eqref{eqn.upn-1} at the $n$th time level $t=t_n$ and $(n-1)$th time level $t=t_{n-1}$, dividing by $\tau$, choosing $\bv_h=d_t\bbeta^n$ and $w_h=d_t \theta^n$ and adding the resultant equations, we obtain, for $K \in \cT_h$,
	\begin{align}
		(d_t(d(\Pi_{k+1}^{0,K}&c_h^{n-1})d_t\theta^n),d_t\theta^n)_{0,K}+(d_t(A(\Pi_{k+1}^{0,K}c_h^{n-1})\bPi_{K}^{0,K}\bbeta^n),\bPi_k^{0,K}d_t\bbeta^n)_{0,K}\\
		&{}=(d_t((A(\Pi_{k+1}^{0,K}c_h^{n-1})-A(c^n))\bPi_k^{0,K}\P_\bu \bu^n),\bPi_k^{0,K}d_t \bbeta^n)_{0,K}\\
		&\quad -\left(d_t\left((d(c^n)-d(\Pi_{k+1}^{0,K}c_h^{n-1}))\frac{\partial p^n}{\partial t}\right),d_t\theta^n\right)_{0,K}\\
		&\quad -\left(d_t\left(d(\Pi_{k+1}^{0,K}c_h^{n-1})\left(\frac{\partial p^n}{\partial t}-\frac{p^n-p^{n-1}}{\tau}\right)\right),d_t \theta^n\right)_{0,K}\\
		&\quad -(d_t(d(\Pi_{k+1}^{0,K}c_h^{n-1})d_t \eta^n),d_t \theta^n)_{0,K}\\
		&\quad -\nu_\cA^K(c_h^{n-1})S_\cA^K((1-\bPi_k^{0,K})\bbeta^n,(1-\bPi_k^{0,K})d_t\bbeta^n)/\tau\\
		&\quad +\nu_\cA^K(c_h^{n-2})S_\cA^K((1-\bPi_k^{0,K})\bbeta^{n-1},(1-\bPi_k^{0,K})d_t\bbeta^n)/\tau\\
		&\quad +(\nu_\cA^K(c_h^{n-1})-\nu_\cA^K(c^n))S_\cA^K((1-\bPi_k^{0,K})\P_\bu \bu^n,(1-\bPi_k^{0,K})d_t\bbeta^n)/\tau\\
		&\quad -(\nu_\cA^K(c_h^{n-2})-\nu_\cA^K(c^{n-1}))S_\cA^K((1-\bPi_k^{0,K})\P_\bu \bu^{n-1},(1-\bPi_k^{0,K})d_t\bbeta^n)/\tau.\label{eqn.dt}
	\end{align}
	Since
	\begin{align*}
(d_t(d(\Pi_{k+1}^{0,K}&c_h^{n-1})d_t\theta^n),d_t\theta^n)_{0,K} \ge  	\half d_t(d(\Pi_{k+1}^{0,K}c_h^{n-1})d_t\theta^n,d_t\theta^n)_{0,K}\\
&	-\half(d_t(d(\Pi_{k+1}^{0,K}c_h^{n-1}))d_t\theta^{n-1},d_t\theta^{n-1}) +(d_t(d(\Pi_{k+1}^{0,K}c_h^{n-1}))d_t\theta^{n-1},d_t\theta^n)_{0,K}
	\end{align*}
	and
	\begin{align*}
	(d_t(A(\Pi_{k+1}^{0,K}c_h^{n-1})\bPi_{K}^{0,K}\bbeta^n),\bPi_k^{0,K}d_t\bbeta^n)_{0,K}&=(A(\Pi_{k+1}^{0,K}c_h^{n-1})\bPi_{K}^{0,K}d_t\bbeta^n,\bPi_k^{0,K}d_t\bbeta^n)_{0,K}\\
	&\quad +(d_t(A(\Pi_{k+1}^{0,K}c_h^{n-1}))\bPi_{K}^{0,K}\bbeta^{n-1},\bPi_k^{0,K}d_t\bbeta^n)_{0,K},
	\end{align*}
	\eqref{eqn.dt} can be rewritten as
		\begin{align}
		&\half d_t(d(\Pi_{k+1}^{0,K}c_h^{n-1})d_t\theta^n,d_t\theta^n)_{0,K}+\cA_h(c_h^{n-1};d_t\bbeta^n,d_t\bbeta^n)_{0,K}\\
		&{}=-(d_t(A(\Pi_{k+1}^{0,K}c_h^{n-1}))\bPi_{K}^{0,K}\bbeta^{n-1},\bPi_k^{0,K}d_t\bbeta^n)_{0,K}\\
		&\quad 	+\half(d_t(d(\Pi_{k+1}^{0,K}c_h^{n-1}))d_t\theta^{n-1},d_t\theta^{n-1})_{0,K} -(d_t(d(\Pi_{k+1}^{0,K}c_h^{n-1}))d_t\theta^{n-1},d_t\theta^n)_{0,K}\\
		&\quad+ (d_t((A(\Pi_{k+1}^{0,K}c_h^{n-1})-A(c^n))\bPi_k^{0,K}\P_\bu \bu^n),\bPi_k^{0,K}d_t \bbeta^n)_{0,K}\\
		&\quad -\left(d_t\left((d(c^n)-d(\Pi_{k+1}^{0,K}c_h^{n-1}))\frac{\partial p^n}{\partial t}\right),d_t\theta^n\right)_{0,K}\\
		&\quad -\left(d_t\left(d(\Pi_{k+1}^{0,K}c_h^{n-1})\left(\frac{\partial p^n}{\partial t}-\frac{p^n-p^{n-1}}{\tau}\right)\right),d_t \theta^n\right)_{0,K}-(d_t(d(\Pi_{k+1}^{0,K}c_h^{n-1})d_t \eta^n),d_t \theta^n)_{0,K}\\
			&\quad +(\nu_\cA^K(c_h^{n-2})-\nu_\cA^K(c_h^{n-1}))S_\cA^K((1-\bPi_k^{0,K})\bbeta^{n-1},(1-\bPi_k^{0,K})d_t\bbeta^n)/\tau\\
				&\quad +(\nu_\cA^K(c_h^{n-1})-\nu_\cA^K(c^n))S_\cA^K((1-\bPi_k^{0,K})\P_\bu \bu^n,(1-\bPi_k^{0,K})d_t\bbeta^n)/\tau\\
			&\quad -(\nu_\cA^K(c_h^{n-2})-\nu_\cA^K(c^{n-1}))S_\cA^K((1-\bPi_k^{0,K})\P_\bu \bu^{n-1},(1-\bPi_k^{0,K})d_t\bbeta^n)/\tau.
	\label{eqn.dt.1}
	\end{align}
	Multiplying \eqref{eqn.dt.1} by $\tau$, summing over all $K \in \cT_h$, he using the coercivity of $\cA_h(c_h^{n-1};\bullet,\bullet)$, and summing $n$ from 1 to $N$, we obtain
			\begin{align}
		&C\|d_t\theta^N\|^2+C\tau \sum_{n=1}^{N}\|d_t\bbeta^n\|^2\\
		&{}\le \tau \sum_{n=1}^{N} -(d_t(A(\Pi_{k+1}^{0}c_h^{n-1}))\bPi_{K}^{0}\bbeta^{n-1},\bPi_k^{0}d_t\bbeta^n)\\
		&\quad 	+\tau \sum_{n=1}^{N}\half(d_t(d(\Pi_{k+1}^{0}c_h^{n-1}))d_t\theta^{n-1},d_t\theta^{n-1}) -\tau \sum_{n=1}^{N}(d_t(d(\Pi_{k+1}^{0}c_h^{n-1}))d_t\theta^{n-1},d_t\theta^n)\\
		&\quad +\tau \sum_{n=1}^{N}(d_t((A(\Pi_{k+1}^{0}c_h^{n-1})-A(c^n))\bPi_k^{0}\P_\bu \bu^n),\bPi_k^{0}d_t \bbeta^n)\\
		&\quad -\tau \sum_{n=1}^{N}\left(d_t\left((d(c^n)-d(\Pi_{k+1}^{0}c_h^{n-1}))\frac{\partial p^n}{\partial t}\right),d_t\theta^n\right)\\
		&\quad -\tau \sum_{n=1}^{N}\left(d_t\left(d(\Pi_{k+1}^{0}c_h^{n-1})\left(\frac{\partial p^n}{\partial t}-\frac{p^n-p^{n-1}}{\tau}\right)\right),d_t \theta^n\right)+\tau \sum_{n=1}^{N}(d_t(d(\Pi_{k+1}^{0}c_h^{n-1})d_t \eta^n),d_t \theta^n)\\
		&\quad +  \sum_{n=1}^{N}\sum_{K\in \cT_h}(\nu_\cA(c_h^{n-2})-\nu_\cA(c_h^{n-1}))S_\cA((1-\bPi_k^{0,K})\bbeta^{n-1},(1-\bPi_k^{0,K})d_t\bbeta^n)\\
&\quad +  \sum_{n=1}^{N}\sum_{K\in \cT_h}(\nu_\cA(c_h^{n-1})-\nu_\cA(c^n))S_\cA((1-\bPi_k^{0,K})\P_\bu \bu^n,(1-\bPi_k^{0,K})d_t\bbeta^n)\\
&\quad -  \sum_{n=1}^{N}\sum_{K\in \cT_h}(\nu_\cA(c_h^{n-2})-\nu_\cA(c^{n-1}))S_\cA((1-\bPi_k^{0,K})\P_\bu \bu^{n-1},(1-\bPi_k^{0,K})d_t\bbeta^n).		
	=:\sum_{i=1}^{10} T_i.
	\label{eqn.dt.ti}
	\end{align}
		The Lipschitz continuity of $A(\bullet)$ leads to $$\|A(\Pi_{k+1}^0 c_h^{n-1})-A(\Pi_{k+1}^0 c_h^{n-2})\|_{0,\infty,\Omega}\lesssim \| \Pi_{k+1}^0 c_h^{n-1}-\Pi_{k+1}^0 c_h^{n-2}\|_{0,\infty,\Omega}.$$ Elementary algebra and the definition of $\xi$ and $\zeta$ read
		\begin{align}
		\Pi_{k+1}^0 c_h^{n-1}-\Pi_{k+1}^0 &c_h^{n-2}
			=-(\Pi_{k+1}^{0}(d_t \xi^{n-1}+d_t\zeta^{n-1})\tau)+\Pi_{k+1}^0(c^{n-1}-c^{n-2}).\label{eqn.Pidiff}
		\end{align}
		This, an inverse estimate, the boundedness of $\Pi_{k+1}^0$, Corollary~\ref{cor.ctPcterror} for $d_t\zeta^{n-1} \approx \frac{\partial}{\partial t}(\zeta^{n-1})$, and Taylors expansion provide
		\begin{align}
			\| 		\Pi_{k+1}^0 c_h^{n-1}-\Pi_{k+1}^0 &c_h^{n-2}\|_{0,\infty,\Omega} \lesssim \tau(1+h^{k+1}+\|d_t \xi^{n-1}\|_{0,\infty,\Omega}).\label{eqn.Pidiffinfty}
		\end{align}
	This, the \Holders inequality, the boundedness of  $\bPi_{k}^{0,K}$, and Youngs inequality show
	\begin{equation}
		|T_1| \le\epsilon \tau \sum_{n=1}^N \|d_t \bbeta^n\|^2+C \tau\sum_{n=1}^N\big(1+\|d_t \xi^{n-1}\|^2_{0,\infty,\Omega}+h^{2(k+1)}\big)\|\beta^{n-1}\|^2.
	\end{equation}
	The \Holders inequality, the boundedness of $d(\bullet)$, \eqref{eqn.Pidiffinfty}, and Youngs inequality provide
	\begin{equation*}
		|T_2|+|T_3| \le C \tau \sum_{n=1}^N\big(1+\|d_t\xi^{n-1}\|_{0,\infty,\Omega}^2+h^{2(k+1)}\big)\|d_t \theta^{n-1}\|^2+C\tau\sum_{n=1}^N\|d_t \theta^n\|^2.
	\end{equation*}
	A simple manipulation leads to
	\begin{align*}
		T_4 &=\sum_{n=1}^N\left({(A(\Pi_{k+1}^{0}c_h^{n-1})-A(c^n))}(\bPi_k^{0}(\P_\bu \bu^n-\P_\bu \bu^{n-1})),\bPi_k^0 d_t \beta^n\right)\\
		&\quad +\sum_{n=1}^N\left({(A(\Pi_{k+1}^{0}c_h^{n-1})-A(c^n)-(A(\Pi_{k+1}^{0}c_h^{n-2})-A(c^{n-1})))}\bPi_k^{0}\P_\bu \bu^{n-1},\bPi_k^0 d_t \beta^n\right)\\
		&=:I_1+I_2.\label{eqn.i1i2}
	\end{align*}
	The Lipschitz continuity of $A(\bullet)$, an introduction of $\Pi_{k+1}^0 c^n$, Lemma~\ref{lem.approx}, inverse estimate, the continuity of $L^2$ projector, Lemma~\ref{lem.projectionup}, and Taylors expansion read 
	\begin{align}
		\|A(\Pi_{k+1}^{0}&c_h^{n-1})-A(c^n)\|_{0,\infty,\Omega}\\
		& \lesssim  h^{k+1}+\|\Pi_{k+1}^{0}\xi^{n-1}\|_{{0,\infty,\Omega}} +\|\Pi_{k+1}^{0}(\P_c c^{n-1}-c^{n-1})\|_{{0,\infty,\Omega}}+\|\Pi_{k+1}^{0}(c^{n-1}-c^{n})\|_{{0,\infty,\Omega}} \\
		& \lesssim \|\xi^{n-1}\|_{{0,\infty,\Omega}} +h^{k+1}+\tau.\label{eqn.infty.1}
	\end{align}
	This, \Holders inequality, the boundedness $\bPi_{k}^0$ and $\P_\bu$, Taylors expansion, and Youngs inequality show
	\begin{align}
		I_1 
		&\le\epsilon \tau\sum_{n=1}^N\|d_t\bbeta^n\|^2+C\tau\left(\sum_{n=1}^N\|\xi^{n-1}\|_{{0,\infty,\Omega}}^2 +h^{2(k+1)}+\tau^2\right).
	\end{align}
The Lipschitz continuity of $A(\bullet)$ leads to
	\begin{align}
		A(\Pi_{k+1}^{0}&c_h^{n-1})-A(c^n)-(A(\Pi_{k+1}^{0}c_h^{n-2})+A(c^{n-1}))\\
		&=\frac{\partial A}{\partial c}( \overline{\Pi_{k+1}^0c}_h^{n-1})(\Pi_{k+1}^{0}c_h^{n-1}-\Pi_{k+1}^{0}c_h^{n-2})-\frac{\partial A}{\partial c}( \overline{c}^{n-1})(c^{n-1}-c^{n-2})\\
		&=\frac{\partial A}{\partial c}( \overline{\Pi_{k+1}^0 c}_h^{n-1})(\Pi_{k+1}^{0}c_h^{n-1}-\Pi_{k+1}^{0}c_h^{n-2}-c^{n}+c^{n-1})\\
		&\quad +(\frac{\partial A}{\partial c}( \overline{\Pi_{k+1}^0 c}_h^{n-1})-\frac{\partial A}{\partial c}( \overline{ c}^{n-1}))(c^{n-1}-c^{n-2})	\\
		&=\frac{\partial A}{\partial c}( \overline{\Pi_{k+1}^0c}_h^{n-1})(\Pi_{k+1}^{0}c_h^{n-1}-\Pi_{k+1}^{0}c_h^{n-2}-c^{n}+c^{n-1})\\
		&\quad +\frac{\partial^2 A}{\partial c^2}(\widehat{c}^{n-1})( \overline{\Pi_{k+1}^0c}_h^{n-1}- \overline{c}^{n-1})(c^{n-1}-c^{n-2}), \label{eqn.A.1}
	\end{align}
where $\overline{c}^{n-1} \in (c^{n-2},c^{n-1})$, $\overline{\Pi_{k+1}^0c}_h^{n-1} \in ({\Pi_{k+1}^0c}_h^{n-2},{\Pi_{k+1}^0c}_h^{n-1})$, and $\widehat{c}^{n-1} \in (\overline{\Pi_{k+1}^0c}_h^{n-1},\overline{c}^{n-1})$ respectively. Elementary algebra and the definition of $\xi$ and $\zeta$ read
	\begin{align}
		\Pi_{k+1}^{0}c_h^{n-1}-&\Pi_{k+1}^{0}c_h^{n-2}-c^{n}+c^{n-1}\\
		&=\Pi_{k+1}^{0}c_h^{n-1}-\Pi_{k+1}^{0}c_h^{n-2}-\Pi_{k+1}^{0}c^{n}+\Pi_{k+1}^{0}c^{n-1}-(d_t(c^n-\Pi_{k+1}^0 c^n)\tau)\\
		&=-(\Pi_{k+1}^{0}(d_t \xi^{n-1}+d_t\zeta^{n-1})\tau)-(d_t(c^n-\Pi_{k+1}^0 c^n)\tau)-\Pi_{k+1}^0(c^{n}-2c^{n-1}+c^{n-2}),
	\end{align}
	This, an inverse estimate, the boundedness of $\Pi_{k+1}^0$, Corollary \ref{cor.ctPcterror} for $d_t\zeta^{n-1}$, Taylors expansion, and Lemma~\ref{lem.approx} provide
	\begin{align}\label{eqn.A2.1}
		\| 	\Pi_{k+1}^{0}c_h^{n-1}-\Pi_{k+1}^{0}c_h^{n-2}-c^{n}+c^{n-1}\|_{0,\infty,\Omega} \lesssim \tau(\tau+h^{k+1}+\|d_t \xi^{n-1}\|_{0,\infty,\Omega}).
	\end{align}
		Taylors expansion, an introduction of $\Pi_{k+1}^0 \P_c c^{n-1}$ and $\Pi_{k+1}^0 c^{n-1}$, boundedness of $\Pi_{k+1}^0$, inverse estimate, Lemma~\ref{lem.projectionup}, and Lemma~\ref{lem.approx} show
	\begin{align}
		\|\overline{\Pi_{k+1}^0c}_h^{n-1}- \overline{c}^{n-1} \|_{0,\infty,\Omega} &\lesssim \|{\Pi_{k+1}^0c}_h^{n-1}-{c}^{n-1} \|_{0,\infty,\Omega}+\tau\lesssim  \|\xi^{n-1}\|_{{0,\infty,\Omega}} +h^{k+1}+\tau.\label{eqn.A1.1}
	\end{align}
	A combination of \eqref{eqn.A1.1} and \eqref{eqn.A2.1} in \eqref{eqn.A.1} together with boundedness of $\Pi_{k+1}^0$ and Taylors expansion yields
	\begin{align*}
		\|A(\Pi_{k+1}^{0}c_h^{n-1})-&A(c^{n})-(A(\Pi_{k+1}^{0}c_h^{n-2})-A(c^{n-1}))\|_{0,\infty,\Omega}\\
		&\lesssim \tau(\tau+h^{k+1}+\|d_t \xi^{n-1}\|_{0,\infty,\Omega}+\|\xi^{n-1}\|_{0,\infty,\Omega}).
	\end{align*}
	Consequently,
	\[I_2 \le  \epsilon \tau\sum_{n=1}^N\|d_t \bbeta^n\|^2+C\tau\left(\tau^2+h^{2(k+1)}+\sum_{n=1}^N(\|d_t \xi^{n-1}\|_{0,\infty,\Omega}^2+\|\xi^{n-1}\|_{0,\infty,\Omega}^2)\right).\]
	A combination of $I_1$ and $I_2$ in $T_4$ shows
	$$ T_4\le \epsilon \tau\sum_{n=1}^N\|d_t\bbeta^n\|^2+C\tau\left(\tau^2+h^{2(k+1)}+\sum_{n=1}^N(\|d_t \xi^{n-1}\|_{0,\infty,\Omega}^2+\|\xi^{n-1}\|_{0,\infty,\Omega}^2)\right).$$
Arguments analogous to $T_4$ leads to
\[|T_5| \le  C\tau\left(\tau^2+h^{2(k+1)}+\sum_{n=1}^N(\|d_t \xi^{n-1}\|_{0,\infty,\Omega}^2+\|\xi^{n-1}\|_{0,\infty,\Omega}^2+\|d_t\theta^n\|^2)\right).\]
The term $T_6$ can be rewritten as
\begin{align*}
	T_6=&\tau\sum_{n=1}^N\left(\frac{d(\Pi_{k+1}^{0}c_h^{n-1})-d(\Pi_{k+1}^{0}c_h^{n-2)}}{\tau}\left(\frac{\partial p^n}{\partial t}-\frac{p^n-p^{n-1}}{\tau}\right),d_t \theta^n\right)\\
	&\quad + \tau\sum_{n=1}^N\left(\frac{d(\Pi_{k+1}^{0}c_h^{n-2})}{\tau}\left(\frac{\partial (p^n-p^{n-1})}{\partial t}-\frac{p^n-2p^{n-1}+p^{n-2}}{\tau}\right),d_t \theta^n\right).
\end{align*}
The \Holders inequality, boundedness of $d(\bullet)$, \eqref{eqn.Pidiffinfty}, and Taylors expansion imply
\[|T_6| \le C \tau \sum_{n=1}^N(\|d_t \theta^n\|^2+\|d_t \xi^{n-1}\|_{0,\infty,\Omega}^2)+C(\tau^2+h^{2(k+1)}).\]
Analogous arguments and Corollary \ref{cor.up} lead to 
\[|T_7| \le C \tau \sum_{n=1}^N(\|d_t \theta^n\|^2+\|d_t \xi^{n-1}\|^2)+C(\tau^2+h^{2(k+1)}).\]
The \Holders inequality, \eqref{eqn.Pidiffinfty}, the boundedness of  $\bPi_{k}^{0,K}$, and Youngs inequality show
\begin{equation}
	|T_8| \le\epsilon \tau \sum_{n=1}^N \|d_t \bbeta^n\|^2+C \tau\sum_{n=1}^N(1+\|d_t \xi^{n-1}\|^2_{0,\infty,\Omega}+h^{2(k+1)})\|\beta^{n-1}\|^2.
\end{equation}
Arguments analogous to the estimate of $T_4$ provides
\[|T_9|+|T_{10}| \le \tau\sum_{n=1}^N\|d_t\bbeta^n\|^2+C\tau\left(\tau^2+h^{2(k+1)}+\sum_{n=1}^N(\|d_t \xi^{n-1}\|_{0,\infty,\Omega}^2+\|\xi^{n-1}\|_{0,\infty,\Omega}^2)\right).\]
A combination of $T_1$ to $T_{10}$ in \eqref{eqn.dt.ti} leads to
\begin{align*}
		\|d_t&\theta^N\|^2+\tau \sum_{n=1}^{N}\|d_t\bbeta^n\|^2\\
		 &\le C\bigg(\tau^2+h^{2(k+1)}+ \tau \sum_{n=1}^N(\|d_t \theta^n\|^2+\|d_t \xi^{n-1}\|_{0,\infty,\Omega}^2+\|\xi^{n-1}\|_{0,\infty,\Omega}^2+\|\bbeta^{n-1}\|^2)\bigg).
\end{align*}
\end{proof}
\subsection{Error estimates for concentration}
This section deals with the error estimates for concentration and is followed by the main result.

\smallskip
\noindent Observe that
\begin{align}\label{eqn.phiN}
	\|\phi^N\|^2-\|\phi^0\|^2 \le C\tau\sum_{n=1}^N\|\phi^n\|^2+\epsilon \tau \sum_{n=1}^N \|d_t \phi^n\|^2.
\end{align}
\begin{thm}\label{thm.c}
	 Let $(\bu,p,c) $ solves \eqref{eqn.weak} at time $t=t_n$ and let $(c_h^n,u_h^n,p_h^n) \in Z_h \times \bV_h \times Q_h$ be the solution to \eqref{eqn.fullydiscrete_uhph}-\eqref{eqn.fullydiscreteq_zh}. Suppose the space and time discretisation satisfy $\tau=\mathcal{O}(h^{k+1})$. Then, under the regularity assumption $(R)$,
\begin{align*}
	\tau\sum_{n=1}^N &\|d_t \xi^n\|^2+|\xi^N|_{1,\cT_h}^2 \le 	C(\tau^2+h^{2(k+1)}+\tau \sum_{n=1}^N(\|\xi^n\|^2+|\xi^n|_{1,\cT_h}^2+\|d_t \theta^n\|^2+\|\beta^n\|^2)).
\end{align*}
\end{thm}
\begin{proof}
	The proof is divided into three steps. 
	
	\medskip
	
\noindent	\textbf{Step 1:} 
	Let $c_h^n-\P_c c^n=\nu^n \in Z_h$. The discrete fully formulation \eqref{eqn.fullydiscreteq_zh} and the definition of $\P_c c^n$ in \eqref{eqn.projectionP} provide, for $z_h \in Z_h$,
	\begin{align}
		\cM_h&\left(\frac{\xi^n-\xi^{n-1}}{\tau},z_h\right)+\cD_h(\bu_h^{n};\xi^n,z_h)\\
		&=\cM_h\left(\frac{\P_c c^n-\P_c c^{n-1}}{\tau},z_h\right)-\cM_h\left(\frac{c_h^n-c_h^{n-1}}{\tau},z_h\right)+\cD_h(\bu_h^{n};\xi^n,z_h)\\
		&=\cM_h\left(\frac{\P_c c^n-\P_c c^{n-1}}{\tau},z_h\right)+	K_h\left(c_h^{n};\frac{p_h^{n}-p_h^{n-1}}{\tau},z_h\right)+\Theta_h(\bu_h^n,c_h^{n};z_h)+(q^nc_h^{n},z_h)_h\\
		&\qquad -(q^{n}\hc^{n},z_h)_h+\cD_h(\bu_h^{n};\P_c c^n,z_h)\\
		&=\cM_h\left(\frac{\P_c c^n-\P_c c^{n-1}}{\tau},z_h\right)+	K_h\left(c_h^{n};\frac{p_h^{n}-p_h^{n-1}}{\tau},z_h\right)+\Theta_h(\bu_h^n,c_h^{n};z_h)+(q^nc_h^{n},z_h)_h\\
		&\qquad -(q^{n}\hc^{n},z_h)_h +\cD_h(\bu_h^{n};\P_c c^n,z_h)-\cD_h^{\bu^{n}}(\P_c c^n,z_h)+\cD_h^{\bu^{n}}(\P_c c^n,z_h)\\
		&=\cM_h\left(\frac{\P_c c^n-\P_c c^{n-1}}{\tau},z_h\right)+	K_h\left(c_h^{n};\frac{p_h^{n}-p_h^{n-1}}{\tau},z_h\right)+\Theta_h(\bu_h^n,c_h^{n};z_h)+(q^nc_h^{n},z_h)_h\\
	&\qquad -(q^{n}\hc^{n},z_h)_h +\cD_h(\bu_h^{n};\P_c c^n,z_h)-\cD_h^{\bu^{n}}(\P_c c^n,z_h)-\Theta_h^{\bu^n}(\P_c c^n,z_h)-\lambda(\P_c c^n,z_h)_h\\
	&\qquad +\cD_{\pw}^{\bu^n}(c^n,z_h)+\Theta_{\pw}^{\bu^n}(c^n,z_h)+\lambda(c^n,z_h)-\cN_h(\bu^n;c^n,z_h).\label{eqn1}
	\end{align}
	An integration by parts yields
	\begin{align}
		\cD_{\pw}^{\bu^{n}}(c^n,z_h)-\sum_{e \in \cE_h}\int_e(D(\bu^n)\nabla c^n \cdot \bn_e )\jump{z_h}\ds&=(-\divc(D(\bu^n)\nabla c^n),z_h).
	\end{align}
	 This in \eqref{eqn1} together with the definition of $\cN_h(\bu^n;c^n,z_h)$ and $-\divc(D(\bu^n)\nabla c^n)+\bu^n\cdot \nabla c^n=q^{n}(\hc^n-c^n)-\phi \frac{\partial c^n}{\partial t}-b(c) \frac{\partial p^n}{\partial t}$ from \eqref{eqn.model} show
	\begin{align}
		\cM_h&\left(\frac{\xi^n-\xi^{n-1}}{\tau},z_h\right)+\cD_h(\bu_h^n;\xi^n,z_h)\\
		=&	\left(\cM_h\left(\frac{\P_c c^n-\P_c c^{n-1}}{\tau},z_h\right)-\cM\left(\phi \frac{\partial c^n}{\partial t},z_h\right)\right)+	\left(K_h\left(c_h^{n};\frac{p_h^{n}-p_h^{n-1}}{\tau},z_h\right)-K\left(c^{n};\frac{\partial p^n}{\partial t},z_h\right)\right)\\
	&\qquad +(\Theta_h(\bu_h^{n},c_h^{n};z_h)-\Theta_h^{\bu^{n}}(\P_c c^n,z_h)) +((q^nc_h^{n},z_h)_h-(q^nc^{n},z_h))\\
		&\qquad +(\cD_h(\bu_h^{n};\P_c c^n,z_h)-\cD_h^{\bu^{n}}(\P_c c^n,z_h)) +\lambda((c^n,z_h)-(\P_c c^n,z_h)_h)\\
		&\qquad+((q^{n}\hc^n,z_h)-(q^{n}\hc^{n},z_h)_h).
	\end{align}
Choosing $z_h=d_t\xi^n$ and multiplying  both sides with $\tau$ and then summing for $n$ from 1 to $N$, we obtain
	\begin{align}
	\tau \sum_{n=1}^N&\cM_h\left(d_t\xi^n,d_t\xi^n\right)+\tau \sum_{n=1}^N\cD_h(\bu_h^n;\xi^n,d_t\xi^n)\\
	=&	\tau \sum_{n=1}^N\left(\cM_h\left(\frac{\P_c c^n-\P_c c^{n-1}}{\tau},d_t\xi^n\right)-\cM\left(\phi \frac{\partial c^n}{\partial t},d_t\xi^n\right)\right)\\
	&\qquad +\tau \sum_{n=1}^N	\left(K_h\left(c_h^{n};\frac{p_h^{n}-p_h^{n-1}}{\tau},d_t\xi^n\right)-K\left(c^{n};\frac{\partial p^n}{\partial t},d_t \xi^n\right)\right)\\
	&\qquad +\tau \sum_{n=1}^N(\Theta_h(\bu_h^{n},c_h^{n};d_t\xi^n)-\Theta_h^{\bu^{n}}(\P_c c^n,d_t\xi^n)) +\tau \sum_{n=1}^N((q^nc_h^{n},d_t\xi^n)_h-(q^nc^{n},d_t \xi^n))\\
	&\qquad +\tau \sum_{n=1}^N(\cD_h(\bu_h^{n};\P_c c^n,d_t\xi^n)-\cD_h^{\bu^{n}}(\P_c c^n,d_t\xi^n)) \\
	&\qquad+\tau \sum_{n=1}^N\lambda((c^n,d_t\xi^n)-(\P_c c^n,d_t\xi^n)_h)+\tau \sum_{n=1}^N((q^{n}\hc^n,d_t\xi^n)-(q^{n}\hc^{n},d_t\xi^n)_h).\label{eqna1a2a3a4a5}
\end{align}	
The coercivity property of $\cM_h(\bullet,\bullet)$ in Lemma~\ref{lem.propertiesdiscrete}.b shows 
\begin{equation}\label{eqn.Mh}
\tau \sum_{n=1}^N\cM_h\left(d_t\xi^n,d_t\xi^n\right) \gtrsim \tau\sum_{n=1}^N \|d_t \xi^n\|^2.
\end{equation}
Since 
\begin{align*}
	d_t((D(\bPi_{k}^{0,K}\bu_h^n)(\Pi_{k+1}^{0,K}\nabla \xi^n),(\Pi_{k+1}^{0,K}\nabla \xi^n)))&=(d_t(D(\bPi_{k}^{0,K}\bu_h^n))(\Pi_{k+1}^{0,K}\nabla \xi^n),(\Pi_{k+1}^{0,K}\nabla \xi^n))\\
	&\quad +2(D(\bPi_{k}^{0,K}\bu_h^n)(\Pi_{k+1}^{0,K}\nabla \xi^n),(\Pi_{k+1}^{0,K}\nabla d_t \xi^n)),
\end{align*}
we have
\begin{align*}
	(D(\bPi_{k}^{0,K}\bu_h^n)(\Pi_{k+1}^{0,K}\nabla \xi^n),\Pi_{k+1}^{0,K}\nabla d_t \xi^n)&=\frac{1}{2\tau}\bigg((D(\bPi_{k}^{0,K}\bu_h^{n-1})(\Pi_{k+1}^{0,K}\nabla \xi^n),(\Pi_{k+1}^{0,K}\nabla \xi^n))\\
	&\quad -(D(\bPi_{k}^{0,K}\bu_h^{n-1})(\Pi_{k+1}^{0,K}\nabla \xi^{n-1}),(\Pi_{k+1}^{0,K}\nabla \xi^{n-1}))\bigg).
\end{align*}
This and \eqref{eqn.Mh} in \eqref{eqn.a1a2a3a4a5} together with the fact that $\xi^0=0$ result in
	\begin{align}
\tau\sum_{n=1}^N &\|d_t \xi^n\|^2+| \xi^N|_{1,\cT_h}^2\lesssim 	\tau \sum_{n=1}^N\left(\cM_h\left(\frac{\P_c c^n-\P_c c^{n-1}}{\tau},d_t\xi^n\right)-\cM\left(\phi \frac{\partial c^n}{\partial t},d_t\xi^n\right)\right)\\
	&\qquad +\tau \sum_{n=1}^N	\left(K_h\left(c_h^{n};\frac{p_h^{n}-p_h^{n-1}}{\tau},d_t\xi^n\right)-K\left(c^{n};\frac{\partial p^n}{\partial t},d_t \xi^n\right)\right)\\
	&\qquad +\tau \sum_{n=1}^N(\Theta_h(\bu_h^{n},c_h^{n};d_t\xi^n)-\Theta_h^{\bu^{n}}(\P_c c^n,d_t\xi^n)) +\tau \sum_{n=1}^N((q^nc_h^{n},d_t\xi^n)_h-(q^nc^{n},d_t \xi^n))\\
	&\qquad +\tau \sum_{n=1}^N(\cD_h(\bu_h^{n};\P_c c^n,d_t\xi^n)-\cD_h^{\bu^{n}}(\P_c c^n,d_t\xi^n)))\\
	&\qquad +\tau \sum_{n=1}^N\lambda((c^n,d_t\xi^n)-(\P_c c^n,d_t\xi^n)_h+\tau \sum_{n=1}^N((q^{n}\hc^n,d_t\xi^n)-(q^{n}\hc^{n},d_t\xi^n)_h)\\
	&\qquad +\tau \sum_{n=1}^N\nu_\cD^K(\bu_h^n)S_\cD^K((I-\Pi_{k+1}^{\nabla,K})\xi^n,(I-\Pi_{k+1}^{\nabla,K})d_t\xi^n)=:\sum_{i=1}^8A_i.\label{eqna1.a8}
\end{align}	

\medskip
	
\noindent	\textbf{Step 2:} Estimation of $A_1,\cdots,A_8.$\\
	
\noindent The definition of $\cM(\bullet,\bullet)$ and $\cM_h(\bullet,\bullet)$ in \eqref{defn.bilinear} and \eqref{defn.Mh}, orthogonality and continuity properties of $\Pi_{k+1}^{0,K}$, Cauchy Schwarz inequality, $S_\cM^K(z_h,\tilde{z}_h)\le M_1^\cM\|z_h\|_{0,K}\|\tilde{z}_h\|_{0,K}$ from \eqref{eqn.SM} prove
	\begin{align*}
	\cM\left(\phi \frac{\partial c^n}{\partial t},d_t\xi^n\right)&-\cM_h\left(\frac{\P_c c^n-\P_c c^{n-1}}{\tau},d_t\xi^n\right)\\
		&=\sum_{K \in \cT_h}\bigg[\left(\phi \frac{\partial c^n}{\partial t},d_t\xi^n\right)_{0,K}-\left(\phi \Pi_{k+1}^{0,K}\left(\frac{\P_c c^n-\P_c c^{n-1}}{\tau}\right),\Pi_{k+1}^{0,K}d_t\xi^n\right)_{0,K}\\
		&\qquad - \nu_\cM^K(\phi) S_\cM^K\left((I-\Pi_{k+1}^{0,K})\left(\frac{\P_c c^n-\P_c c^{n-1}}{\tau}\right),(I-\Pi_{k+1}^{0,K})d_t\xi^n\right)\bigg]\\
		&=\sum_{K \in \cT_h}\bigg[\left(\phi \frac{\partial c^n}{\partial t},d_t\xi^n\right)_{0,K}-\left(\Pi_{k+1}^{0,K}\left(\phi \Pi_{k+1}^{0,K}\left(\frac{\P_c c^n-\P_c c^{n-1}}{\tau}\right)\right),d_t\xi^n\right)_{0,K}\\
		&\qquad - \nu_\cM^K(\phi) S_\cM^K\left((I-\Pi_{k+1}^{0,K})\left(\frac{\P_c c^n-\P_c c^{n-1}}{\tau}\right),(I-\Pi_{k+1}^{0})d_t\xi^n\right)\bigg]\\
		&\le \eta\bigg[\left\|\phi \frac{\partial c^n}{\partial t}-\Pi_{k+1}^{0}\left(\phi \Pi_{k+1}^{0}\left(\frac{\P_c c^n-\P_c c^{n-1}}{\tau}\right)\right)\right\|\| d_t\xi^n\|\\
		&\qquad + \left\|(I-\Pi_{k+1}^{0})\left(\frac{\P_c c^n-\P_c c^{n-1}}{\tau}\right)\right\|\|d_t\xi^n\|\bigg]=:\eta(A_{1,1}+A_{1,2})\|d_t\xi^n\|.\label{eqna1}
	\end{align*}
	The continuity of the $L^2$ projector $\Pi_{k+1}^{0}$, boundedness of $\phi$, and Lemma \ref{lem.approx} read
	\begin{align}
		A_{1,1}&\le \left\|(I-\Pi_{k+1}^{0})\phi \frac{\partial c^n}{\partial t}\right\|+\left\|\Pi_{k+1}^{0}\left(\phi \frac{\partial c^n}{\partial t}-\phi\Pi_{k+1}^{0} \frac{\partial c^n}{\partial t}\right)\right\|\\
		&\qquad +\left\|\Pi_{k+1}^{0}\left(\phi \Pi_{k+1}^{0}\left(\frac{\partial c^n}{\partial t}-\frac{\P_c c^n-\P_c c^{n-1}}{\tau}\right)\right)\right\|\\
		&\le \eta \bigg[h^{k+2}\left(\left|\phi \frac{\partial c^n}{\partial t}\right|_{k+2,\cT_h}+\left|\frac{\partial c^n}{\partial t}\right|_{k+2,\cT_h}\right)+\left\| \frac{\partial c^n}{\partial t}-\frac{\P_c c^n-\P_c c^{n-1}}{\tau}\right\|\bigg].\label{eqna11}
	\end{align} 
	Analogous arguments provides
	\begin{align}
		A_{2,1}&= \left\|(I-\Pi_{k+1}^{0})\left(\frac{\P_c c^n-\P_c c^{n-1}}{\tau}-\frac{\partial c^n}{\partial t}\right)\right\|+\left\|(I-\Pi_{k+1}^{0})\frac{\partial c^n}{\partial t}\right\|\\
		&\le \eta\bigg[ \left\|\frac{\P_c c^n-\P_c c^{n-1}}{\tau}-\frac{\partial c^n}{\partial t} \right\|+h^{k+2}\left|\frac{\partial c^n}{\partial t}\right|_{k+2,\cT_h}\bigg].
	\end{align}
	This and \eqref{eqna11} in \eqref{eqna1} together with Lemma \ref{lem.cPcnuuhn} and Youngs inequality result in
	\begin{align}
		A_1&\le C(h^{2(k+2)}+\tau^2) +\epsilon \tau\sum_{i=1}^N\|d_t\xi^n\|^2.\label{eqna1new}
	\end{align}
	A simple manipulation leads to
\begin{align*}
A_2&=	\tau \sum_{n=1}^N	\bigg[K_h\left(c_h^{n};\frac{p_h^{n}-p_h^{n-1}}{\tau}-\frac{p^{n}-p^{n-1}}{\tau},d_t\xi^n\right)+K_h\left(c_h^{n};\frac{p^{n}-p^{n-1}}{\tau}-\frac{\partial p^n}{\partial t},d_t\xi^n\right)\\
&\qquad +\left(K_h\left(c_h^{n};\frac{\partial p^n}{\partial t},d_t\xi^n\right)-K\left(c^{n};\frac{\partial p^n}{\partial t},d_t \xi^n\right)\right)\bigg]=:A_{2,1}+A_{2,2}+A_{2,3}.
	\end{align*}
The \Holders inequality, Taylors expansion, Corollary~\ref{cor.up}, continuity of $L^2$ projector, and Youngs inequality provide
\[A_{2,1}\le C\Big(h^{2(k+1)}+\tau \sum_{n=1}^N\|d_t \theta^n\|^2\Big)+\epsilon \tau \sum_{n=1}^N \|d_t \xi^n\|^2.\]
The \Holder inequality, Taylors expansion, and Youngs inequality leads to $A_{2,2}\le C\tau^2+\epsilon \tau \sum_{n=1}^N \|d_t \xi^n\|^2.$
The definition of the projection operator $\Pi_{k+1}^0$, \Holders inequality, the boundedness of $b(\bullet)$, an introduction of $\Pi_{k+1}^{0}c^n$ and $\P_c c^n$, Lemma~\ref{lem.approx}, and Lemma~\ref{lem.cPcerror} imply
\begin{align*}
	A_{2,3}&=\tau \sum_{n=1}^N\left((b(\Pi_{k+1}^0c_h^{n})-b(c^{n}))\frac{\partial p^n}{\partial t},d_t \xi^n\right) \le C\tau \sum_{n=1}^N\|\Pi_{k+1}^0 c_h^n-c^n\|\|d_t \xi^n\|\\
	&\le C\Big(h^{2(k+2)}+\tau \sum_{n=1}^N\|\xi^n\|^2\Big)+\epsilon \tau \sum_{n=1}^N \|d_t \xi^n\|^2.
\end{align*}
A combination of $A_{2,1}-A_{2,3}$ in $A_2$ shows
\begin{align*}
	A_2&\le C\Big(\tau^2+h^{2(k+1)}+\tau \sum_{n=1}^N\|d_t \theta^n\|^2+\tau \sum_{n=1}^N\|\xi^n\|^2\Big)+\epsilon \tau \sum_{n=1}^N \|d_t \xi^n\|^2.
\end{align*}
The \Holders inequality, continuity of $L^2$ projector,  Lemma~\ref{lem.projectionup}.a, and Lemma~\ref{lem.approx} read 
\begin{align*}
A_3&=\tau \sum_{n=1}^N\big[(\bPi_{k}^0\bu_h^{n} \cdot \bPi_{k}^0 (\nabla c_h^{n}),\Pi_{k+1}^0d_t\xi^n)-(\bu^{n} \cdot \bPi_{k}^{0}(\nabla \P_c c^n),\Pi_{k+1}^0d_t\xi^n)\big]\\
&=\tau \sum_{n=1}^N\big[(\bPi_{k}^0(\bu_h^{n}-\bu^n) \cdot \bPi_{k}^0 (\nabla c_h^{n}),\Pi_{k+1}^0d_t\xi^n)+((\bPi_{k}^0\bu^n-\bu^n) \cdot \bPi_{k}^0 (\nabla c_h^{n}),\Pi_{k+1}^0d_t\xi^n)\\
&\qquad +(\bu^{n} \cdot \bPi_{k}^{0}\nabla (c_h^n-\P_c c^n),\Pi_{k+1}^0d_t\xi^n)\big]\\
&\le \epsilon \tau \sum_{n=1}^N \|d_t \xi^n\|^2+ C\Big(h^{2(k+1)}+\tau \sum_{n=1}^N(\|\beta^n\|^2+\|\nabla \xi^n\|^2)\Big).
\end{align*}
The \Holder inequality, an introduction of $\Pi_{k+1}^0$, Lemma~\ref{lem.approx},~\ref{lem.cPcerror}.b, continuity and orthogonality property of $\Pi_{k+1}$ together with $qc \in H^{k+1}(\cT_h)$ and Youngs inequality show
\begin{align*}
	A_4&=\tau \sum_{n=1}^N(q^n(\Pi_{k+1}^0c_h^{n}-c^n),\Pi_{k+1}^0d_t\xi^n)+\tau \sum_{n=1}^N((\Pi_{k+1}^0-1)q^nc^n,d_t\xi^n)\\
	&\le C\Big(h^{2(k+1)}+\tau \sum_{n=1}^N\|\xi^n\|^2\Big)+\epsilon \tau \sum_{n=1}^N \|d_t \xi^n\|^2.
\end{align*}
 A simple manipulation leads to 
\begin{align}
	A_5
	&=\tau\sum_{n=1}^N\big[( (D(\bPi_k^{0}\bu_h^n)-D(\bu^n)) \bPi_{k}^{0}(\nabla \P_c c^n),\bPi_{k}^{0}(\nabla d_t \xi^n))\\
	&\qquad + (\nu_\cD(\bu_h^{n})-\nu_\cD(\bu^n))S_\cD((I-\Pi_{k+1}^{\nabla})\P_c c^n,(I-\Pi_{k+1}^{\nabla})d_t\xi^n)\big]\\
	&= \tau\sum_{n=1}^N\big[(d_t((D(\bu^n)-D(\bPi_k^{0}\bu_h^n)) \bPi_{k}^{0}(\nabla \P_c c^n)),\bPi_{k}^{0}\nabla \xi^n)\\
	&\qquad + (\nu_\cD(\bu_h^{n})-\nu_\cD(\bu^n))S_\cD((I-\Pi_{k+1}^{\nabla})\P_c c^n,(I-\Pi_{k+1}^{\nabla})d_t\xi^n)\big]\\
	&\qquad +((D(\bPi_k^{0}\bu_h^N)-D(\bu^N)) \bPi_{k}^{0}(\nabla \P_c c^N),\bPi_{k}^{0} \nabla\xi^N).
\end{align}
The generalised \Holder inequality, Lipschitz continuity of $D(\bullet)$, Lemma~\ref{lem.approx}, $\|\bPi_k^{0,K}\nabla \P_c c\|_{\infty,K} \lesssim 1$ from \cite[(62)]{Veiga_miscibledisplacement_2021}, Lemma~\ref{lem.projectionup}.a, Corollary~\ref{cor.up}, continuity of the projection operator, the definition of $\nu_\cD^K(\bullet)$, \eqref{eqn.phiN} for $\phi=\bbeta$, Theorem~\ref{thm.up}, and the stability property of $S_\cD^K(\bullet,\bullet)$ in \eqref{eqn.SD} show 
\begin{align}
	A_5 \le C\Big(h^{2(k+1)}+\tau \sum_{n=1}^N(\big\|\bbeta^n\|^2+\|\xi^n\|^2+\|d_t \theta^n\|^2\big)\Big)+\epsilon \tau \sum_{n=1}^N \Big(\|d_t \xi^n\|^2+| \xi^n|_{1,\cT_h}^2\Big)+\epsilon | \xi^N|_{1,\cT_h}^2.
\end{align}
The Cauchy Schwarz inequality, Lemma~\ref{lem.approx},~\ref{lem.cPcerror}.b, and Youngs inequality provide
\begin{align*}
	A_6&=\lambda\tau \sum_{n=1}^N(((I-\Pi_{k+1}^0)c^n,d_t\xi^n)-(\Pi_{k+1}^0(\P_c c^n-c^n),d_t\xi^n))\le Ch^{2(k+2)}+\epsilon \tau \sum_{n=1}^N \|d_t \xi^n\|^2.
\end{align*}
The orthogonality property of $\Pi_{k+1}$, $q\hc \in H^{k+1}(\cT_h)$, and Youngs inequality show
\begin{align*}
	A_7&=\tau \sum_{n=1}^N((\Pi_{k+1}^0-1)q^n\hc^n,d_t\xi^n)
	\le Ch^{2(k+1)}+\epsilon \tau \sum_{n=1}^N \|d_t \xi^n\|^2.
\end{align*}
The property of $S_\cD^K(\bullet,\bullet)$ in \ref{eqn.SD} and Youngs inequality imply
\[A_8 \le C\tau \sum_{n=1}^N\|\xi^n\|^2+\epsilon \tau \sum_{n=1}^N \|d_t \xi^n\|^2.\]

\medskip

\noindent	\textbf{Step 3:} Conclusion\\

\noindent A combination of $A_1$ to $A_8$ in \eqref{eqna1.a8} leads to
\begin{align*}
\tau\sum_{n=1}^N &\|d_t \xi^n\|^2+|\xi^N|_{1,\cT_h}^2 \le 	C\bigg(\tau^2+h^{2(k+1)}+\tau \sum_{n=1}^N\Big(\|\xi^n\|_{}^2+| \xi^n|_{1,\cT_h}^2+\|d_t \theta^n\|^2+\|\beta^n\|^2\Big)\bigg).
\end{align*}
\end{proof}
A combination of \eqref{eqn.phiN}, Theorem~\ref{thm.up}, and Theorem~\ref{thm.c} shows
\begin{align*}
\|\theta^N\|^2+&\|\bbeta^N\|^2+\|\xi^N\|^2+| \xi^N|_{1,\cT_h}^2+\|d_t\theta^N\|^2+\tau \sum_{n=1}^{N}\|d_t\bbeta^n\|^2 +\tau\sum_{n=1}^N \|d_t \xi^n\|^2\\
	&\le C\tau\sum_{n=1}^N\Big(\|\theta^{n}\|^2+\|\bbeta^{n}\|^2+\|\xi^n\|^2+| \xi^n|_{1,\cT_h}^2+\|d_t \theta^n\|^2\Big)+C\Big(\tau^2+h^{2(k+1)}\Big).
\end{align*}
An application of discrete Gronwall's Lemma provides
\begin{align}
		\max_n\|\theta^n\|^2+&\max_n\|\bbeta^n\|^2+\max_n\|\xi^n\|^2+\max_n | \xi^n|_{1,\cT_h}^2+\max_n\|d_t\theta^n\|^2\nonumber\\
		&+\tau \sum_{n=1}^{N}\|d_t\bbeta^n\|^2 +\tau\sum_{n=1}^N \|d_t \xi^n\|^2\le C(\tau^2+h^{2(k+1)}).\label{eqn.xithetabeta}
\end{align}
{\it Proof of Theorem~\ref{thm.main}.}
The results follow from triangle inequalities, Lemma~\ref{lem.projectionup}, Lemma~\ref{lem.cPcerror}, and \eqref{eqn.xithetabeta} respectively.

\section{Numerical Results}\label{sec:numericalresults}
\noindent This section presents a few examples on general polygonal meshes for the lowest order case $k=0$ to illustrate the theoretical estimates in the previous section. These experiments are conducted on both an ideal test case (Example \ref{sec.example1} and a more realistic test case (Example \ref{sec.example2}), providing comprehensive validation. An interesting aspect of VEM is its ability to be implemented solely based on the degrees of freedom and the polynomial component of the approximation space, see \cite{Veiga_hitchhikersVEM} for details on the implementation procedure.

\subsection{Example 1}\label{sec.example1}
\noindent The model problem is constructed in such a way that the exact solution is known. Let the errors be denoted by
\begin{align*}
	&\err(p):=\|p^n-\Pi_0p_h^n\|_{},\,\quad  \err(\bu):=\|\bu^n-\bPi_0\bu_h^n\|_{},\, \mbox{ and }\; \err(c):=\|c^n-\Pi_1c_h^n\|_{},\,
\end{align*} where ($p^n, \bu^n, c^n$) (resp. ($\Pi_0p_h^n, \bPi_0\bu_h^n, \Pi_1 c_h^n$)) is the exact (resp. numerical) solution at the final time $t_n=T$.

\smallskip

\noindent  Let the computational domain be $\Omega=(0,1)^2$ and consider the generalised problem of \eqref{eqn.model} with 
\begin{align}
	&d(c)\frac{\partial p}{\partial t}+\nabla \cdot\bu=q,\quad \bu=-a(c)(\nabla p-\bg(c)),\nonumber\\
	&\phi\frac{\partial c}{\partial t}+b(c)\frac{\partial p}{\partial t}+\bu\cdot \nabla c -\divc(D(\bu)\nabla c)=f.\nonumber
\end{align}
with the boundary conditions \eqref{eqn.bc} and \eqref{eqn.ic} respectively. The exact solution is given by 
\begin{align*}
	c(x,y,t)&=t^2(x^2(x-1)^2+y^2(y-1)^2)\\
	\bu(x,y,t)&=2t^2\begin{pmatrix}
		x(x-1)(2x-1)\\
		y(y-1)(2y-1)
	\end{pmatrix}\\
	p(x,y,t)&=\frac{-1}{2}c^2-2c+\frac{17}{6300}t^4+\frac{2}{15}t^2.
\end{align*}
 Choose $D(\bu)=|\bu|+0.02$, where $d_m=0.02$ and $d_\ell=d_t=1$, $T=0.1$, $a(c)=b(c)=d(c)=c+2$, $\phi=1$, $ \bg(c)=0$, and $c_0=0=p_0$. 

\smallskip

\noindent A series of square, triangular, concave, structured Voronoi, and random Voronoi meshes (see Figure~\ref{fig.Triangle}-\ref{fig.RV}) are employed to test the convergence results for the VEM.

\begin{figure}[h!!]
	\begin{center}
		\begin{minipage}[b]{0.25\linewidth}
			{\includegraphics[width=5cm]{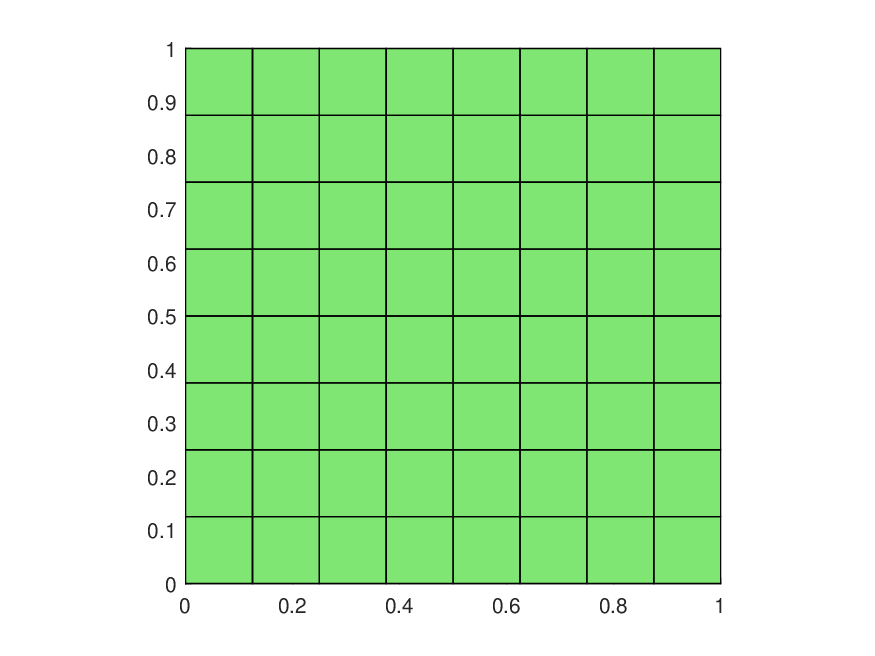}}
			\caption{Square Mesh}\label{fig.Square}
		\end{minipage}
		\begin{minipage}[b]{0.25\linewidth}
			{\includegraphics[width=5cm]{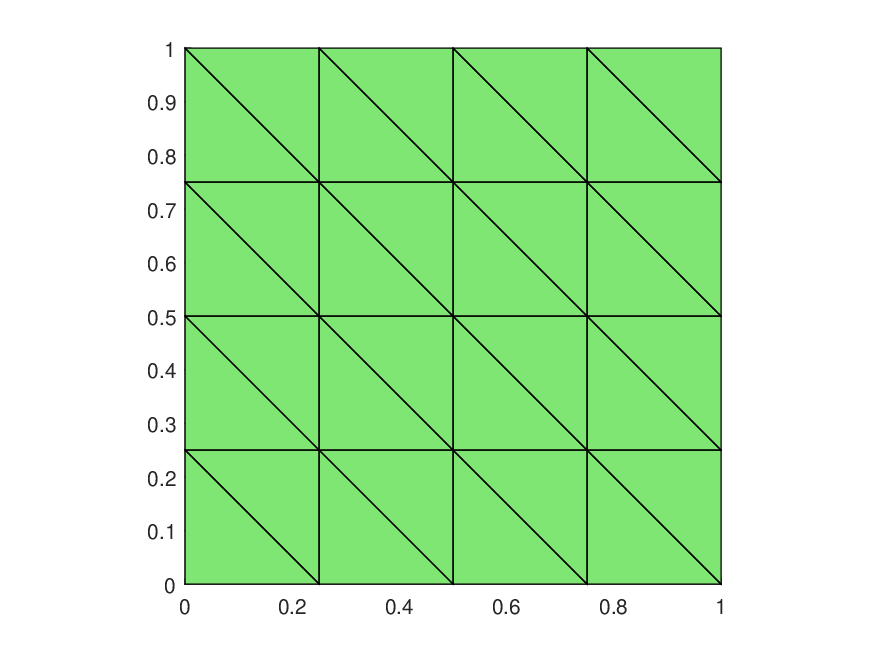}}
			\caption{Triangular Mesh}\label{fig.Triangle}
		\end{minipage}
		\begin{minipage}[b]{0.35\linewidth}
			{\includegraphics[width=5cm]{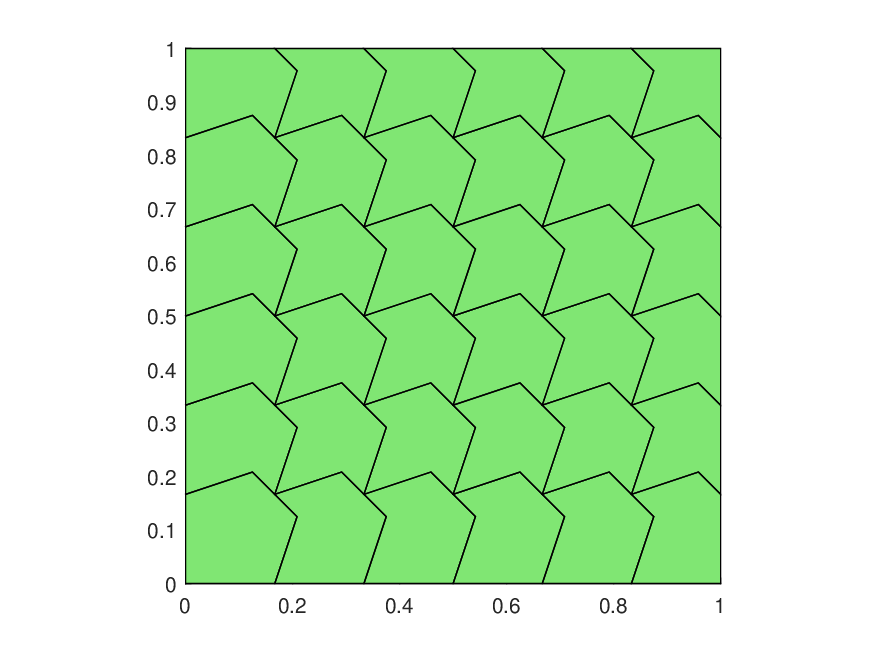}}
			\caption{Concave Mesh}\label{fig.concave}
		\end{minipage}
		\begin{minipage}[b]{0.35\linewidth}
			{\includegraphics[width=5cm]{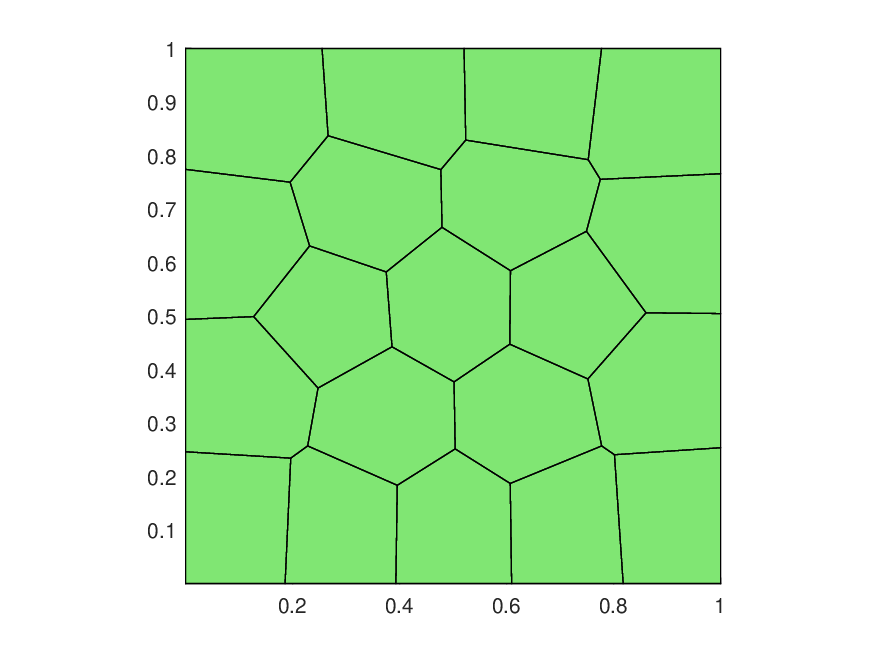}}
			\caption{Structured Voronoi Mesh}\label{fig.SV}
		\end{minipage}
		\begin{minipage}[b]{0.35\linewidth}
			{\includegraphics[width=5cm]{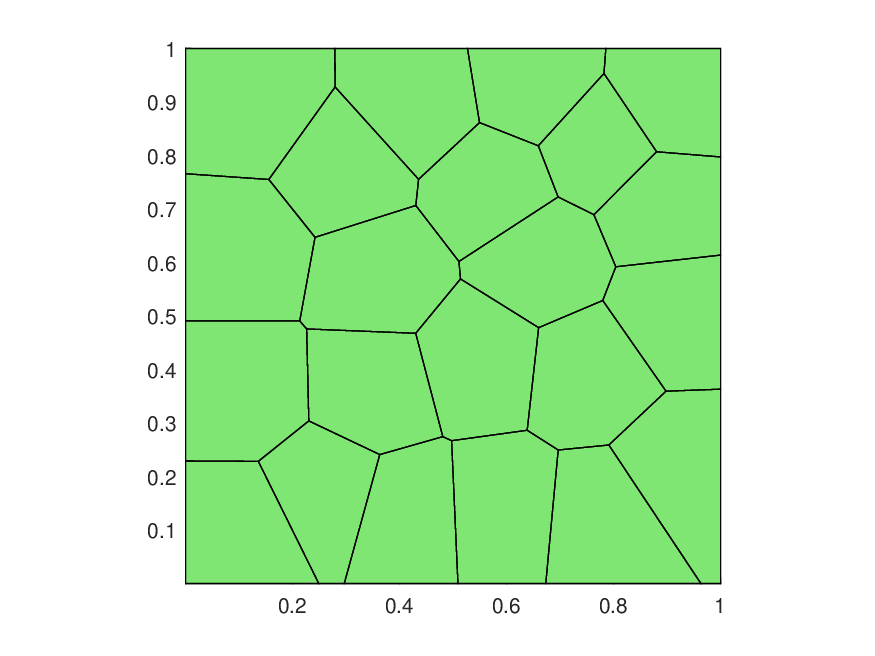}}
			\caption{Random Voronoi Mesh}\label{fig.RV}
		\end{minipage}
	\end{center}
\end{figure}


\begin{table}[h!!]
	\caption{\small{Convergence results, Example 1, Square mesh}}
	{\small{\footnotesize					\begin{center}
				\begin{tabular}{|c|c ||c
						|c||c | c ||c|c|}
				\hline
				$h$&$\tau$&$\err(\bu)$ & Order  & $\err(p)$ & Order  &$\err(c)$ & Order  \\ 
				\hline
	0.353553 & 0.020000 & 0.592054 & - & 0.447413 & -& 0.422921 & - \\
	0.176777 & 0.010000 & 0.261412 & 1.1794 & 0.226153 & 0.9843 & 0.134148 & 1.6566 \\
	0.088388 & 0.005000 & 0.118591 & 1.1403 & 0.113819 & 0.9906 & 0.065810 & 1.0274 \\
	0.044194 & 0.002500 & 0.058576 & 1.0176 & 0.057017 & 0.9973 & 0.032780 & 1.0055 \\
	0.022097 & 0.001250 & 0.029607 & 0.9844 & 0.028522 & 0.9993 & 0.016387 & 1.0002 \\			
				\hline				
			\end{tabular}
	\end{center}}					
}\label{table.eg1.Square}	
\end{table}

\begin{table}[h!!]
	\caption{\small{Convergence results, Example 1, Triangular mesh}}
	{\small{\footnotesize					\begin{center}
				\begin{tabular}{|c|c ||c
						|c||c | c ||c|c|}
				\hline
				$h$&$\tau$&$\err(\bu)$ & Order  & $\err(p)$ & Order  &$\err(c)$ & Order  \\ 
				\hline
0.707107 & 0.010000 & 0.758006 & - & 0.833294 & - & 1.844811 & -\\
0.353553 & 0.005000 & 0.491501 & 0.6250 & 0.369730 & 1.1724 & 0.436673 & 2.0788\\
0.176777 & 0.002500 & 0.263897 & 0.8972 & 0.185239 & 0.9971 & 0.103642 & 2.0749\\
0.088388 & 0.001250 & 0.134321 & 0.9743 & 0.093007 & 0.9940 & 0.028068 & 1.8846\\
0.044194 & 0.000625 & 0.067481 & 0.9931 & 0.046563 & 0.9981 & 0.010140 & 1.4689\\
0.022097&0.000313&0.033787&0.9980&0.023290&	 0.9995& 0.004443& 1.1904\\		
	\hline				
			\end{tabular}
	\end{center}}					
}\label{table.eg1.Triangle}	
\end{table}

\begin{table}[h!!]
\caption{\small{Convergence results, Example 1, Concave mesh}}
{\small{\footnotesize					\begin{center}
		\begin{tabular}{|c|c ||c
				|c||c | c ||c|c|}
		\hline
		$h$&$\tau$&$\err(\bu)$ & Order  & $\err(p)$ & Order  &$\err(c)$ & Order  \\ 
		\hline
0.485913 & 0.020000 & 0.832388 & - & 0.629846 & - & 0.953857 &-\\
0.242956 & 0.010000 & 0.433720 & 0.9405 & 0.311890 & 1.0140 & 0.156015 & 2.6121\\
0.121478 & 0.005000 & 0.184488 & 1.2332 & 0.156376 & 0.9960& 0.069476 & 1.1671\\
0.060739 & 0.002500 & 0.082744 & 1.1568 & 0.078336 & 0.9973 & 0.033228 & 1.0641\\
0.030370 & 0.001250 & 0.040334 & 1.0367 & 0.039195 & 0.9990 & 0.016208 & 1.0357\\
			\hline				
	\end{tabular}
	\end{center}}					
}\label{table.eg1.Concave}	
\end{table}

\begin{table}[h!!]
\caption{\small{Convergence results, Example 1, Structured Voronoi mesh}}
{\small{\footnotesize					\begin{center}
	\begin{tabular}{|c|c ||c
			|c||c | c ||c|c|}
	\hline
	$h$&$\tau$&$\err(\bu)$ & Order  & $\err(p)$ & Order  &$\err(c)$ & Order  \\ 
	\hline
	0.707107 & 0.020000 & 1.000000 & -& 0.999985 & - & 2.680426 & - \\
	0.340697 & 0.010000 & 0.576136 & 0.7552 & 0.395946 & 1.2688 & 0.230344 & 3.3610 \\
	0.171923 & 0.005000 & 0.261674 & 1.1540 & 0.206366 & 0.9527 & 0.069771 & 1.7463 \\
	0.083555 & 0.002500 & 0.114142 & 1.1498 & 0.104156 & 0.9476 & 0.034435 & 0.9787 \\
	0.047445 & 0.001250 & 0.059287 & 1.1575 & 0.056754 & 1.0728 & 0.016668 & 1.2821 \\
	\hline				
\end{tabular}
\end{center}}					
}\label{table.eg1.SV}	
\end{table}

\begin{table}[h!!]
\caption{\small{Convergence results, Example 1, Random Voronoi mesh}}
{\small{\footnotesize					\begin{center}
\begin{tabular}{|c|c ||c
		|c||c | c ||c|c|}
\hline
$h$&$\tau$&$\err(\bu)$ & Order  & $\err(p)$ & Order  &$\err(c)$ & Order  \\ 
\hline
   0.736793 & 0.020000 & 1.000070 & - & 0.998901 & - & 2.675413 & - \\
0.373676 & 0.010000 & 0.559934 & 0.8543 & 0.412095 & 1.3041 & 0.320976 & 3.1233 \\
0.174941 & 0.005000 & 0.230365 & 1.1703 & 0.187065 & 1.0407 & 0.072242 & 1.9650 \\
0.089478 & 0.002500 & 0.097404 & 1.2839 & 0.086810 & 1.1451 & 0.034179 & 1.1163 \\
0.041643 & 0.001250 & 0.040472 & 1.1482 & 0.037883 & 1.0841 & 0.016312 & 0.9671 \\
\hline				
\end{tabular}
\end{center}}					
}\label{table.eg1.RV}	
\end{table}
\noindent Table~\ref{table.eg1.Triangle}-\ref{table.eg1.RV} show  errors and orders of convergence for the velocity $\bu$, pressure $p$, and concentration $c$ for the aforementioned five types of meshes. Observe that linear order of convergences are obtained for these variables in $L^2$ norm. These numerical order of convergence clearly matches the expected order of convergence given in Theorem~\ref{thm.main} with $k=0$.

\subsection{Example 2}\label{sec.example2}
In this example, we examine a more realistic test \cite{compressiblemiscible_Zhang_2020} aiming to provide a qualitative comparison with the anticipated benchmark solution found in existing literature. Here, consider \eqref{eqn.model} with \eqref{eqn.bc} and \eqref{eqn.ic}, where the
spatial domain is $\Omega=(0,1000)\times (0,1000)$ ft$^2$, the time period is $[0,T]=[0,3600]$ days,
and the viscosity of the oil is $\mu(0)=1.$ cp. The injection well is situated in the upper-right corner of the domain at $(1000, 1000)$, with an injection rate of $q = 30 $ ft$^2/$day and an injection concentration of $\hc = 1.0.$ Meanwhile, the production well is positioned at the lower-left corner at $(0, 0)$. The initial concentration across the domain is given by $c_0(x,y) = 0$. The parameters are chosen as $\phi=0.1, d_m=1, d_\ell=0, d_t=0,$ and $a(c)=k(\bx)(1+(M^{1/4}-1)c)^4$, where $k(\bx)$, $M$ and $d_m$ are mentioned below for each of the four tests. Also, the compressibility of the medium $c_\phi= 0.000001$, reference pressure $p_r = 1$ atm, and initial pressure $p_0 = 3000$ psia respectively. The schemes were tested on a triangular mesh with 512 elements and on a $32 \times 32$ square mesh, both using a time step size of $\tau=36$ days. 

\smallskip

\subsubsection*{Test 1} Assume that the permeability is $k(\bx)=1000$, and mobility ratio between the resident and injected fluid is $M = 1$. Figure~\ref{fig.test1} displays the surface and contour plots of the concentration at $t=3$ years (1080 days) and $t=10$ years (3600 days) for triangular and square meshes. Since $M=1$, $a(c)$ is constant and this implies that the fluid has a constant viscosity $\mu(c)=\mu(0)$, Figure~\ref{fig.test1} shows that the contour plots for the concentration is circular until the injecting fluid reaches the production well at $t=3$ years. Because of the effect of the no-flow boundary conditions and the production well, the invading fluid is expected to move toward the production well faster along the
diagonal. This has been observed at $t=10$ years. 
	\begin{figure}[h!!]
	\begin{center}
		\begin{minipage}[b]{0.23\linewidth}
			{\includegraphics[width=4.5cm]{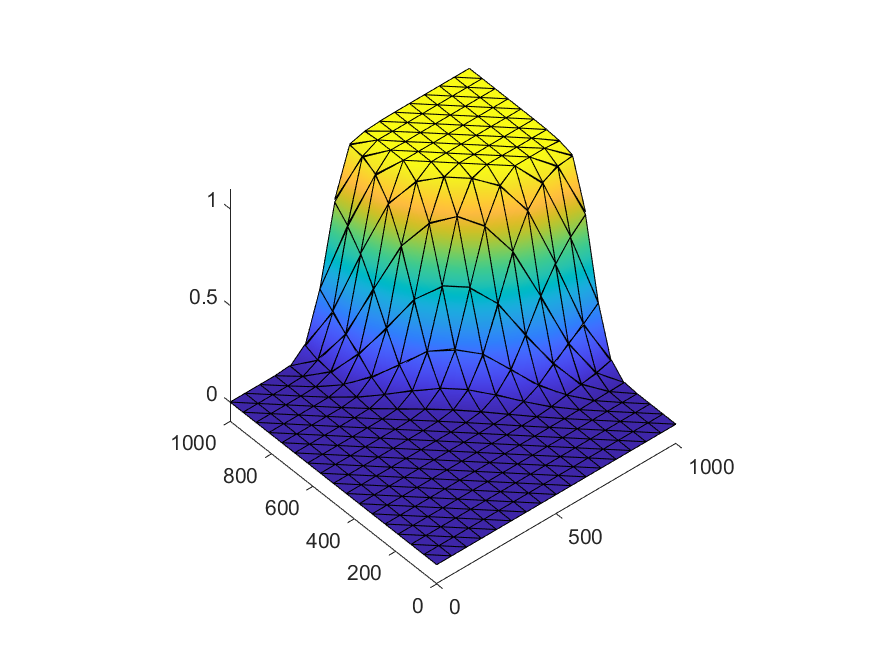}}
		\end{minipage}
		\begin{minipage}[b]{0.22\linewidth}
			{\includegraphics[width=4.1cm]{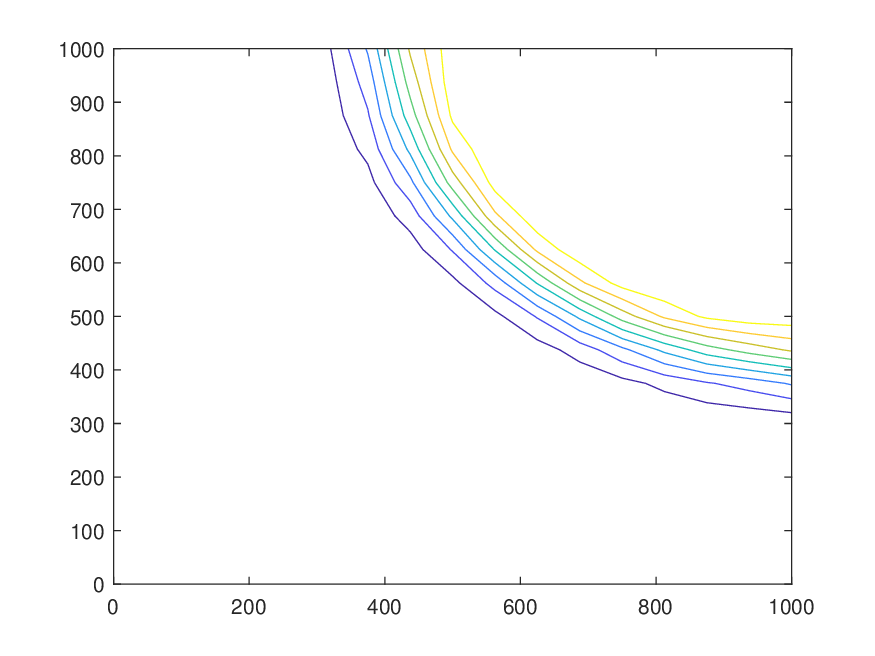}}
		\end{minipage}
		\begin{minipage}[b]{0.23\linewidth}
			{\includegraphics[width=4.5cm]{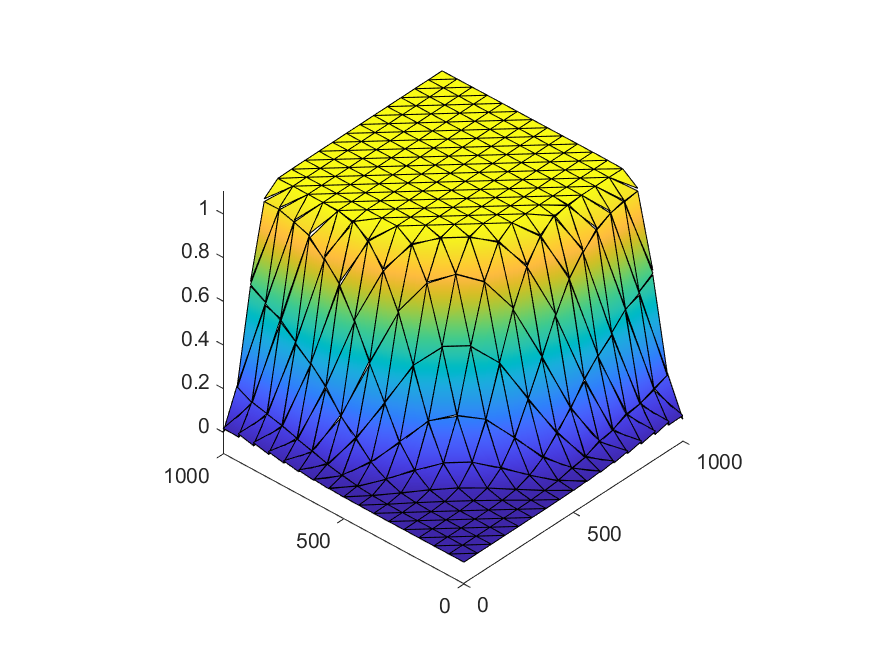}}
		\end{minipage}
		\begin{minipage}[b]{0.23\linewidth}
			{\includegraphics[width=4.2cm]{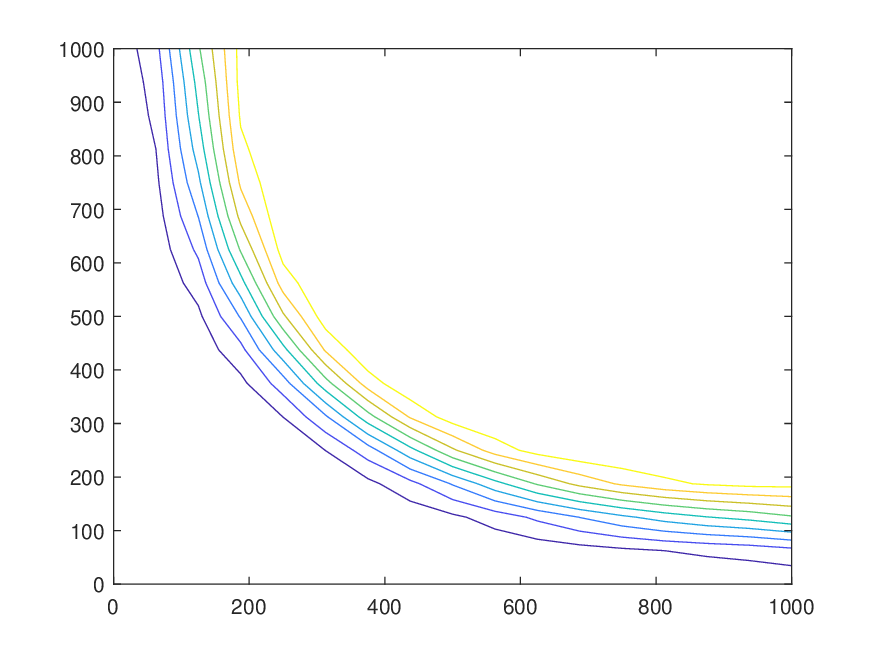}}
		\end{minipage}
		\begin{minipage}[b]{0.23\linewidth}
			{\includegraphics[width=4.5cm]{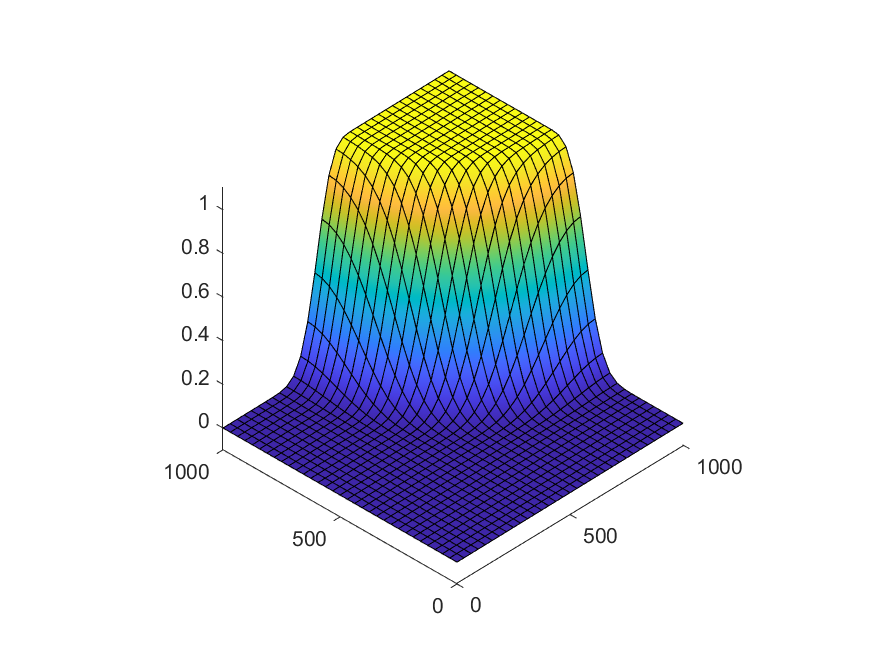}}
		\end{minipage}
		\begin{minipage}[b]{0.22\linewidth}
			{\includegraphics[width=4.1cm]{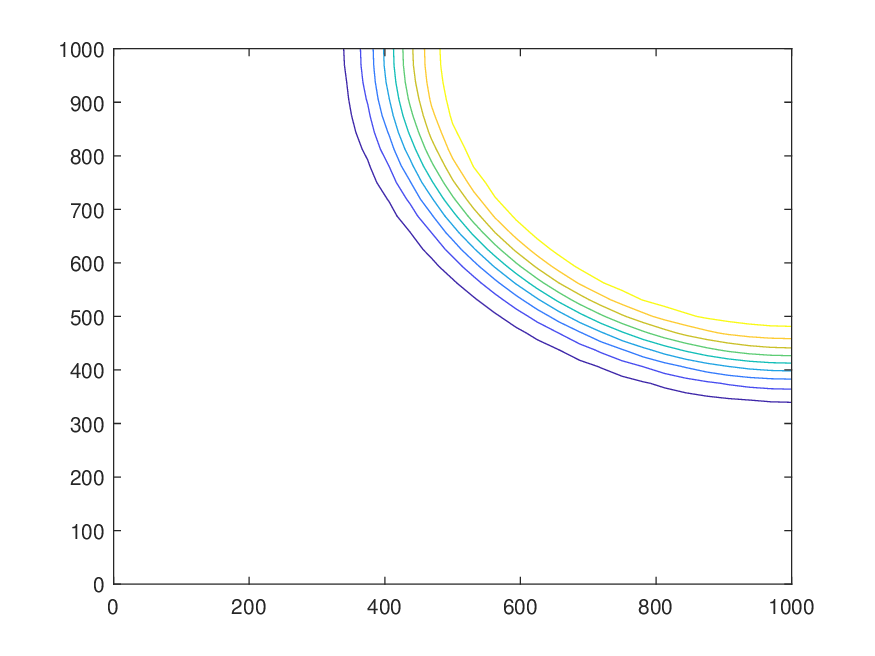}}
		\end{minipage}
		\begin{minipage}[b]{0.23\linewidth}
			{\includegraphics[width=4.5cm]{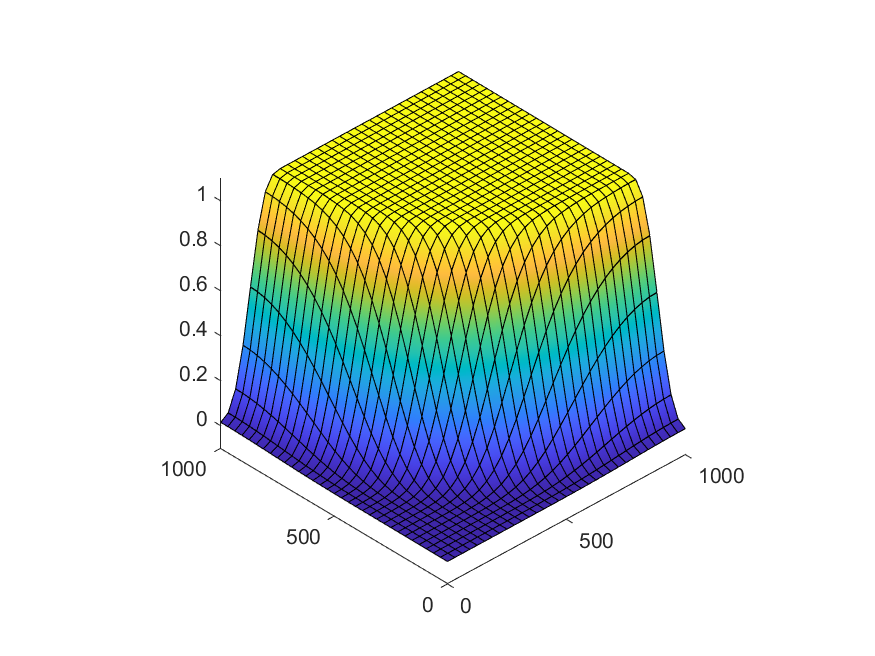}}
		\end{minipage}
		\begin{minipage}[b]{0.23\linewidth}
			{\includegraphics[width=4.2cm]{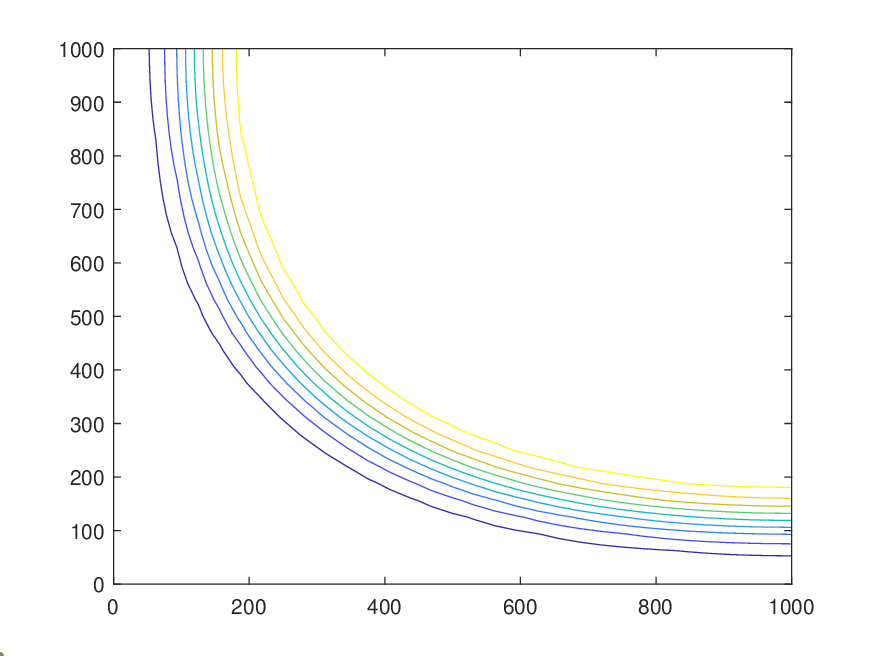}}
		\end{minipage}
		\caption{Surface and contour plots of the concentration in Test 1 at $t=3$ and $t=10$ years for triangular (top) and square (bottom) meshes}\label{fig.test1}
	\end{center}
\end{figure}
\subsubsection*{Test 2} In this test, we take $k(\bx)=1000$ and $M = 41$. The surface and contour plots of the concentration at $t=3$ and $t=10$ years are presented in Figure~\ref{fig.test2}. The viscosity $\mu(c)$ here depends on the concentration $c$ unlike test 1 due to the choice of $M$. As a result, we cannot anticipate the presence of a series of concentric circle, as illustrated in the figure.
	\begin{figure}[h!!]
	\begin{center}
		\begin{minipage}[b]{0.23\linewidth}
			{\includegraphics[width=4.5cm]{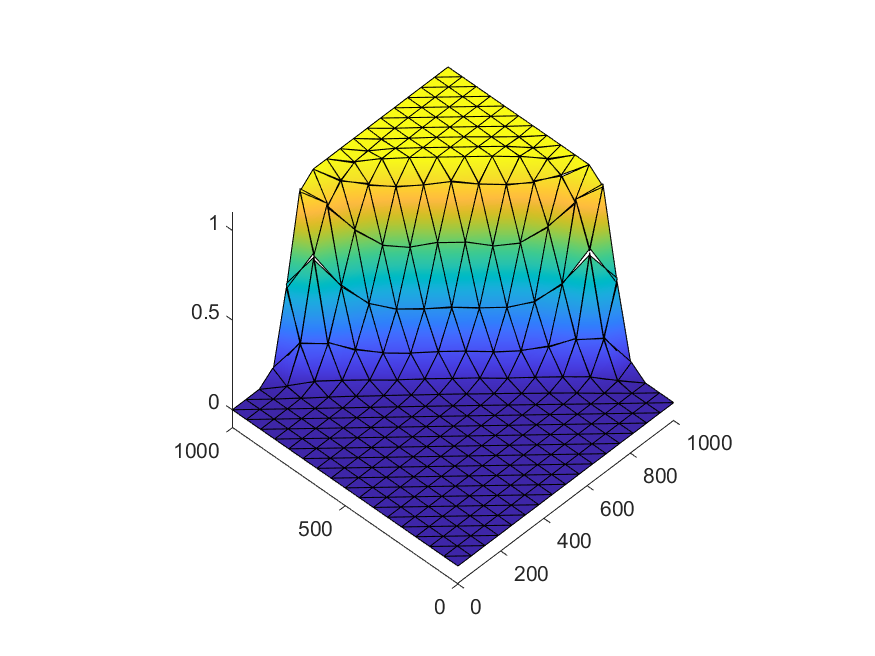}}
		\end{minipage}
		\begin{minipage}[b]{0.22\linewidth}
			{\includegraphics[width=4.1cm]{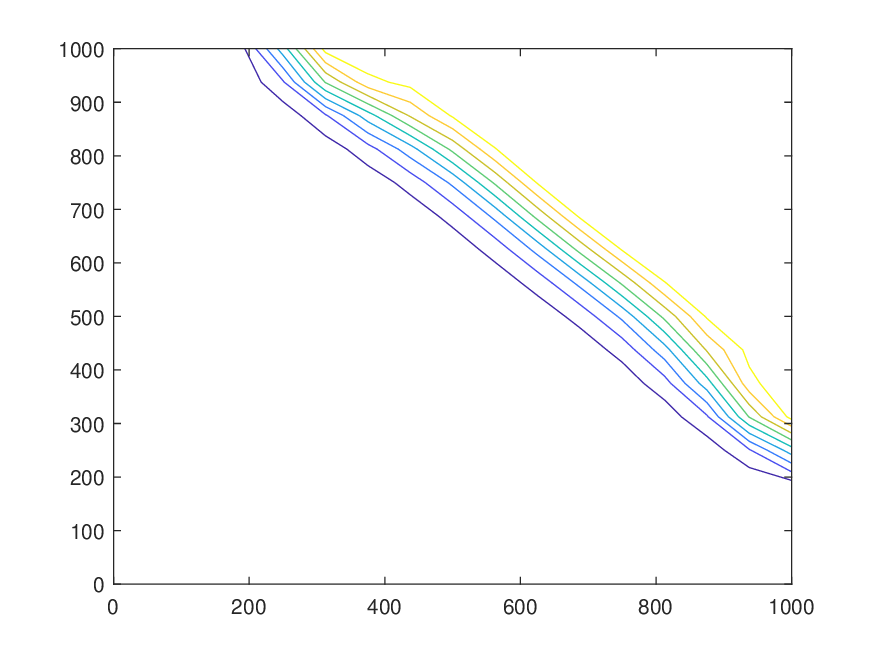}}
		\end{minipage}
		\begin{minipage}[b]{0.23\linewidth}
			{\includegraphics[width=4.5cm]{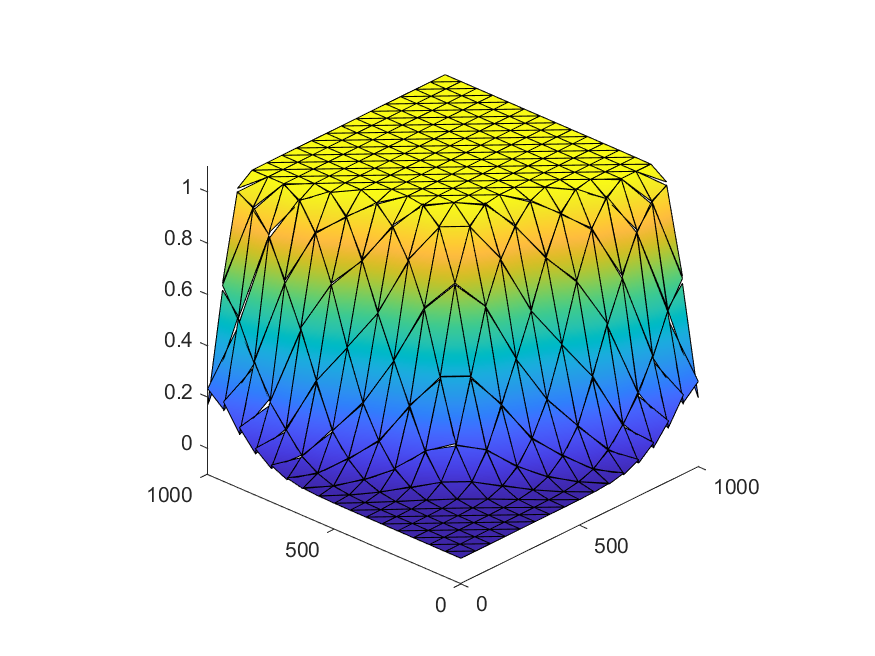}}
		\end{minipage}
		\begin{minipage}[b]{0.23\linewidth}
			{\includegraphics[width=4.2cm]{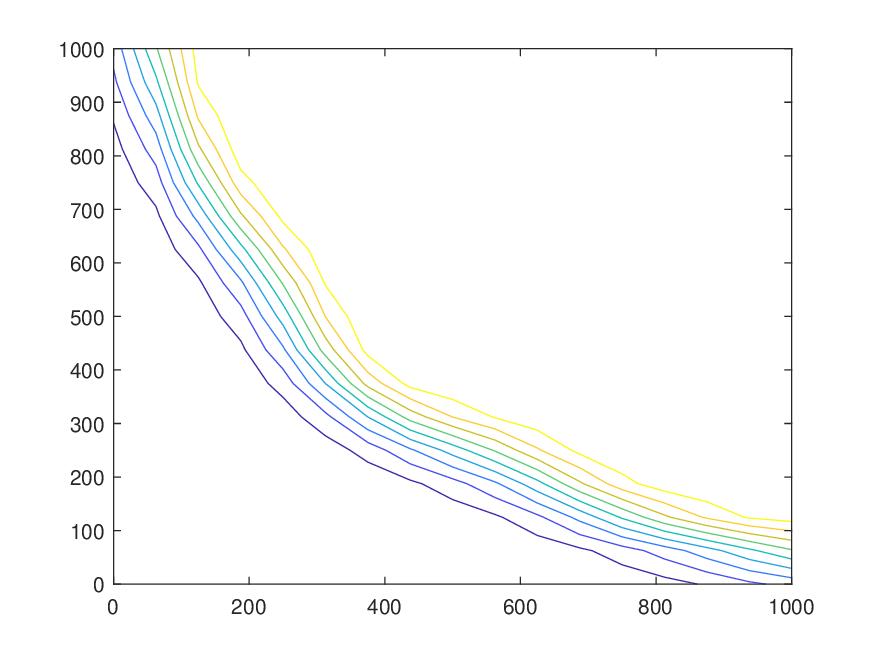}}
		\end{minipage}
		\begin{minipage}[b]{0.23\linewidth}
			{\includegraphics[width=4.5cm]{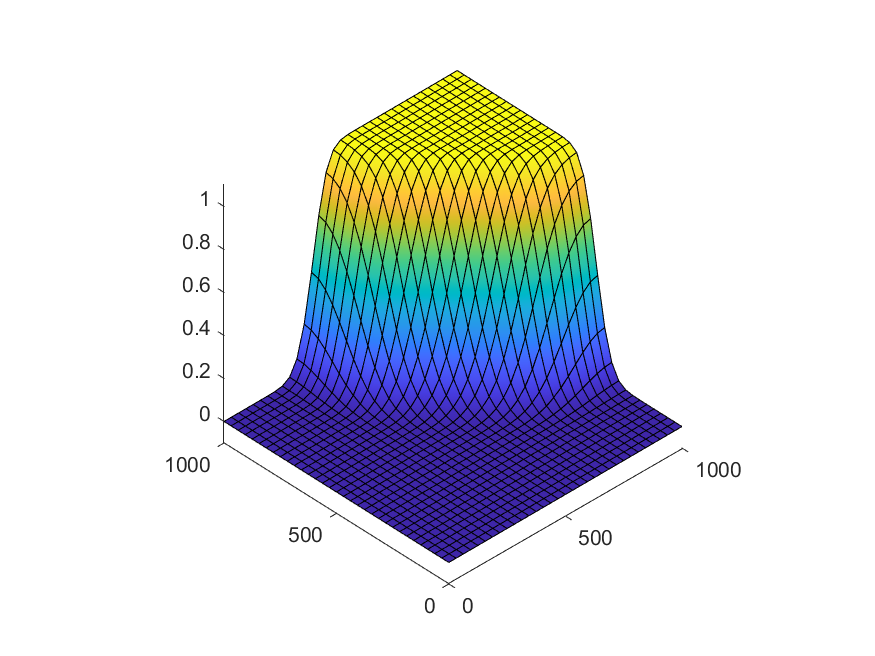}}
		\end{minipage}
		\begin{minipage}[b]{0.22\linewidth}
			{\includegraphics[width=4.1cm]{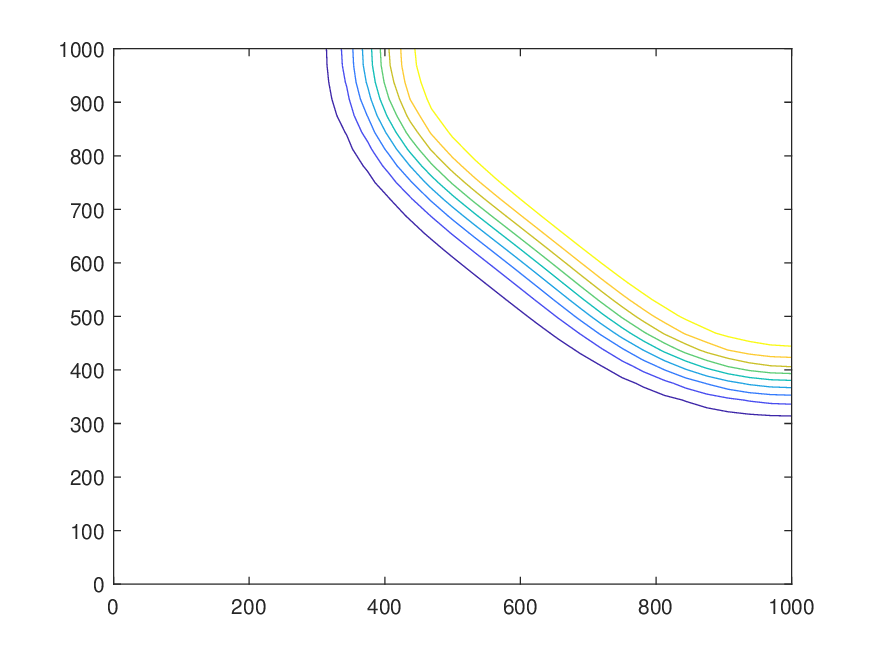}}
		\end{minipage}
		\begin{minipage}[b]{0.23\linewidth}
			{\includegraphics[width=4.5cm]{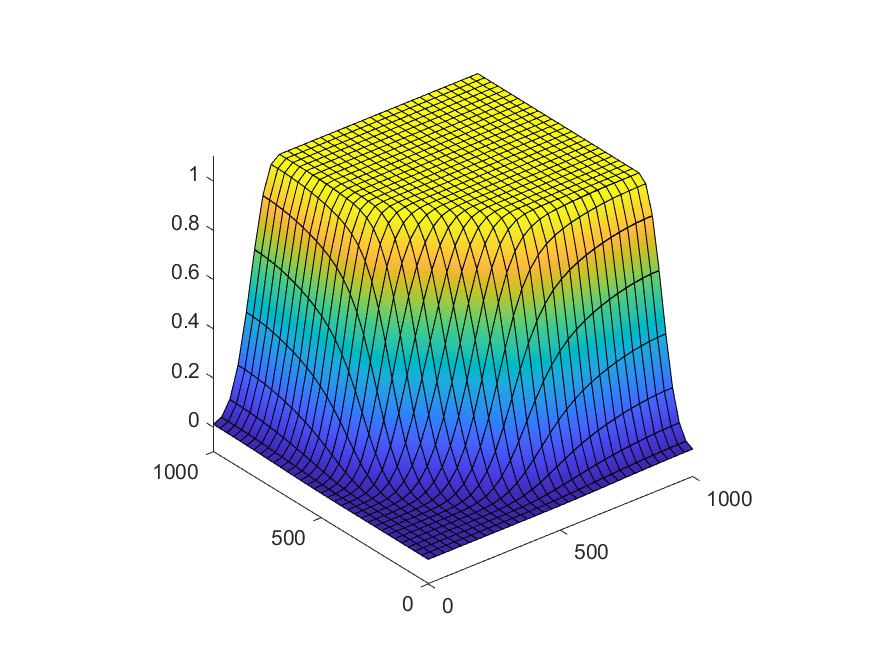}}
		\end{minipage}
		\begin{minipage}[b]{0.23\linewidth}
			{\includegraphics[width=4.2cm]{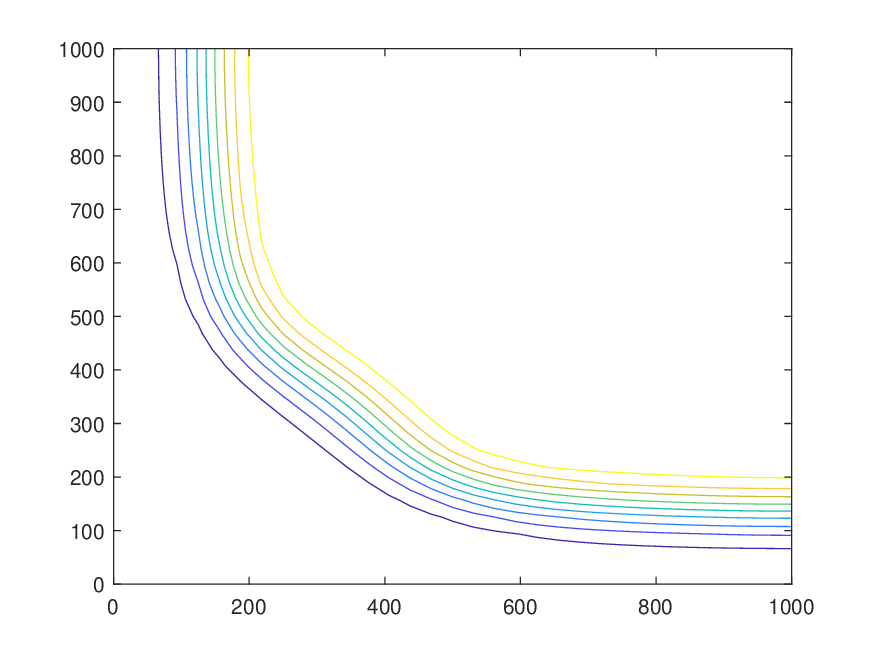}}
		\end{minipage}
		\caption{Surface and contour plots of the concentration in Test 2 at $t=3$ and $t=10$ years for triangular (top) and square (bottom) meshes}\label{fig.test2}
	\end{center}
\end{figure}
\subsubsection*{Test 3} Here, we consider the numerical simulation of a miscible displacement problem with discontinuous permeability, which is commonly encountered in many field applications.. The data is same as given in Test 1 except the permeability
of the medium $k(\bx)$. Let $k(\bx)=1000$ on the sub-domain $\Omega_L:=(0,1000) \times (0,500)$ and$k(\bx)=400$ on the sub-domain $\Omega_U:=(0,1000) \times (500,1000)$. The contour and surface plot at $t = 3$ and $t = 10$ years are given in Figure \ref{fig.test3}. 
	\begin{figure}[h!!]
	\begin{center}
		\begin{minipage}[b]{0.23\linewidth}
			{\includegraphics[width=4.5cm]{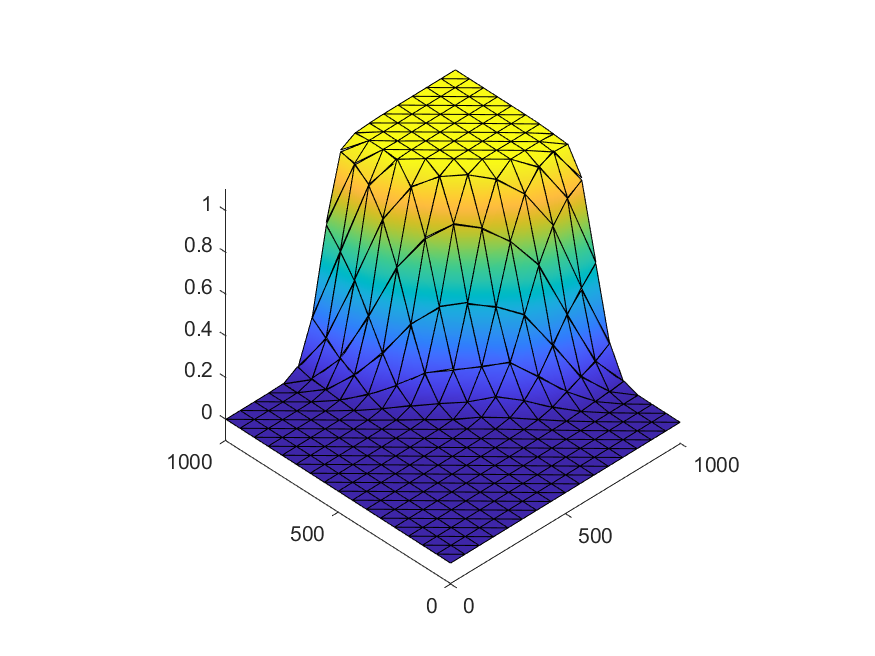}}
		\end{minipage}
		\begin{minipage}[b]{0.22\linewidth}
			{\includegraphics[width=4.1cm]{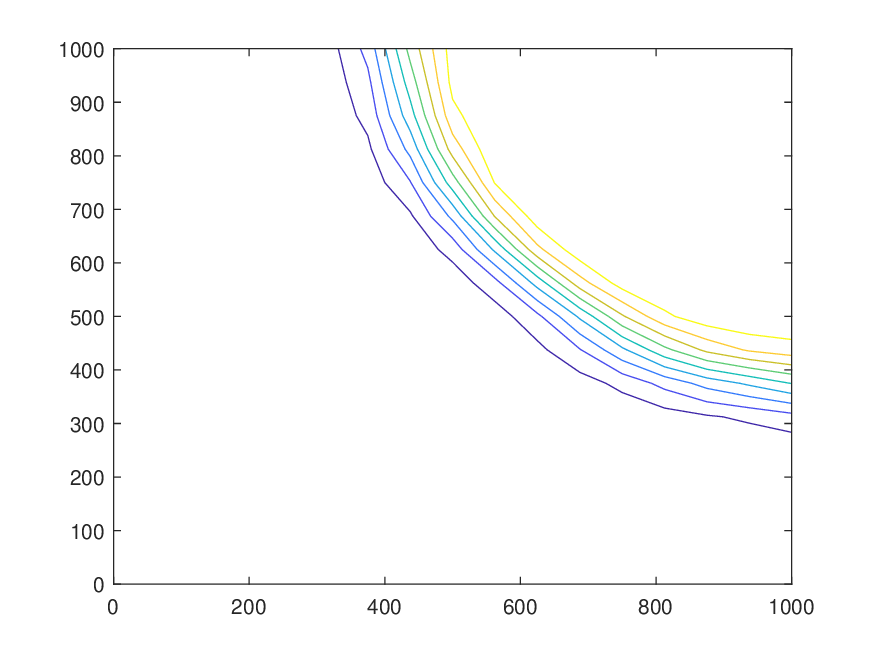}}
		\end{minipage}
		\begin{minipage}[b]{0.23\linewidth}
			{\includegraphics[width=4.5cm]{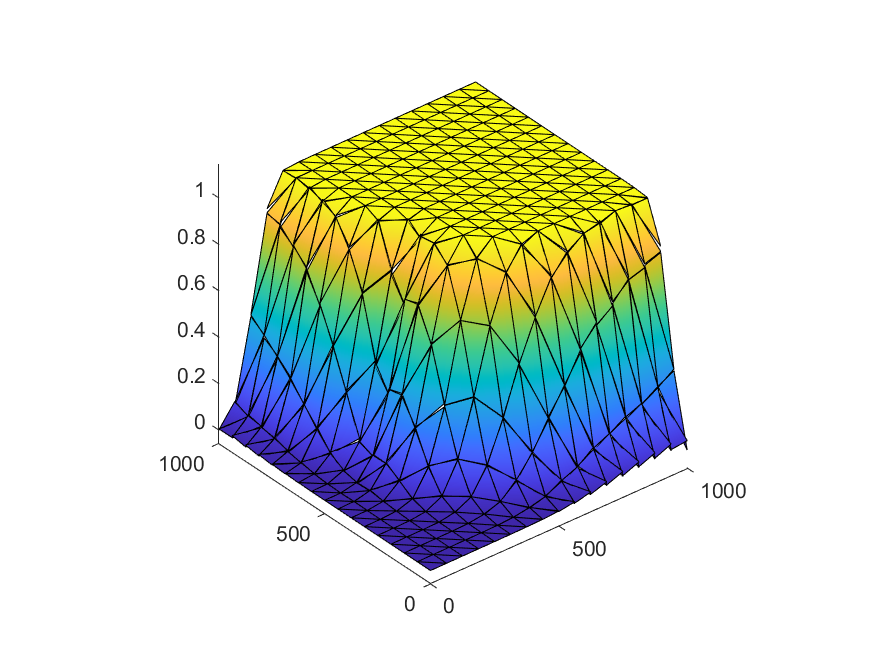}}
		\end{minipage}
		\begin{minipage}[b]{0.23\linewidth}
			{\includegraphics[width=4.2cm]{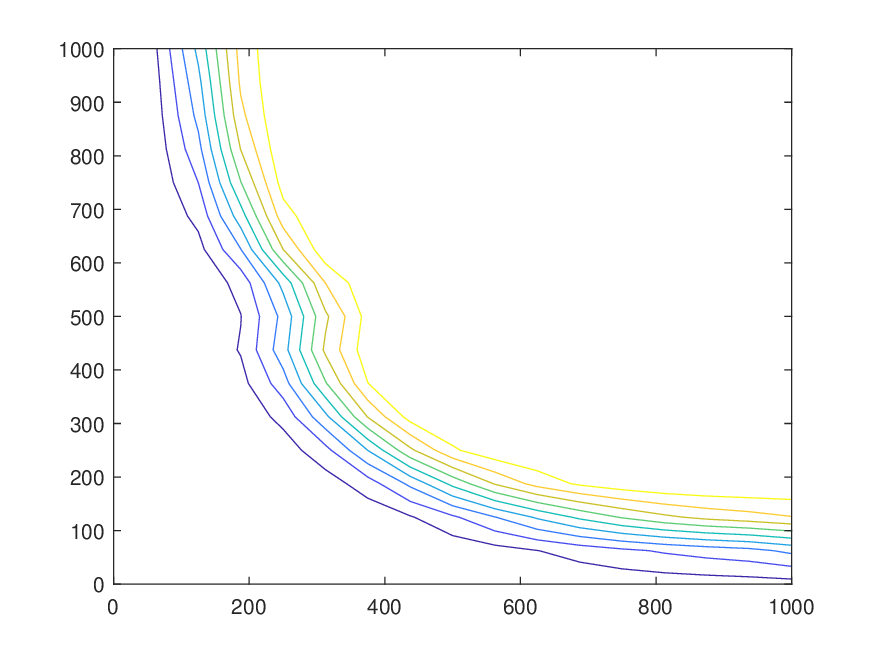}}
		\end{minipage}
		\begin{minipage}[b]{0.23\linewidth}
			{\includegraphics[width=4.5cm]{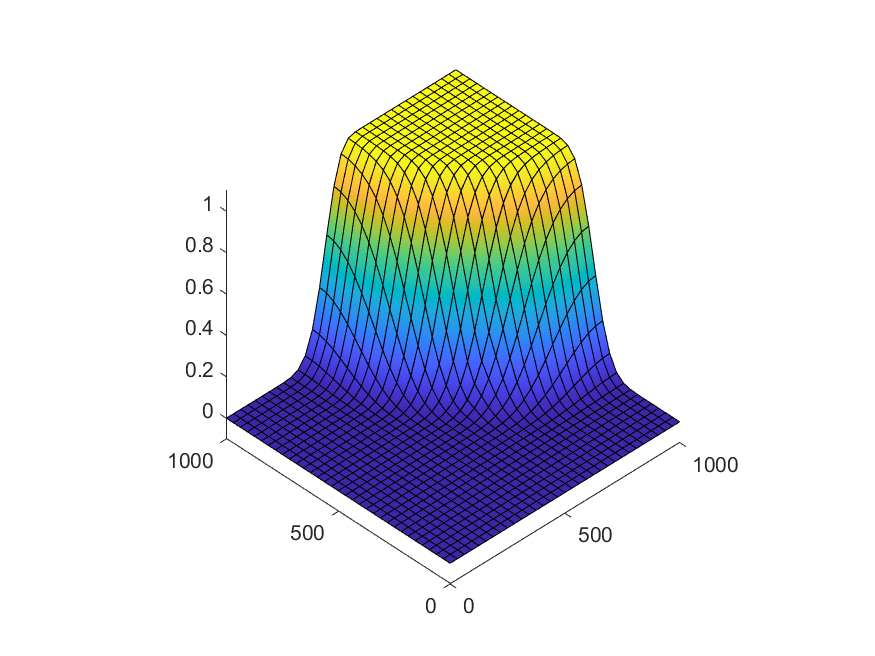}}
		\end{minipage}
		\begin{minipage}[b]{0.22\linewidth}
			{\includegraphics[width=4.1cm]{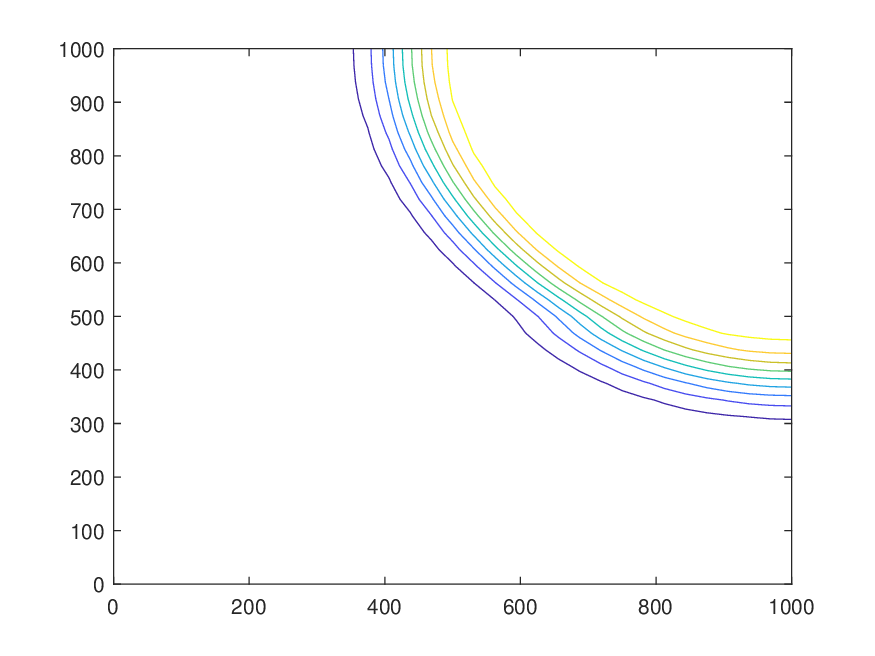}}
		\end{minipage}
		\begin{minipage}[b]{0.23\linewidth}
			{\includegraphics[width=4.5cm]{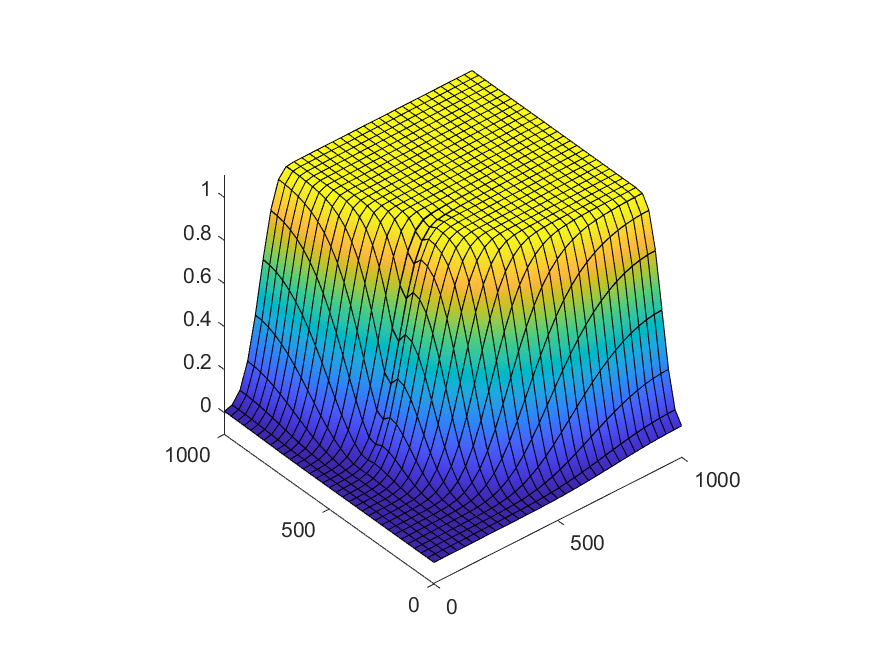}}
		\end{minipage}
		\begin{minipage}[b]{0.23\linewidth}
			{\includegraphics[width=4.2cm]{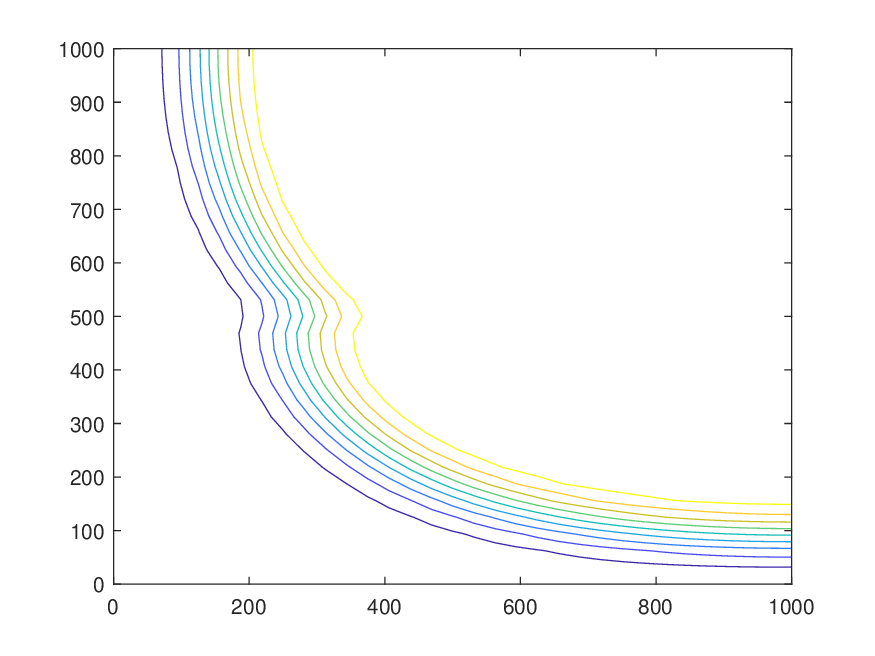}}
		\end{minipage}
		\caption{Surface and contour plots of the concentration in Test 3 at $t=3$ and $t=10$ years for triangular (top) and square (bottom) meshes}\label{fig.test3}
	\end{center}
\end{figure}
\subsubsection*{Test 4} Consider Test 2 with n$k(\bx)=1000$ on the sub-domain $\Omega_L$ and$k(\bx)=400$ on the sub-domain $\Omega_U$. The surface and contour plots of the concentration are depicted in Figure~\ref{fig.test4}. As seen in the figure, the concentration front initially moves faster in the vertical direction than in the horizontal direction. This discrepancy arises from the fact that the subdomain $\Omega_L$ possesses a greater permeability, resulting in a higher Darcy velocity than in the subdomain $\Omega_U$. Once the injecting fluid reaches $\Omega_L$, it starts to move much faster in the horizontal direction on $\Omega_L$ than on $\Omega_U$a due to the same reason.
	\begin{figure}[h!!]
	\begin{center}
		\begin{minipage}[b]{0.23\linewidth}
			{\includegraphics[width=4.5cm]{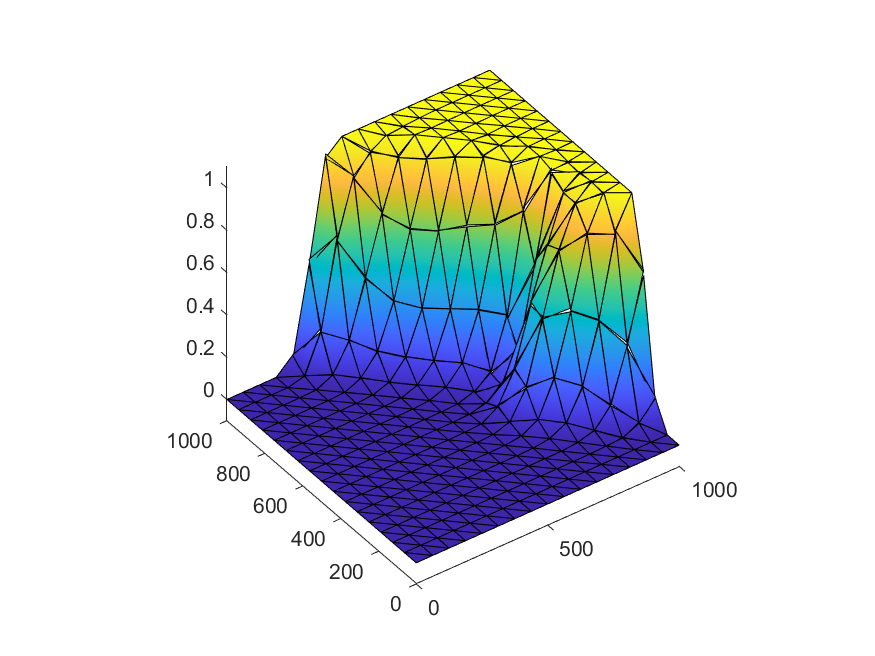}}
		\end{minipage}
		\begin{minipage}[b]{0.22\linewidth}
			{\includegraphics[width=4.1cm]{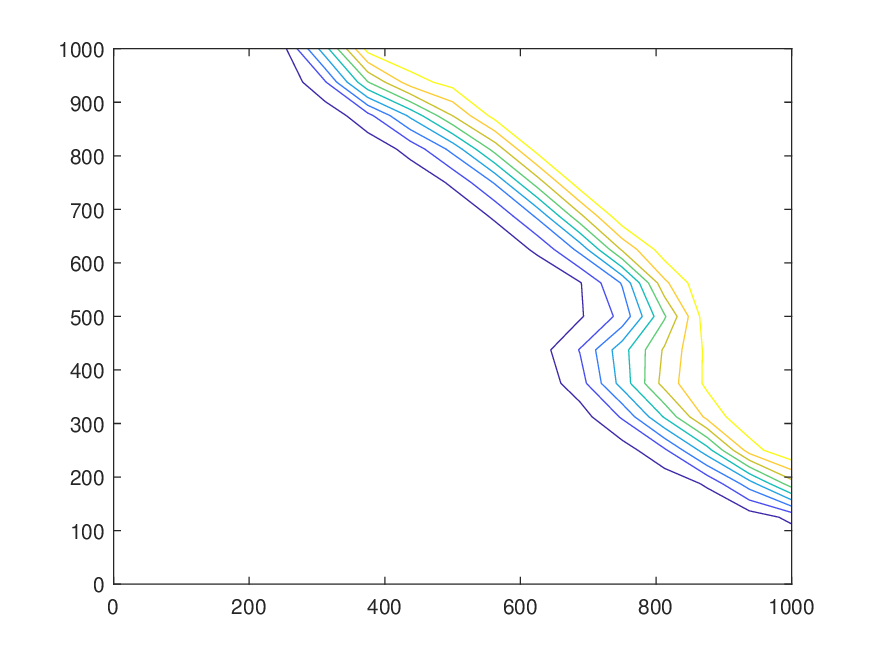}}
		\end{minipage}
		\begin{minipage}[b]{0.23\linewidth}
			{\includegraphics[width=4.5cm]{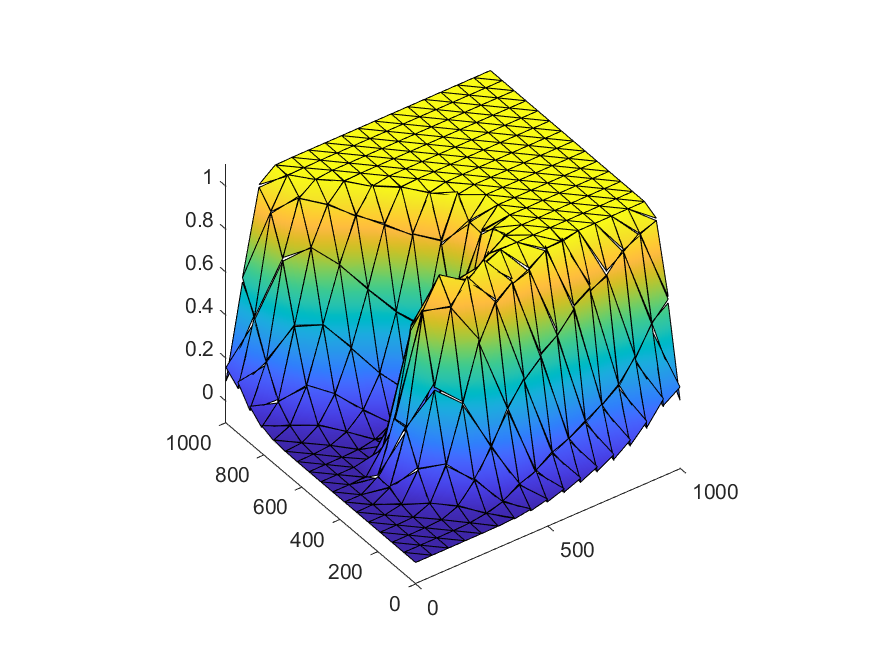}}
		\end{minipage}
		\begin{minipage}[b]{0.23\linewidth}
			{\includegraphics[width=4.2cm]{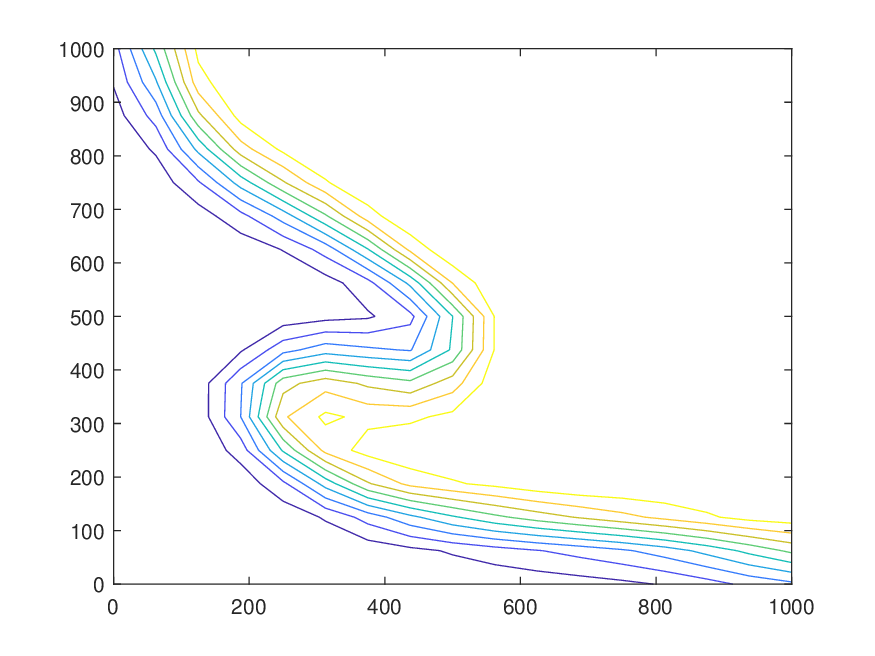}}
		\end{minipage}
		\begin{minipage}[b]{0.23\linewidth}
			{\includegraphics[width=4.5cm]{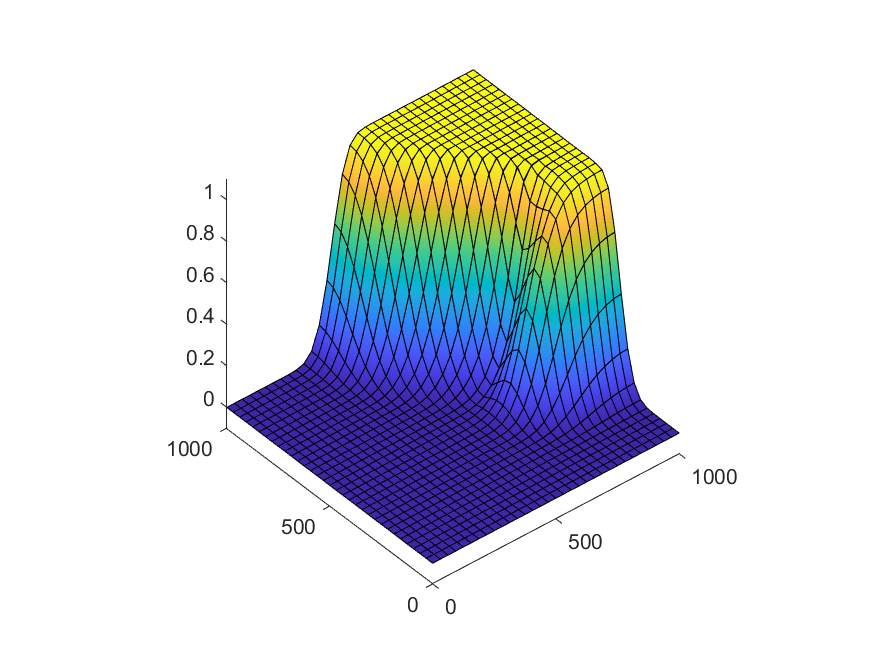}}
		\end{minipage}
		\begin{minipage}[b]{0.22\linewidth}
			{\includegraphics[width=4.1cm]{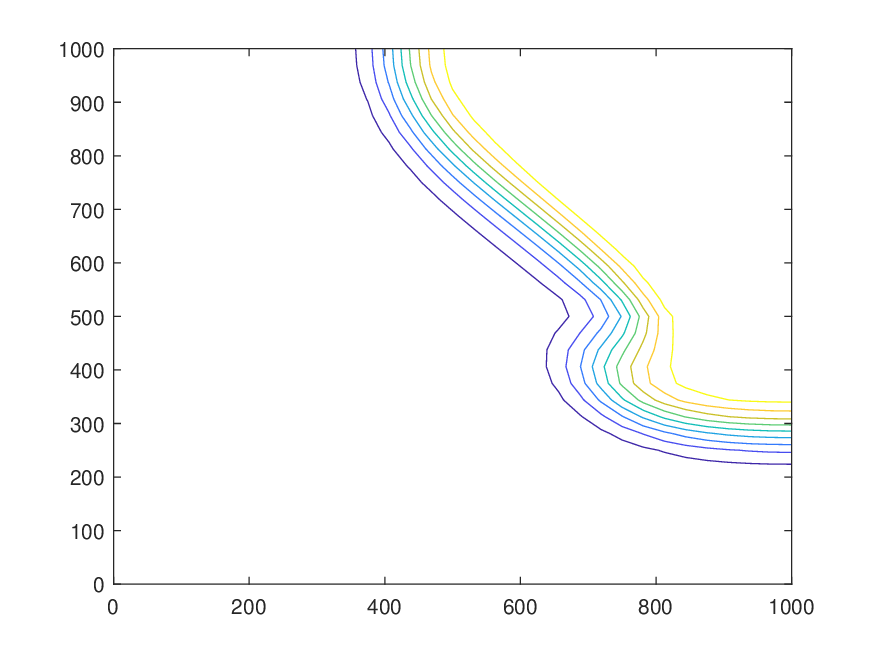}}
		\end{minipage}
		\begin{minipage}[b]{0.23\linewidth}
			{\includegraphics[width=4.5cm]{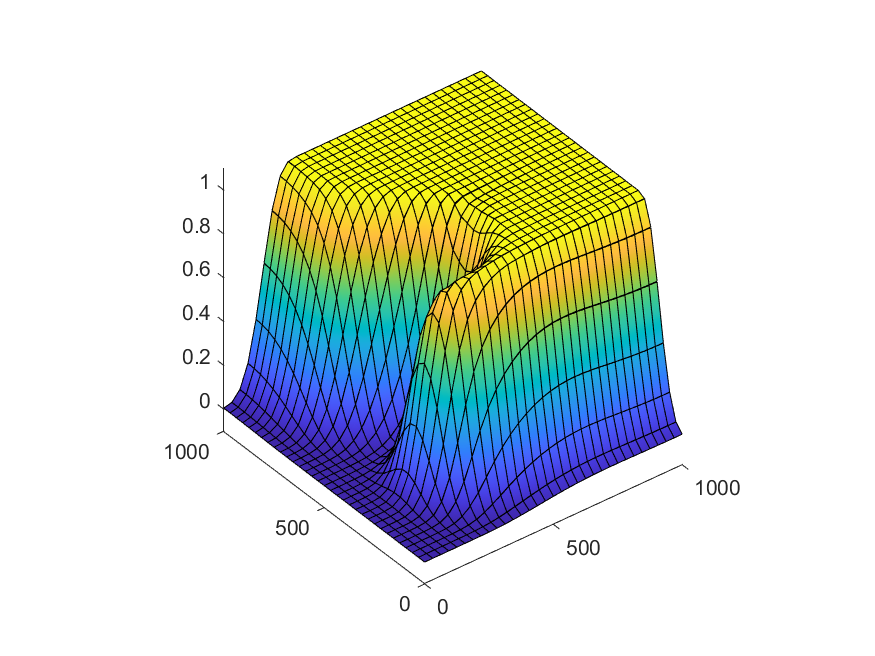}}
		\end{minipage}
		\begin{minipage}[b]{0.23\linewidth}
			{\includegraphics[width=4.2cm]{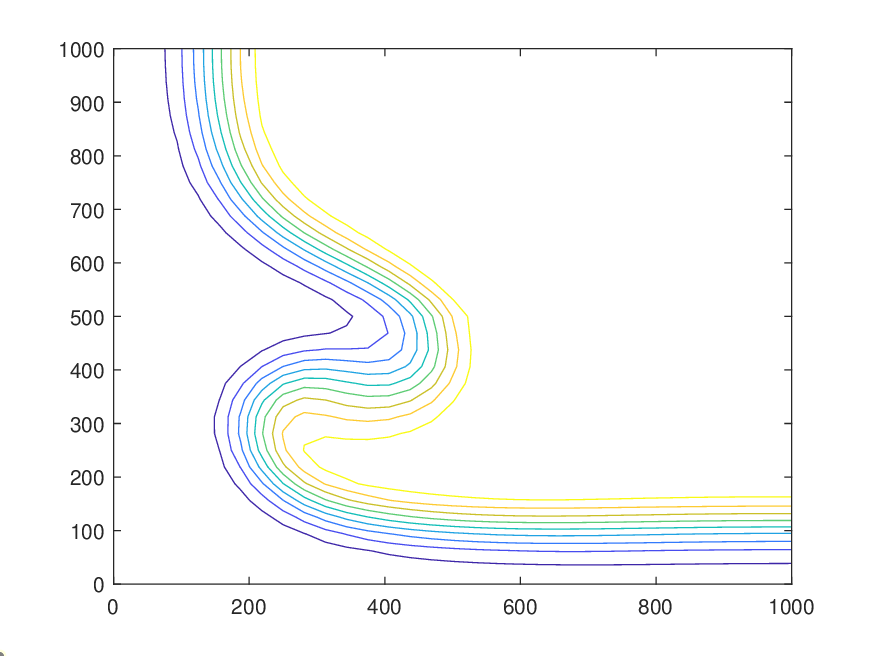}}
		\end{minipage}
		\caption{Surface and contour plots of the concentration in Test 4 at $t=3$ and $t=10$ years for triangular (top) and square (bottom) meshes}\label{fig.test4}
	\end{center}
\end{figure}
\medskip

\noindent {\bf{Acknowledgements.}} The first author thanks the Department of Science and Technology (DST-SERB), India, for supporting this work through the core research grant CRG/2021/002410.  The second author thanks Indian Institute of Space Science and Technology (IIST) for the financial support towards the research work.

\subsection*{Declarations}

\noindent {\bf Conflict of Interest.} The authors declare that they have no conflict of interest.

\bibliographystyle{amsplain}
\bibliography{VEMBib}
\end{document}

%% file: command_vem.tex
\theoremstyle{definition}

\numberwithin{equation}{section}

\newcommand{\bV}{\text{\bf V}}

\newcommand{\bv}{\boldsymbol{v}}

\newcommand{\N}{\mathbb{N}}

\newcommand{\T}{\mathcal{T}}
\renewcommand{\P}{\mathcal{P}}

\newcommand{\Holders}{H\"{o}lder's }